\documentclass[12pt]{article}
\usepackage{appendix}
\usepackage{amsthm,bbm,amssymb,amsmath,bm,mathrsfs}
\usepackage{graphicx}
\usepackage{subfigure}
\usepackage{float}
\usepackage{epsfig}
\usepackage[colorlinks=true, citecolor=blue]{hyperref}
\usepackage{abstract}
\usepackage{natbib}
\usepackage{array}
\usepackage{booktabs}
\usepackage[english]{babel}
\usepackage[utf8]{inputenc}
\usepackage{algorithm,algpseudocode}  
\usepackage{booktabs}
\usepackage{multirow}

\allowdisplaybreaks[4]
\newtheorem{thm}{\textbf{Theorem}}
\newtheorem{cor}{\textbf{Corollary}}
\newtheorem{definition}{\textbf{Definition}}
\newtheorem{lem}{\textbf{Lemma}}

\newtheorem{assumption}{\textbf{Assumption}}
\newtheorem{remark}{\textbf{Remark}}

\DeclareMathOperator*{\argmax}{argmax}
\DeclareMathOperator*{\argmin}{argmin}

\newcommand{\VaR}{\mathrm{VaR}}

\begin{document}

    \title{\vspace{-3.5cm}
        Markov Decision Processes with Value-at-Risk Criterion}

    \author{{Li Xia, \ Jinyan Pan \thanks{L. Xia and J. Pan are both with the School of Business, Sun Yat-Sen University, Guangzhou 510275, China. (email: xiali5@sysu.edu.cn).}}
    }
    \date{}
    \maketitle

\begin{abstract}
Value-at-risk (VaR), also known as quantile, is a crucial risk
measure in finance and other fields. However, optimizing VaR metrics
in Markov decision processes (MDPs) is challenging because VaR is
non-additive and the traditional dynamic programming is
inapplicable. This paper conducts a comprehensive study on VaR
optimization in discrete-time finite MDPs. We consider VaR in two
key scenarios: the VaR of steady-state rewards over an infinite
horizon and the VaR of accumulated rewards over a finite horizon. By
establishing the equivalence between the VaR maximization MDP and a
series of probabilistic minimization MDPs, we transform the VaR
maximization MDP into a constrained bilevel optimization problem.
The inner-level is a policy optimization of minimizing the
probability that MDP rewards fall below a target $\lambda$, while
the outer-level is a single parameter optimization of $\lambda$,
representing the target VaR. For the steady-state scenario, the
probabilistic minimization MDP can be resolved using the expectation
criterion of a standard MDP. In contrast, for the finite-horizon
case, it can be addressed via an augmented-state MDP. Using bilevel
optimization, we prove the optimality of deterministic stationary
policies for steady-state VaR MDPs and deterministic
history-dependent policies for finite-horizon VaR MDPs. We derive
both a policy improvement rule and a necessary-sufficient condition
for optimal policies in these two scenarios. Furthermore, we develop
policy iteration type algorithms to maximize the VaR in MDPs and
prove their convergence. Our results are also extended to the
counterpart of VaR minimization MDPs after appropriate
modifications. Finally, we conduct numerical experiments, including
renewable energy management, to demonstrate the computational
efficiency and practical applicability of our approach. Our study
paves a novel way to explore the optimization of quantile-related
metrics in MDPs through the duality between quantiles and
probabilities.
\end{abstract}

bilevel optimization; policy iteration

\section{Introduction}\label{sec:intro}
Defined as the $\alpha$-quantile of a return distribution,
value-at-risk (VaR) is a fundamental measure of downside risk,
widely used in decision-making with uncertainty. Due to its
intuitive interpretation and regulatory significance, VaR has become
a cornerstone of financial risk management
\citep{duffie1997overview, christoffersen2011elements}. It also
plays a key role in the Basel Accord \citep{berkowitz2002accurate},
shaping capital requirements and risk control frameworks. Moreover,
as an equivalent quantile metric, VaR has found extensive
applications in other domains beyond finance, such as supply chain
\citep{kouvelis2019integrated}, energy \citep{spada2018comparison},
healthcare \citep{olsen2017use}, and queueing systems
\citep{bandi2024robust}, where risk-sensitive and quantile
optimization are also crucial.

The metric of VaR is non-convex and of combinatorial nature. The
computation of VaR cannot be initiated before the entire sample set
has been fully collected, which differs substantially from the
computation of expectation-type metrics. These features make the
optimization of VaR significantly challenging. Actually, VaR
optimization has been proved to be an NP-hard problem
\citep{benati2007mixed}. Research on VaR optimization remains
relatively scarce. Some substantial progress has been made in the
VaR optimization in portfolio management \citep{ahn1999optimal,
basak2001value, ghaoui2003worst, cuoco2008optimal, oum2010optimal}.
However, the above literature requires that the return distribution
is normal or lognormal, making the problem computable
\citep{goh2012portfolio}. For general VaR optimization, various
studies have proposed heuristic algorithms, including a smoothing
method that filters out local non-convexities
\citep{gaivoronski2005value} and another method using conditional
value-at-risk (CVaR) to approximate VaR \citep{larsen2002algorithms,
mausser2014cvar, romanko2016robust}. There also exists an exact
solution approach for the VaR optimization that reformulates the
problem as a mixed-integer programming (MIP) \citep{benati2007mixed,
luedtke2014branch, feng2015practical, babat2018computing,
pavlikov2018optimization}. However, it is difficult to solve with
current MIP solvers for medium to large-scale instances of the
problem. These aforementioned works focus on either static scenarios
\citep{larsen2002algorithms, gaivoronski2005value, benati2007mixed,
mausser2014cvar, romanko2016robust, babat2018computing,
pavlikov2018optimization}, or dynamic portfolio management with
specific distribution assumptions \citep{basak2001value,
cuoco2008optimal}. The dynamic VaR optimization in general cases
remains underexplored, posing both theoretical and computational
challenges \citep{rockafellar2017risk}.

Markov decision process (MDP) is a fundamental model to study
stochastic dynamic optimization problems \citep{puterman1994markov,
Hernandez-Lerma1996}. It is natural to study dynamic VaR
optimization in the framework of MDPs, which belongs to the
community of risk-sensitive MDPs. The main challenge of
risk-sensitive MDPs is the failure of dynamic programming principle,
caused by the non-additive and non-Markovian properties of risk
metrics \citep{bauerle2024markov}. The study of risk-sensitive MDPs
can be traced back to \cite{howard1972}, who use an exponential
utility function to capture risk preferences. Based on the special
structure of the exponential function, they establish a Bellman
optimality equation in product form. \cite{bauerle2014more} further
investigate the general expected utility criterion in MDPs. They
adopt a state-augmented approach by introducing accumulated rewards
as an auxiliary state and subsequently establish the Bellman
optimality equation. The state-augmented approach is a widely used
technique for the risk optimization of accumulated rewards in MDPs,
such as in the probability criterion \citep{white1993,huo2017},
variance criterion \citep{huang2015mean}, and CVaR criterion
\citep{bauerle2011, huang2016,Ugurlu17}. For a more comprehensive
introduction of risk-sensitive MDPs, interested audience can refer
to a recent survey by \cite{bauerle2024markov}.

As a quantile-based risk metric, CVaR MDPs have been widely studied
in the literature to tackle the challenge of the failure of dynamic
programming, exploiting the convexity property of CVaR. One research
stream is based on the primal representation of CVaR
\citep{Rockafellar02}, converting the CVaR MDP into a bilevel
optimization problem \citep{bauerle2011, huang2016,Ugurlu17,
xia2023risk}. The bilevel optimization problem is then solved by
adopting the aforementioned state-augmented approach for CVaR
optimization of accumulated rewards \citep{bauerle2011,
huang2016,Ugurlu17}, and by the sensitivity-based optimization
method for CVaR optimization of steady-state rewards
\citep{xia2023risk}. Another stream utilizes the dual representation
of CVaR to derive a risk-level decomposition of risk measures
\citep{pflug2016time}. The core idea of this method is to decompose
CVaR by different probability levels in the range $[0,1]$  and
introduce these probability levels to augment the state space. Using
this decomposition, the Bellman equations and associated
reinforcement learning algorithms are then derived to optimizing
CVaR MDPs \citep{chow2015risk, chapman2021risk, ding2022cvar,
ding2022sequential}.

While risk-sensitive MDPs have been extensively studied across
various scenarios, existing methods predominantly rely on specific
convex optimization formulations for risk metrics, such as CVaR and
variance. These approaches, however, are not readily applicable to
VaR MDPs since VaR is non-convex. The computation of VaR or
quantile, which depends on the entire sample trajectory of MDPs, is
neither additive nor Markovian. Thus, the optimization of VaR MDPs
remains fundamentally challenging. \cite{filar1995percentile}
investigate the VaR (they call it \emph{percentile}) maximization of
the limiting average reward of infinite-horizon MDPs, which is a
rather restricted setting since the limiting average reward tends to
be deterministic for communicating MDPs. Their focus is the case of
multichain MDPs, where they utilize the chain decomposition theory
to iteratively determine the maximum VaR as well as an optimal
policy. They remark that the VaR maximization in a general setting
such as discounted MDPs is an \emph{open problem}.
\cite{li2022quantile} are the first to treat this challenge of VaR
optimization in a general MDP setting, i.e., the optimization of VaR
(they call it \emph{quantile}) of finite-horizon and
infinite-horizon discounted accumulated rewards. They transform the
VaR of aggregated random variables into a max-min optimization
problem that computes the VaR of each individual variable. Building
upon this transformation, they establish the optimality equation by
leveraging VaR's translation invariance and the state-augmented
approach, and then propose a value iteration type algorithm for
solving the problem. Following the stream of risk-level
decomposition, \cite{hau2023dynamic} rigorously propose the dual
representation of VaR that closely resembles the VaR MDP
decomposition in \cite{li2022quantile}. Based on the analysis in
\cite{li2022quantile} and \cite{hau2023dynamic}, \cite{hau2024q}
further study the model-free method to optimize the VaR MDP, and
propose a new Q-learning algorithm for optimizing the VaR of
discounted accumulated rewards with global convergence and
performance guarantees. \cite{Gilbert2016,jiang2023quantile} both
develop a two-timescale policy gradient algorithm to optimize the
VaR of accumulated rewards, but it lacks dynamic programming
techniques and suffers from local optima. Although the investigation
of VaR MDPs has been initiated by recent literature
\citep{li2022quantile}, this approach is not intuitive as it heavily
depends on a sophisticated decomposition technique. Moreover, it is
also of significance to develop other optimization algorithms beyond
value iteration, such as policy iteration. Thus, a more
comprehensive study of VaR MDPs is ongoing and challenging.

In this paper, we study the optimization of VaR MDPs from a new and
more intuitive perspective, by establishing its equivalence with
probabilistic MDPs. We focus on two types of MDP settings, including
the VaR optimization of accumulated rewards over a finite horizon
and the VaR optimization of steady-state rewards over an infinite
horizon. By proposing a duality relation between VaR and probability
metrics, we convert the VaR maximization MDP to a bilevel
optimization problem with constraints, where the inner level is a
probabilistic minimization MDP with a given target value and the
outer level is a static minimization problem with respect to the
target value. With the nested reformulation, we prove the existence
of an optimal deterministic policy for the VaR MDPs. By comparing
the optimal value of the inner probabilistic minimization MDP with
the given probability level $\alpha$, we further derive a policy
improvement rule and an optimality condition, under which we propose
a policy iteration type algorithm for optimizing VaR in MDPs.

Specifically, we first focus on the maximization of steady-state VaR
MDP among randomized stationary policy space, aiming to find a
policy that maximizes the VaR of steady-state rewards. With the
relation of probability and expectation metrics, we degenerates the
inner probabilistic minimization MDP into a standard average MDP
with a new defined reward function. Therefore, policy improvement
can be achieved by solving a standard average MDP, and subsequently
policy evaluation is conducted by computing the corresponding VaR of
the improved policy based on its steady-state distribution.
Repeating the procedure of policy iteration can generate a series of
strictly improved policies and converge to an optimal policy. We
then consider the maximization of finite-horizon VaR MDP among
randomized history-dependent policy space, aiming to find a policy
that maximizes the VaR of accumulated rewards. Via the
state-augmented approach, we convert the inner probabilistic
minimization MDP to a standard finite-horizon MDP, introducing an
auxiliary state of the remaining goal at current stage.
Consequently, policy improvement is equivalent to solving a standard
finite-horizon MDP with an augmented state space, which also enables
dynamic programming to evaluate the corresponding VaR. The global
convergences of policy iteration algorithms for both steady-state
VaR maximization MDP and finite-horizon VaR maximization MDP are
proved. Furthermore, we extend our approach to minimizing VaR for
both steady-state costs and accumulated costs. The corresponding
policy iteration algorithms are also proposed after appropriate
modifications on the version of VaR maximization. Finally, we
conduct numerical experiments to validate our main results,
including the algorithm's global convergence, the comparison between
our policy iteration algorithm and the value iteration algorithm by
\cite{li2022quantile}, and an application to renewable energy
management in microgrid systems.

The contributions of this paper are threefold. First, we propose a
new methodology to study VaR MDPs through establishing the
equivalent relations with probabilistic MDPs. The optimization of a
VaR MDP can be transformed to solving a series of probabilistic MDPs
under certain conditions. This methodology is intuitive and
promising to provide a general framework for VaR optimization in
various settings of MDPs. Second, we develop a policy iteration
algorithm for solving VaR MDPs, which complements a more complete
theoretical and algorithmic framework. Furthermore, our policy
iteration algorithm is shown to be more efficient than the existing
value iteration algorithm, which is similar to their counterparts
for average criteria in classical MDP theory. Third, our approach
handles both the finite-horizon VaR MDPs and steady-state VaR MDPs,
where the latter one is a new MDP setting that has not been explored
in the literature.

The rest of the paper is organized as follows. In
Section~\ref{sec:model}, we propose the duality relations between
VaR and probability metrics, and also give the definition of VaR
optimality criteria in MDPs. We study the optimization theory and
algorithm for steady-state VaR maximization MDPs and finite-horizon
VaR maximization MDPs in Sections~\ref{sec:svarmdp} and
\ref{sec:tvarmdp}, respectively. Then we discuss the extension to
VaR minimization MDPs in Section~\ref{sec:exten}. Numerical studies
to illustrate our main results are presented in
Section~\ref{sec:numer}. Finally, we conclude the paper in
Section~\ref{sec:conclu}.

\section{Problem Formulation}\label{sec:model}
In this section, we first present some fundamental results to show
the relations between VaR and probability optimization, which are
utilized to study VaR MDPs in later sections. Then, we introduce the
definition of VaR optimization in various MDP settings.

\subsection{Preliminaries}
Let $X$ be a random variable with distribution function
$F_X(\cdot)$. The VaR of $X$ at confidence level $\alpha \in (0,1]$
is defined as
     \begin{equation}\label{eq_var-def}
        \VaR_{\alpha}(X) := \inf \{ \lambda \in \mathbb R: F_{X}(\lambda) \geq \alpha
        \} = \inf \{ \lambda \in \mathbb R: \mathbbm P(X \leq \lambda) \geq \alpha
        \}.
     \end{equation}
     When $\alpha = 0$, we define $\VaR_{\alpha}(X) := \inf\{X\}$. The
     above definition of $\VaR_{\alpha}(X)$ is also called the
     \emph{$\alpha$-quantile} or \emph{inverse distribution function} of
     $X$, which may also be denoted by $F^{-1}_{X}(\alpha)$. When $X$ is
     discrete, $F^{-1}_{X}(\alpha)$ is a \emph{left-continuous} inverse
     distribution function. Obviously, $\VaR_{\alpha}(X)$ or
     $F^{-1}_{X}(\alpha)$ is non-decreasing in $\alpha$ and
     $F_{X}(\lambda)$ is non-decreasing in $\lambda$.

     We show that the VaR or $\alpha$-quantile of random variables has
     the following properties.
     \begin{equation*}
        F_{X}(\VaR_{\alpha}(X)) = F_{X}(F^{-1}_{X}(\alpha)) \geq \alpha,
     \end{equation*}
     where the inequality can be replaced by ``$=$" when $X$ is a
     continuous random variable and it needs to remain when $X$ is
     discrete. Furthermore, we have the following lemma.

     \begin{lem}\label{lem_VaRmax}{\rm(\textbf{Duality Relation for VaR Maximization})}
        For two random variables $X$ and $Y$, let $\lambda =
        \VaR_{\alpha}(X)$, we have
        \begin{eqnarray}
            \mathbbm P(Y \leq \lambda) < \alpha &\Longleftrightarrow & \VaR_{\alpha}(Y) > \VaR_{\alpha}(X), \label{eq_VaRmax-impr}\\
            \mathbbm P(Y \leq \lambda) \geq \alpha &\Longleftrightarrow&
            \VaR_{\alpha}(Y) \leq \VaR_{\alpha}(X).  \label{eq_VaRmax-optim}
        \end{eqnarray}
     \end{lem}

     \begin{proof}
        We first prove (\ref{eq_VaRmax-impr}). Since $F_{Y}(\lambda) :=
        \mathbbm P(Y \leq \lambda) < \alpha$ and $F_{Y}(\lambda)$ is
        non-decreasing in $\lambda$, from the definition of VaR in
        (\ref{eq_var-def}), we directly have $\VaR_{\alpha}(Y) > \lambda =
        \VaR_{\alpha}(X)$. Thus, the ``$\Rightarrow$'' direction in
        (\ref{eq_VaRmax-impr}) is proved. For the ``$\Leftarrow$'' direction in
        (\ref{eq_VaRmax-impr}), if $\VaR_{\alpha}(Y) > \VaR_{\alpha}(X) =
        \lambda$, we must have $F_{Y}(\lambda) < \alpha$, because the
        otherwise assumption of $F_{Y}(\lambda) \geq \alpha$ indicates
        $\VaR_{\alpha}(Y) \leq \lambda$ by definition (\ref{eq_var-def}),
        which contradicts the condition $\VaR_{\alpha}(Y) > \lambda$. Thus,
        (\ref{eq_VaRmax-impr}) is proved.

        We can use a similar argument to prove (\ref{eq_VaRmax-optim}), and
        the two statements (\ref{eq_VaRmax-impr}) and
        (\ref{eq_VaRmax-optim}) are actually equivalent.
     \end{proof}

     Lemma~\ref{lem_VaRmax} holds naturally for continuous random
     variables. However, for discrete random variables,
     Lemma~\ref{lem_VaRmax} seems not straightforward since the
     associated function $F_X(\cdot)$ and $F^{-1}_X(\cdot)$ is
     right-continuous and left-continuous, respectively, as illustrated
     by Figure~\ref{fig_leftcontinuousF}.
     Thus, Lemma~\ref{lem_VaRmax} is
     especially useful for the VaR maximization of discrete random
     variables.

     \begin{figure}[htbp]
        \centering
        \includegraphics[width=0.85\columnwidth]{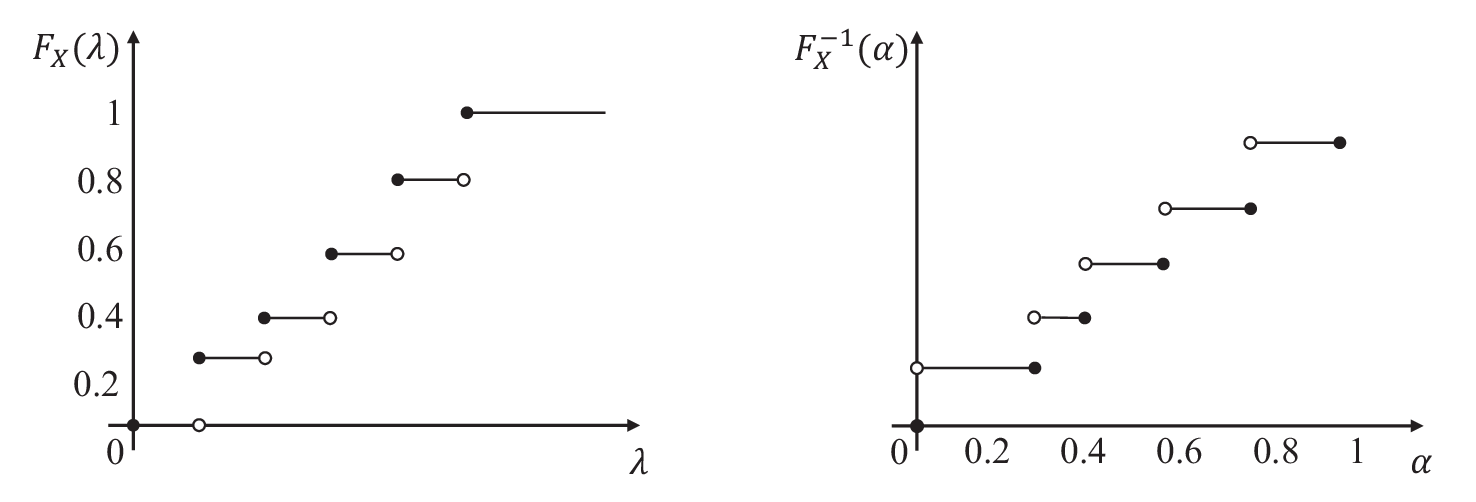}
        \caption{The right-continuous distribution function $F_X(\cdot)$ and
            the left-continuous inverse distribution function $F^{-1}_X(\cdot)$
            of a discrete random variable $X$.} \label{fig_leftcontinuousF}
     \end{figure}

For a discrete random variable $Y$ with support on $\mathcal Y$ and
a real number $\lambda \in \mathbb R$, we define
$\lambda_{-}(\mathcal Y)$ as $\lambda$'s \emph{left predecessor} in
$\mathcal Y$, i.e.,
        \begin{equation}\label{equ:pred}
        \begin{array}{lc}
            \lambda_{-}(\mathcal Y):=     \left\{
            \begin{array}{lr}
                \max\limits_{y \in \mathcal Y}\{y: y < \lambda\}, &~\text{if}~ \lambda >  \min \mathcal Y,\\
                \quad\quad\quad \lambda-1, ~\quad\quad\quad &~\text{if}~ \lambda \leq \min \mathcal Y.
            \end{array}
            \right.
        \end{array}
     \end{equation}
    Then, it holds that $\mathbb P(Y \le \lambda_{-}(\mathcal Y)) = \mathbb P(Y < \lambda)$.
  We derive the following lemma for VaR minimization of discrete random variables.

 \begin{lem}\label{lem_VaRmin}{\rm(\textbf{Duality Relation for VaR Minimization})}
    For two discrete random variables $X$ and $Y$, let $\lambda =
    \VaR_{\alpha}(X)$, $\mathcal Y$ be the support of $Y$, and $\lambda_{-}(\mathcal Y)$ be the left predecessor of $\lambda$ in $\mathcal Y$, we have
    \begin{eqnarray}
        \mathbbm P(Y \le \lambda_{-}(\mathcal Y)) \geq \alpha &\Longleftrightarrow&
        \VaR_{\alpha}(Y) < \VaR_{\alpha}(X), \label{eq_VaRmin-impr}\\
        \mathbbm P(Y \le \lambda_{-}(\mathcal Y)) < \alpha &\Longleftrightarrow&
        \VaR_{\alpha}(Y) \geq \VaR_{\alpha}(X).  \label{eq_VaRmin-optim}
    \end{eqnarray}
\end{lem}

\begin{proof}
    Obviously, the two statements (\ref{eq_VaRmin-impr}) and
    (\ref{eq_VaRmin-optim}) are equivalent. In what follows, we prove
    (\ref{eq_VaRmin-impr}).
    For the ``$\Rightarrow$'' direction in (\ref{eq_VaRmin-impr}), if
    $\mathbbm P(Y \le \lambda_{-}(\mathcal Y)) \geq \alpha$, we directly have
    $\VaR_{\alpha}(Y) \le \lambda_{-}(\mathcal Y) < \lambda = \VaR_{\alpha}(X)$
    by the definition of VaR in (\ref{eq_var-def}) and the definition of left predecessor $\lambda_{-}(\mathcal Y)$ in (\ref{equ:pred}). For the ``$\Leftarrow$''
    direction in (\ref{eq_VaRmin-impr}), if $\VaR_{\alpha}(Y) <
    \VaR_{\alpha}(X) = \lambda$,  since $\VaR_{\alpha}(Y) \in \mathcal Y$,
    we must have $\VaR_{\alpha}(Y) \le \lambda_{-}(\mathcal Y)$ by the definition of left predecessor $\lambda_{-}(\mathcal Y)$ in (\ref{equ:pred}), then the definition of VaR in
    (\ref{eq_var-def}) implies that $\mathbbm P(Y \le \lambda_{-}(\mathcal Y))
    \geq \alpha$. Therefore, (\ref{eq_VaRmin-impr}) holds and thus
    (\ref{eq_VaRmin-optim}) also holds.
\end{proof}

     \begin{remark}
        Lemmas~\ref{lem_VaRmax} and \ref{lem_VaRmin} present the important
        \emph{duality relation} between probability functions and VaR
        functions, which can guide the VaR optimization of random variables.
        (\ref{eq_VaRmax-impr}) and (\ref{eq_VaRmin-impr}) can be used to
        find a strictly improved solution for the VaR maximization and
        minimization problem, respectively. (\ref{eq_VaRmax-optim}) and
        (\ref{eq_VaRmin-optim}) can be used as the optimality guarantee when
        the above strictly improving procedure cannot be executed anymore.
        Thus, the VaR maximization (minimization) problem can be converted
        to the probability minimization (maximization) problem. The
        optimization problems of VaR and probability criteria have close
        relations, which is used to study the VaR optimization of MDPs in
        the later sections.
     \end{remark}


     In what follows, we omit the subscript $\alpha$ in the VaR
     formulation for notational simplicity, as $\alpha$ is fixed unless
     otherwise declared.

     \subsection{Markov Decision Processes}
     A discrete-time MDP is denoted by a tuple $\langle \mathcal S,
     \mathcal A, (\mathcal{A}(s),  s \in \mathcal{S}), P, r \rangle$,
     where $\mathcal S$ and $\mathcal A$ represent the finite spaces of
     states and actions, respectively; $\mathcal{A}(s)$ denotes the
     admissible action set in state $s \in \mathcal{S}$ with $\cup_{s \in
     \mathcal{S}}\mathcal{A}(s)=\mathcal{A}$; we denote by
     $\mathcal{K}:=\left\{(s,a): s \in \mathcal{S}, a \in
     \mathcal{A}(s)\right\}$ the set of admissible state-action pairs for
     simplicity; $P$ denotes the transition probability function with
     $P(\cdot|s,a) \in \mathcal{P}(\mathcal S)$ for each given $(s,a) \in
     \mathcal K$, where $\mathcal{P}(\cdot)$ represents the space of
     probability distributions on a set; and $r : \mathcal K \rightarrow
     \mathbb{R}$ is the reward function, where $r(s,a)$ denotes the reward determined by
     the current state-action pair $(s,a) \in \mathcal{K}$.
     Suppose the system state is $s_t \in \mathcal S$ at the current time
     $t$ and an action $a_t \in \mathcal{A}(s_t)$ is adopted, the system
     will receive an instantaneous reward $r(s_t,a_t)$ and move to a new
     state $s_{t+1} \in \mathcal S$ at the next time $t+1$ according to
     the transition probability $P(s_{t+1}|s_t,a_t)$. It is worth noting
     that we may further generalize the reward function as
     $r(s_t,a_t,s_{t+1})$ or even $r(s_t,a_t,s_{t+1})+\xi$, where $\xi$
     is a zero-mean uncontrollable random variable.

     We denote the policy as $u := (u_t; ~ t \ge 0)$, i.e., the
     collection of action-selection rules at each decision time epoch.
     Specifically, we call $u$ a \emph{randomized history-dependent
     policy}  if $u_t: \mathcal{K}^t \times \mathcal{S} \rightarrow
     \mathcal P(\mathcal A)$, where $\sum_{a \in \mathcal
     A(s_t)}u_t(a|h_t)=1$ for each history $h_t :=
     (s_0,a_0,\ldots,s_{t-1},a_{t-1},s_t) \in
     \mathcal{H}_t:=\mathcal{K}^t \times \mathcal{S}$. Further,  $u$
     degenerates into a \emph{randomized Markov policy} if $u_t$ depends
     on the current state $s_t$ instead of history $h_t$, i.e.,
     $u_t(\cdot|h_t)=u_t(\cdot|s_t)$, $\forall h_t \in \mathcal{H}_t$. In
     addition, if $u_t$ is a deterministic decision rule, i.e., $u_t:
     \mathcal{H}_t \rightarrow \mathcal{A}$ or $u_t: \mathcal{S}
     \rightarrow \mathcal{A}$, we call $u$ a \emph{deterministic
     history-dependent policy} or \emph{deterministic Markov policy},
     respectively. If $u_t$ is independent of the decision time $t$,
     i.e., there exists $u:\mathcal S \rightarrow \mathcal P(\mathcal A)$
     (or  $u:\mathcal S \rightarrow \mathcal A$) such that
     $u(\cdot|s_t)=u_t(\cdot|s_t), \forall t \ge 0$, we call $u$ a
     \emph{randomized stationary policy} (or a \emph{deterministic
     stationary policy}). For notational simplicity, we denote by
     $\mathcal U^{\rm RH}$, $\mathcal U^{\rm RM}$, $\mathcal U^{\rm RS}$,
     $\mathcal U^{\rm DH}$, $\mathcal U^{\rm DM}$ and $\mathcal U^{\rm
     DS}$ the sets of all randomized history-dependent policies,
     randomized Markov policies, randomized stationary policies,
     deterministic history-dependent policies, deterministic Markov
     policies and deterministic stationary policies, respectively.

     For each initial state $s_0 \in \mathcal S$ and policy $u \in
     \mathcal U^{\rm RH}$, by the Theorem of C. Ionescu-Tulcea \citep{Hernandez-Lerma1996}, there
     exists a unique probability space $(\Omega,\mathcal F,
     \mathbbm{P}_{s_0}^{u})$, where $\Omega:=(\mathcal{S} \times
     \mathcal{A})^\infty$ is the sample space, $\mathcal F$ denotes the
     corresponding product $\sigma$-algebra and the probability measure
     is
     \begin{equation*}
     \mathbbm{P}_{s_0}^{u}(\omega)=u_0(a_0|s_0)P(s_1|s_0,a_0)u_1(a_1|s_0,a_0,s_1)\cdots,
     \quad\forall \omega=(s_0,a_0,s_1,a_1,\ldots) \in \Omega.
     \end{equation*}
     We denote by $\mathbbm{E}_{s_0}^{u}$ the expectation operator with
     respect to $\mathbbm{P}_{s_0}^{u}$, and define a reward process
     $(R_t; ~ t \ge 0)$ on $(\Omega,\mathcal F, \mathbbm{P}_{s_0}^{u})$
     such that
     \begin{equation*}
     R_t(\omega)=r(s_t,a_t), \quad\forall \omega=(s_0,a_0,s_1,a_1,\ldots)
     \in \Omega.
     \end{equation*}

     \subsection{VaR Criterion in MDPs}

     This paper aims to study the VaR criteria in discrete-time MDPs with
     infinite-horizon and finite-horizon, respectively, which are
     described as problem 1 (P1) and problem 2 (P2) below:
     \begin{itemize}
     \item[P1.] \textbf{Infinite-horizon model:} we focus on maximizing the VaR of instantaneous rewards when the system tends to be steady over an infinite horizon, say \emph{steady-state VaR MDPs}.
     \item[P2.] \textbf{Finite-horizon model:} we focus on maximizing the VaR of accumulated rewards over a finite horizon, say \emph{finite-horizon VaR MDPs}.
     \end{itemize}
     Next, we give rigorous definitions of the steady-state and
     finite-horizon VaR MDPs, respectively. In order to conveniently
     define the steady-state metrics of MDPs, we restrict our attention
     to randomized stationary policy space $\mathcal U^{\rm RS}$ and make
     the following assumption on ergodicity for steady-state VaR MDPs.
     \begin{assumption}
     The Markov chain under any stationary policy $u \in \mathcal U^{\rm
        RS}$ is assumed to be ergodic (i.e., irreducible and aperiodic).
        \end{assumption}

        The ergodicity assumption indicates that both the Markov chain
        $((s_t,a_t); ~ t \ge 0)$ and the reward process $(R_t; ~ t \ge 0)$
        tend to be steady as $t\rightarrow \infty$. That is, the limiting
        distribution of the discrete-time MDP under policy $u\in \mathcal
        U^{\rm RS}$ is well defined,
        \begin{equation}\label{equ:steady}
     \pi^u(s,a) := \lim\limits_{t \rightarrow \infty}
     \mathbbm{P}^{u}_{s_0}(s_t=s,a_t=a), \qquad s_0 \in \mathcal S, ~
     (s,a) \in \mathcal K.
     \end{equation}
     It is obvious that the limiting (also steady-state) distribution
     $\pi^u(s,a)$ is independent of initial state $s_0$. The steady-state
     reward under a stationary policy $u$ is denoted as $R^u_{\infty}$,
     which is a random variable with support on $\left\{r(s,a): ~(s,a)
     \in \mathcal K \right\}$ and corresponding probability distribution
     $\pi^u$. The definition of the steady-state VaR is given as follows.
     \begin{definition}\label{equ:ss-Varmdp}
     Given a policy $u \in \mathcal U^{\rm RS}$, we define the
     steady-state VaR at probability level $\alpha \in (0,1]$ as
     \begin{equation}\label{equ:svar}
        \VaR^u:=\inf\left\{\lambda \in \mathbb R: \mathbbm P(R^u_{\infty}
        \le \lambda) \ge \alpha\right\}.
     \end{equation}
     \end{definition}

     The steady-state VaR MDP aims to maximize $\VaR^u$ among randomized
     stationary policy space, i.e.,
     \begin{flalign}\label{equ:svarmdp}
     &\mathcal M_1: \hspace{5cm} \VaR^*:= \sup\limits_{u \in \mathcal
        U^{\rm RS}}   \VaR^u, &
        \end{flalign}
        where $\VaR^*$ is called the \emph{optimal value} of the
        steady-state VaR MDP. A policy $u^* \in \mathcal U^{\rm RS}$ is
        called an \emph{optimal policy} of the steady-state VaR MDP if
        $\VaR^{u^*} = \VaR^*$.

        As a counterpart concept, we give the definition of the
        finite-horizon VaR in MDPs as follows.
        \begin{definition}\label{equ:fh-Varmdp}
     For each initial state $s_0 \in \mathcal S$ and policy $u \in
     \mathcal U^{\rm RH}$, we define the finite-horizon VaR at
     probability level $\alpha \in (0,1]$ as
     \begin{equation}\label{equ:tvar}
        \VaR^u(s_0):=\inf\left\{\lambda \in \mathbb R: \mathbbm P_{s_0}^u
        (R_{0:T} \le \lambda) \ge \alpha\right\},
     \end{equation}
     where $R_{0:T}:= \sum\limits_{t=0}^{T} R_t$ denotes the accumulated
     rewards from time $0$ to $T$, and the terminal reward $R_T$ is
     assumed to be $0$ without loss of generality.
     \end{definition}
     The finite-horizon VaR MDP aims to maximize $\VaR^u(s_0)$ among
     randomized history-dependent policy space  for each initial state
     $s_0$, i.e.,
     \begin{flalign}\label{equ:tvarmdp}
     &\mathcal M_2: \hspace{4cm} \VaR^*(s_0) := \sup\limits_{u \in
        \mathcal U^{\rm RH}}   \VaR^u(s_0), \qquad s_0 \in \mathcal{S},&
        \end{flalign}
        where $\VaR^*(\cdot)$ is called the \emph{optimal value function} of
        the finite-horizon VaR MDP. A policy $u^* \in \mathcal U^{\rm RH}$
        is \emph{optimal} for the finite-horizon VaR MDP if
        \begin{equation}\label{equ:optu_def}
     \VaR^{u^*}(s_0) = \VaR^*(s_0), \qquad\forall s_0 \in \mathcal{S}.
     \end{equation}

     Note that the steady-state VaR MDP  (\ref{equ:svarmdp}) and the
     finite-horizon VaR MDP (\ref{equ:tvarmdp}) cannot be solved by
     directly using the method of dynamic programming. It is because the
     VaR metric is non-additive and non-Markovian, which makes the VaR
     optimization problem does not fit the standard model of MDPs. This
     fundamentally challenging problem attracts a lot of research
     attention across various disciplines. However, as discussed in
     Section~\ref{sec:intro}, it has not yet been fully resolved in the
     literature. In this paper, we aim to propose a new optimization
     approach to effectively tackle this challenge.

     \section{Maximization of Steady-State VaR MDPs}\label{sec:svarmdp}
     In this section, we solve the steady-state VaR MDP $\mathcal M_1$
     and propose a policy iteration type algorithm to find an optimal
     stationary policy which can attain the maximum VaR of steady-state rewards.
     It is worth noting that $R^u_{\infty}$ is a discrete random variable
     with support on $\Lambda:=\left\{r(s,a):(s,a) \in \mathcal
     K\right\}$. We can simplify the definition of steady-state VaR in
     (\ref{equ:svar}) as below.
     \begin{equation}\label{equ:new_svar}
     \VaR^u=\min\left\{\lambda \in \Lambda: \mathbbm P(R^u_{\infty} \le \lambda) \ge \alpha\right\}, \qquad  u \in \mathcal{U}^{\rm RS}.
     \end{equation}
     Thus, the steady-state VaR MDP $\mathcal M_1$ in (\ref{equ:svarmdp})
     can be rewritten as the following form of \emph{bilevel
     optimization}, where the inner has a constraint of probability that
     is larger than the confidence level $\alpha$.
     \begin{equation}\nonumber
     \VaR^* =\sup\limits_{u \in \mathcal U^{\rm RS}}
     \min\limits_{\lambda\in \Lambda}\left\{\lambda: \mathbbm
     P(R^u_{\infty} \le \lambda) \ge \alpha\right\}.
     \end{equation}
     Although the above problem has a form like max-min, it is not a
     saddle point problem because of the probability constraint. We have
     difficulties to directly solve it. By taking $Y=R^u_{\infty}$ in
     (\ref{eq_VaRmax-impr}) of Lemma~\ref{lem_VaRmax}, we can obtain the
     relation between steady-state VaR and steady-state probability,
     i.e.,
     \begin{equation}\label{equ:relation_steady}
     \mathbbm P(R^u_{\infty} \le \lambda)<\alpha \Longleftrightarrow
     \VaR^u>\lambda, \qquad u \in \mathcal U^{\rm RS}, ~ \lambda \in
     \mathbb R.
     \end{equation}
     The above relation motivates us to develop a policy improvement
     approach to maximize VaR. Let $\lambda$ be the value of
     steady-state VaR under an initial policy $u^{(0)}$, i.e., $\lambda =
     \VaR^{u^{(0)}}$. We can solve a probabilistic minimization MDP to
     obtain a policy $ u \in \argmin\limits_{ u\in \mathcal U^{\rm RS}}
     \mathbbm P(R^u_{\infty} \le \lambda)$. According to
     (\ref{equ:relation_steady}), if $u$ satisfies $\mathbbm
     P(R^u_{\infty} \le \lambda)  < \alpha$, then $u$ is an improved
     policy such that $\VaR^u > \lambda$. Otherwise, it holds that
     $\min\limits_{u \in \mathcal U^{\rm RS}} \mathbbm P(R^u_{\infty} \le
     \lambda)\ge \alpha$, with (\ref{eq_VaRmax-optim}), which indicates
     that there is no feasible policy that has a greater VaR value than
     $\lambda$. Therefore, the maximization of steady-state VaR in MDPs
     can be achieved by solving a series of probabilistic minimization MDPs.

     To rigorously formulate the above optimization approach, we define
     an auxiliary infinite-horizon probabilistic minimization MDP as
     follows. For any $\lambda \in \mathbb R$, the probability criterion
     under policy $u \in \mathcal U^{\rm RS}$ is defined as
     \begin{equation*}
     F^u(\lambda) := \mathbbm P(R^u_{\infty} \le \lambda).
     \end{equation*}
     The objective is to minimize the probability among randomized stationary policy space, i.e.,
     \begin{flalign}\label{equ:promdp_s}
     &\hat{\mathcal M}_1(\lambda): \hspace{5cm} F^*(\lambda):=
     \inf\limits_{u \in \mathcal U^{\rm RS}}   F^u(\lambda),&
     \end{flalign}
     where $F^*(\lambda)$ is called the optimal value function of the
     infinite-horizon probabilistic minimization MDP. A policy $ u^* \in
     \mathcal U^{\rm RS}$ is optimal for $\hat{\mathcal M}_1(\lambda)$ if
     $F^{ u^*}(\lambda) = F^*(\lambda)$.

     It is worth noting that $\hat{\mathcal M}_1(\lambda)$ is a standard
     model of long-run average MDP with a suitable reward function
     defined as $r_\lambda(s,a):=\mathbb{I}\left\{r(s,a) \le
     \lambda\right\}$, where $\mathbb{I}$ denotes an indicator function.
     With the classical theory of average MDPs, the optimal value of
     $\hat{\mathcal M}_1(\lambda)$ can be found in deterministic
     stationary policy space $\mathcal U^{\rm DS}$, and we can use the
     classical policy iteration or value iteration to solve the
     probabilistic minimization MDP $\hat{\mathcal M}_1(\lambda)$. In
     what follows, we replace inf and $\mathcal U^{\rm RS}$ in
     (\ref{equ:promdp_s}) with min and $\mathcal U^{\rm DS}$,
     respectively. We establish the optimization equivalence between the
     steady-state VaR MDP $\mathcal M_1$ and the probabilistic
     minimization MDP $\hat{\mathcal M}_1(\lambda)$, as stated by
     Theorem~\ref{thm:equiva_s}.

      \begin{thm}\label{thm:equiva_s}
        {\rm (\textbf{Optimization Equivalence between $\mathcal M_1$ and
                $\hat{\mathcal M}_1(\lambda)$})}

        \begin{itemize}
            \item[(a)] Define
            \begin{equation}\label{equ:opt_value_s}
                \lambda^* :=  \min\limits_{\lambda\in \Lambda}\left\{\lambda:
                \min\limits_{u \in \mathcal U^{\rm DS}}F^u(\lambda)\ge
                \alpha\right\},
            \end{equation}
            then the optimal value of the steady-state VaR maximization MDP $\mathcal M_1$ in (\ref{equ:svarmdp}) equals $\lambda^*$, i.e., $\VaR^* =\lambda^*$.
            \item[(b)] Let $\lambda^*_{-}(\Lambda)$ be the left predecessor of $\lambda^*$ in $\Lambda$, and suppose ${u}^* \in \mathcal U^{\rm DS}$ is an optimal policy of
            $\hat{\mathcal M}_1(\lambda^*_{-}(\Lambda))$, then $u^*$ is also optimal for
            $\mathcal M_1$.
        \end{itemize}
    \end{thm}

    \begin{proof}
        To prove (a) and (b), it is sufficient to show that
        \begin{equation}\label{equ:opt_ge_s}
            \lambda^* \ge \VaR^u, \quad\forall u \in \mathcal U^{\rm RS},
        \end{equation}
        \begin{equation}\label{equ:opt_equ_s}
            \lambda^* \le \VaR^{u^*}.
        \end{equation}
        First, we prove (\ref{equ:opt_ge_s}). It follows from  (\ref{equ:promdp_s}) and the definition of $\lambda^*$ in (\ref{equ:opt_value_s}), we have
        \begin{equation*}
            F^u(\lambda^*) \ge F^*(\lambda^*)\ge\alpha, \quad\forall u \in \mathcal
            U^{\rm RS},
        \end{equation*}
        which directly implies (\ref{equ:opt_ge_s}) by the definition of
        $\VaR^{u}$ in (\ref{equ:new_svar}).

        Next, we prove (\ref{equ:opt_equ_s}) in two cases. For the case that
        $\lambda^* = \min\Lambda$, since $\VaR^{u^*}$ takes values in
        $\Lambda$, (\ref{equ:opt_equ_s}) obviously holds. For the other case
        that $\lambda^* \neq \min\Lambda$, from the definition of
        left predecessor $\lambda^*_{-}(\Lambda)$ in (\ref{equ:pred}), we have $\lambda^*_{-}(\Lambda) < \lambda^* $, and we
        can further derive
        \begin{equation}
            \nonumber
            F^{u^*}(\lambda^*_{-}(\Lambda))=\min\limits_{u \in \mathcal U^{\rm DS}}F^{u}(\lambda^*_{-}(\Lambda)) < \alpha,
        \end{equation}
where the equality and inequality follow from the
        definitions of $u^*$ and $\lambda^*$, respectively. By taking $\lambda = \lambda^*_{-}(\Lambda)$
        in (\ref{eq_VaRmax-impr}) of Lemma~\ref{lem_VaRmax}, we have
        $\VaR^{u^*} > \lambda^*_{-}(\Lambda)$. From the definition of left predecessor $\lambda^*_{-}(\Lambda)$, we know that $\lambda^*$ is the minimum
        element of $\Lambda$ that is larger than $\lambda^*_{-}(\Lambda)$. Since
        $\VaR^{u^*} > \lambda^*_{-}(\Lambda)$ and $\VaR^{u^*} \in \Lambda$, we further
        derive that $\VaR^{u^*} \ge \lambda^*$. Thus, (\ref{equ:opt_equ_s}) is
        proved in both cases.
    \end{proof}

     \begin{remark}
    (i). It is worth noting that the optimal policy $u^*$ for the
     steady-state VaR maximization MDP $\mathcal M_1$ is derived by
     solving the probabilistic minimization MDP $\hat{\mathcal
        M}_1(\lambda^*_{-}(\Lambda))$ with $\lambda^*_{-}(\Lambda)$ as its target level, rather than
     the optimal value $\VaR^*$. In most cases, the optimal policy of
     $\hat{\mathcal M}_1(\VaR^*)$ is not optimal for $\mathcal M_1$.

     (ii). Theorem~\ref{thm:equiva_s} indicates that the maximum $\VaR^*$
     can be attained by a deterministic stationary policy. Thus, we can
     rewrite $\mathcal M_1$ in (\ref{equ:svarmdp}) and $\hat{\mathcal
        M}_1(\lambda)$ in (\ref{equ:promdp_s}) as follows, respectively.
     \begin{flalign}\label{equ:svarmdp_2}
        &\mathcal M_1: \hspace{5cm} \VaR^*:= \max\limits_{u \in \mathcal
            U^{\rm DS}}   \VaR^u, &
     \end{flalign}
     \begin{flalign}\label{equ:promdp_s2}
        &\hat{\mathcal M}_1(\lambda): \hspace{4.5cm} F^*(\lambda):=
        \min\limits_{u \in \mathcal U^{\rm DS}}   F^u(\lambda),&
     \end{flalign}
     whose search space can be limited in a finite space $\mathcal
     U^{\rm DS}$.
     \end{remark}

     Theorem~\ref{thm:equiva_s} implies that we can convert the
     steady-state VaR maximization MDP $\mathcal M_1$ to a \emph{bilevel
     optimization problem} with constraints as follows.
     \begin{flalign}\label{equ:bileve_s}
     &\mathcal M_1: \hspace{1.5cm} \VaR^* = \max\limits_{u \in \mathcal
        U^{\rm DS}} \min\limits_{\lambda\in \Lambda}\big\{\lambda:
     F^u(\lambda) \ge \alpha \big\} = \min\limits_{\lambda\in
        \Lambda}\left\{\lambda: \min\limits_{u \in \mathcal U^{\rm
            DS}}F^u(\lambda)\ge \alpha\right\},&
            \end{flalign}
            where the inner level is a policy optimization problem of
            probabilistic minimization MDPs whose optimal value should satisfy
            the constraint of confidence level $\alpha$, and the outer level is
            a single parameter optimization problem with variable $\lambda \in
            \Lambda$. Based on Theorem~\ref{thm:equiva_s}, we can directly
            derive a brute-force approach of solving the bilevel optimization
            problem (\ref{equ:bileve_s}): We can enumeratively solve the
            probabilistic minimization MDP $\hat{\mathcal M}_1(\lambda)$ in
            (\ref{equ:promdp_s2}) for every $\lambda \in \Lambda$ such that we
            can obtain $\VaR^*$ and $u^*$, as stated by
            Theorem~\ref{thm:equiva_s}. However, this naive approach is
            computationally intractable because it requires solving $|\Lambda|$ individual probabilistic
            minimization MDPs $\hat{\mathcal M}_1(\lambda)$.

            To improve the computational efficiency, we further study the
            optimality structure of this VaR maximization problem $\mathcal M_1$
            in (\ref{equ:svarmdp_2}) or (\ref{equ:bileve_s}), and develop a
            policy iteration algorithm to efficiently find optimal policies.
            With (\ref{equ:relation_steady}) and the definition of
            $\hat{\mathcal M}_1(\lambda)$ in (\ref{equ:promdp_s2}), we propose
            the following policy improvement rule and the optimality condition
            of the steady-state VaR MDP $\mathcal M_1$.

            \begin{thm} \label{thm:imp_aver}
     {\rm (\textbf{Policy Improvement Rule for Steady-State VaR Maximization MDPs})}

     Given a policy $u \in \mathcal U^{\rm {DS}}$ and let
     $\lambda=\VaR^u$, if there exists another policy $u' \in \mathcal
     U^{\rm {DS}}$ such that $F^{u'}(\lambda)<\alpha$, then we have
     $\VaR^{u'}>\VaR^u$, i.e., $u'$ is a strictly improved policy.
     \end{thm}

     \begin{proof}
     The result follows directly  from (\ref{equ:relation_steady}).
     \end{proof}

     Theorem~\ref{thm:imp_aver} provides an efficient way to improve the
     VaR value based on an existing policy. For any given policy $u \in
     \mathcal U^{\rm DS}$, we can solve a probabilistic MDP
     $\hat{\mathcal M}_1(\VaR^u)$ in (\ref{equ:promdp_s2}) and its
     optimal policy is denoted by $u'$. If its optimal value $F^*(\VaR^u)
     < \alpha$, then $\VaR^{u'} > \VaR^u$ and we update the new policy as
     $u'$ and repeat this procedure until the above condition
     $F^*(\VaR^u) < \alpha$ does not hold anymore. That is, $F^*(\VaR^u)
     \geq \alpha$ indicates that the current policy $u$ is optimal, which
     is guaranteed by the following theorem about a
     \emph{necessary-sufficient condition} of optimal policies.

     \begin{thm}\label{thm:opt_aver}
     {\rm (\textbf{Optimality Condition for Steady-State VaR Maximization MDPs})}

     A policy $u^* \in \mathcal U^{\rm {DS}}$ is optimal for the
     steady-state VaR MDP $\mathcal M_1$ in (\ref{equ:svarmdp_2}) if and
     only if it holds that $F^*(\VaR^{u^*})\ge\alpha$.
     \end{thm}
     \begin{proof}
     We first prove the necessity. Suppose $u^*$ is an optimal policy of
     $\mathcal M_1$, then we have
     \begin{equation}
        \nonumber
        \VaR^{u^*} \ge \VaR^{u}, \qquad \forall u \in \mathcal U^{\rm DS}.
     \end{equation}
     By taking $X=R^{u^*}_{\infty}$ and $Y=R^{u}_{\infty}$ in
     (\ref{eq_VaRmax-optim}) of Lemma~\ref{lem_VaRmax}, 
     we obtain
     \begin{equation}
        \nonumber
        F^u(\VaR^{u^*}) \ge \alpha, \qquad \forall u \in \mathcal U^{\rm DS}.
     \end{equation}
     Therefore, the minimal probability $F^*(\VaR^{u^*})=\min\limits_{u
        \in \mathcal U^{\rm DS}}  F^u(\VaR^{u^*}) \ge \alpha$.

     The sufficiency can also be proved by using a similar argument.
     \end{proof}

     \begin{figure}[htbp]
     \centering
     \includegraphics[width=0.9\textwidth]{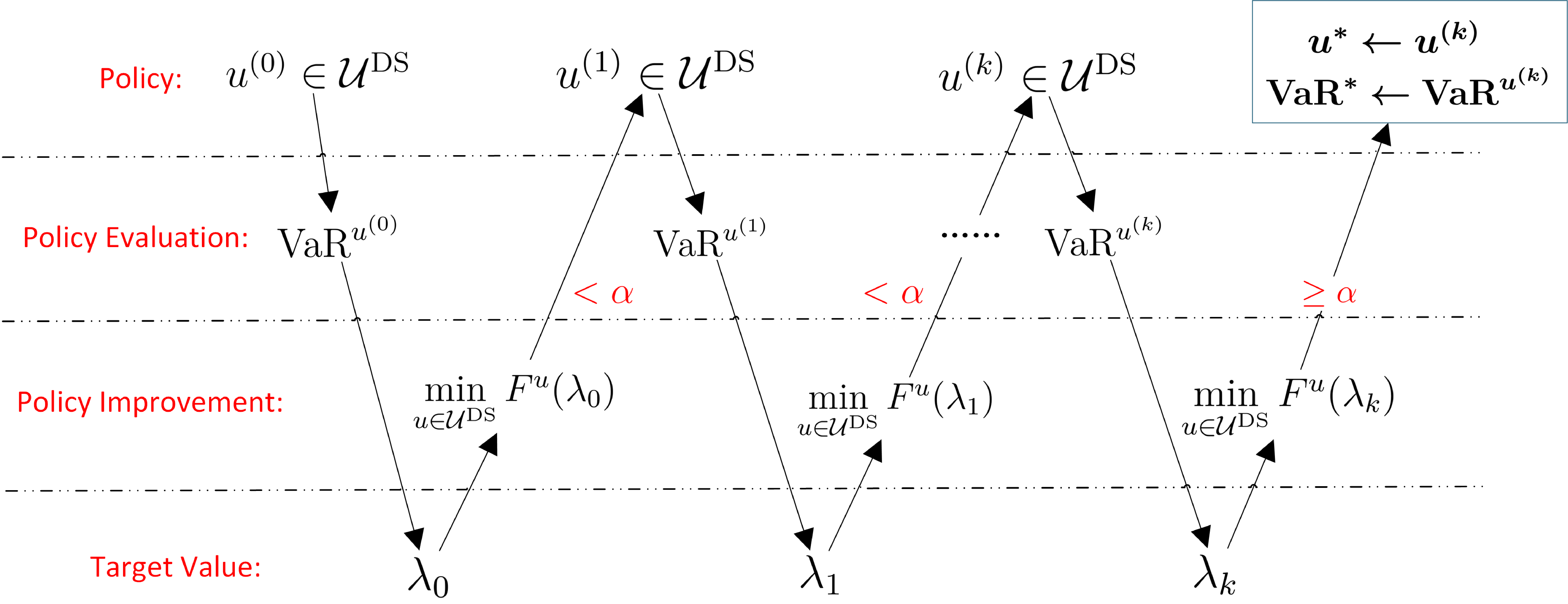}
     \caption{The iterative procedure illustration of
        Algorithm~\ref{alg:PI_steady}.}\label{fig:iter_dynamic}
        \end{figure}

        With Theorems~\ref{thm:equiva_s}, \ref{thm:imp_aver} and
        \ref{thm:opt_aver}, we can develop a policy iteration type algorithm
        to efficiently solve the steady-state VaR maximization MDP $\mathcal
        M_1$, where the policy improvement step solves a probabilistic
        minimization MDP (\ref{equ:promdp_s2}).
        Figure~\ref{fig:iter_dynamic} illustrates the iterative procedure of
        Algorithm~\ref{alg:PI_steady}. Specifically,
        Algorithm~\ref{alg:PI_steady} consists of two main components:
        policy evaluation and policy improvement. The evaluation of a policy
        can be computed with the aid of the definition of VaR and its
        steady-state distribution. An improved policy can be obtained by
        solving a probabilistic minimization MDP $\hat{\mathcal
     M}_1(\lambda)$, as stated in Theorem~\ref{thm:imp_aver}. The
     stopping rule indicates that the algorithm reaches the optimality
     condition, as guaranteed by Theorem~\ref{thm:opt_aver}. The detailed
     procedure is described in Algorithm~\ref{alg:PI_steady}. The
     algorithm's convergence is guaranteed by the following
     Theorem~\ref{thm:conver_steady}.

     \begin{algorithm}[h]
     \caption{Policy iteration type algorithm for steady-state VaR maximization MDPs}\label{alg:PI_steady}
     \begin{algorithmic}[1]
        \Require {MDP parameters $\mathcal{M}
            = \langle \mathcal S, \mathcal A, (\mathcal A(s), s \in \mathcal S), P, r  \rangle$.}
        \Ensure {An optimal policy and the maximum VaR.}
        \State \textbf{Initialization}:
        Arbitrarily choose an initial policy $u^{(0)} \in \mathcal U^{\rm DS}$, $k \gets 0$.
        \State \textbf{Policy Evaluation}: \\
        \quad Compute $F^{u^{(k)}}(\lambda)$ for all $\lambda \in \Lambda$, by using the steady-state distribution of $u^{(k)}$.\\
        \quad Compute $\VaR^{u^{(k)}} = \min\left\{\lambda \in \Lambda: F^{u^{(k)}}(\lambda) \ge \alpha\right\}.$
        \State \textbf{Policy Improvement}: \\
        \quad Update target level $\lambda_k=\VaR^{u^{(k)}}$.\\
        \quad Solve the probabilistic minimization MDP $\hat{\mathcal M}_1(\lambda_k)$ in (\ref{equ:promdp_s2}): \\
        \quad\quad  $F^*(\lambda_k)=\min\limits_{ u\in \mathcal U^{\rm DS}} F^{u}(\lambda_k)$,  $\hat{u}^* = \argmin\limits_{ u\in \mathcal U^{\rm DS}} F^{u}(\lambda_k)$.
        \While{$F^*(\lambda_k)<\alpha$}
        \State\textbf{Parameters Update}: $u^{(k+1)} \gets \hat{u}^*$, $k \gets k+1$, and go to line 2.
        \EndWhile
        \State \textbf{return}  $u^{(k)}$ and $\VaR^{u^{(k)}}$.
     \end{algorithmic}
     \end{algorithm}

     \begin{thm}\label{thm:conver_steady}
     Algorithm \ref{alg:PI_steady} converges to an optimal policy of the
     steady-state VaR maximization MDP within finite iterations.
     \end{thm}
     \begin{proof}
     First, it follows from Theorem~\ref{thm:imp_aver} that the sequence
      $\langle\lambda_k\rangle_{k \ge 0}$ generated by Algorithm~\ref{alg:PI_steady}
     is strictly increasing. Since the optimum can be attained in
     deterministic stationary policy space $\mathcal U^{\rm DS}$ which is
     finite, Algorithm~\ref{alg:PI_steady} will stop at a policy $u^{(n)}
     \in \mathcal U^{\rm DS}$ after a finite number of iterations. Via
     the stopping rule in line 9, we have $F^*(\VaR^{u^{(n)}}) \ge
     \alpha$, implying that $u^{(n)}$ is an optimal policy of VaR MDP
     $\mathcal M_1$ according to Theorem~\ref{thm:opt_aver}. Thus, the
     theorem is proved.
     \end{proof}


     \section{Maximization of Finite-Horizon VaR MDPs}\label{sec:tvarmdp}
     In this section, we propose an optimization approach to solve the
     finite-horizon VaR maximization MDP $\mathcal M_2$ defined in (\ref{equ:tvarmdp})
     with the aid of finite-horizon probabilistic MDPs. It is worth
     noting that the $T$-horizon accumulated reward $R_{0:T}$ is a
     discrete random variable with support on
     $\Lambda_0:=\left\{\sum\limits_{t=0}^{T-1}r(s_t,a_t):(s_t,a_t) \in
     \mathcal K\right\}$. Thus, we can simplify (\ref{equ:tvar}) to a
     computable form below.
     \begin{equation}\label{equ:new_tvar}
     \VaR^u(s_0)=\min\left\{\lambda_0 \in \Lambda_0: \mathbbm{P}_{s_0}^u
     (R_{0:T} \le \lambda_0) \ge \alpha\right\}, \qquad s_0 \in \mathcal
     S, u \in \mathcal U^{\rm RH}.
     \end{equation}
     Thus, the finite-horizon VaR MDP $\mathcal M_2$ in
     (\ref{equ:tvarmdp}) can be rewritten as the following form of
     \emph{bilevel optimization} with constraints.
     \begin{equation}\nonumber
     \VaR^*(s_0) = \sup_{u \in \mathcal U^{\rm RH}}\min_{\lambda_0 \in
        \Lambda_0}\left\{\lambda_0 : \mathbbm{P}_{s_0}^u (R_{0:T} \le
     \lambda_0) \ge \alpha\right\}, \qquad s_0 \in \mathcal S.
     \end{equation}
     However, the above problem is difficult to solve directly. Based on
     Lemma~\ref{lem_VaRmax}, if we set $Y=R_{0:T}$ in
     (\ref{eq_VaRmax-impr}), we have
     \begin{equation}\label{equ:relation}
     \mathbbm{P}_{s_0}^u(R_{0:T} \le \lambda_0)<\alpha
     \Longleftrightarrow \VaR^u(s_0)>\lambda_0, \qquad s_0 \in \mathcal
     S, u \in \mathcal U^{\rm RH}, \lambda_0 \in \mathbb R.
     \end{equation}
     Similar to the method used in Section~\ref{sec:svarmdp}, we define
     an auxiliary finite-horizon probabilistic minimization MDP to help
     solve $\mathcal M_2$. Given an initial state $s_0 \in \mathcal S$
     and a target level $\lambda_0 \in \mathbb R$, the probability of
     $T$-horizon accumulated rewards that does not exceed $\lambda_0$
     under policy $u \in \mathcal U^{\rm RH}$ is
     \begin{equation}\label{equ:pro}
     F^u(s_0,\lambda_0):=\mathbbm{P}^{u}_{s_0} (R_{0:T} \le \lambda_0).
     \end{equation}
     We define a finite-horizon probabilistic minimization MDP
     $\hat{\mathcal M}_2(\lambda_0)$, whose
     objective is to minimize the probability (\ref{equ:pro}) among
     randomized history-dependent policies for each initial state $s_0$,
     i.e.,
     \begin{flalign}\label{equ:promdp}
     &\hat{\mathcal M}_2(\lambda_0): \hspace{3cm} F^*(s_0,\lambda_0):=
     \inf\limits_{u \in \mathcal U^{\rm RH}}   F^u(s_0,\lambda_0), \qquad
     s_0 \in \mathcal{S},&
     \end{flalign}
     where $F^*(s_0,\lambda_0)$ is called the \emph{optimal value
     function} of the finite-horizon probabilistic minimization MDP. A
     policy $u^* \in \mathcal U^{\rm RH}$ is \emph{optimal} for the
     probabilistic minimization MDP if
     \begin{equation}\nonumber
     F^{u^*}(s_0,\lambda_0) = F^*(s_0,\lambda_0),  \qquad\forall s_0 \in
     \mathcal{S}.
     \end{equation}
     If we can solve $\hat{\mathcal M}_2(\lambda_0)$ and obtain an
     optimal policy $u$ that satisfies $F^{u}(s_0,\lambda_0) < \alpha$,
     (\ref{equ:relation}) guarantees that $\VaR^u(s_0)>\lambda_0$ and the
     VaR value under this policy $u$ is improved compared with the target
     level $\lambda_0$. Thus, we can update the target level as
     $\lambda_0=\VaR^u(s_0)$ and repeat this process to continually find
     improved policies.

     However, due to the non-linearity and non-additivity of the
     probability metric, the classical dynamic programming is not
     applicable to solve this finite-horizon probabilistic minimization
     MDP $\hat{\mathcal M}_2(\lambda_0)$. To overcome this challenge, we
     convert $\hat{\mathcal M}_2(\lambda_0)$ to a standard MDP by using
     the technique of augmenting state space, where we treat the
     remaining goal at current stage as the auxiliary state \citep{white1993,huo2017}.

     We define an augmented MDP by tuple $\langle
     \tilde{\mathcal{S}},\tilde{\mathcal{A}},(\tilde{\mathcal{A}}(\tilde{s}),
     \tilde{s} \in \tilde{\mathcal{S}}),\tilde P,\tilde r\rangle$ with a
     2-dimensional state space $\tilde{\mathcal{S}} := \mathcal{S} \times
     \mathbb{R}$, where the first dimension is the state of the original
     MDP and the second dimension represents the remaining goal at
     current stage. The action space $\tilde{\mathcal{A}}:=\mathcal{A}$
     and the admissible action set
     $\tilde{\mathcal{A}}(s,\lambda):=\mathcal{A}(s)$, for any augmented
     state $(s,\lambda) \in \tilde{\mathcal{S}}$. Suppose the state is
     $(s_t,\lambda_t) \in \tilde{\mathcal{S}}$ at time $t$ and an action
     $a_t \in \mathcal{A}(s_t)$ is adopted, the system will receive an
     instantaneous reward $\tilde{r}_t(s_t,\lambda_t,a_t)$ and move to a
     new state $(s_{t+1},\lambda_{t+1}) \in \tilde{\mathcal{S}}$ at time
     $t+1$. The transition kernel $\tilde P$ and the reward
     $\tilde r_t(s_t,\lambda_t,a_t)$ are determined as follows.
     \begin{equation}\nonumber
     \tilde{P}(s_{t+1},\lambda_{t+1}|s_t,\lambda_t, a_t) :=
     P(s_{t+1}|s_t,a_t)\times \mathbbm I \{\lambda_{t+1} = \lambda_{t} -
     r(s_t, a_t)\}, \quad t=0,1,\dots, T-1,
     \end{equation}
     \begin{equation}\nonumber
     \tilde{r}_t(s_t,\lambda_t,a_t) := \left\{
     \begin{array}{ll}
        0, \quad & t=0,1,\dots,T-1, \\
        \mathbbm{I}\left\{0 \le \lambda_T\right\}, \quad & t=T.
     \end{array}
     \right.
     \end{equation}
     We denote by $\tilde{\mathcal U}^{\rm RH}$ the set of all randomized
     history-dependent policies $\tilde{u}:=(\tilde{u}_t;~ t \ge 0 )$,
     where $\tilde{u}_t$ is a probability measure on $\mathcal{A}$ given
     history $\tilde{h}_t := (s_0,\lambda_0,a_0,\ldots,s_t,\lambda_t)$.
     Similarly, we denote by $\tilde{\mathcal U}^{\rm RM}$ and
     $\tilde{\mathcal U}^{\rm DM}$ the sets of all randomized Markov
     policies and deterministic Markov policies of this augmented MDP,
     respectively. Given initial state $(s_0,\lambda_0) \in
     \tilde{\mathcal{S}}$ and policy $\tilde{u} \in \tilde{\mathcal
     U}^{\rm RH}$, we denote by $\mathbbm{P}_{s_0,\lambda_0}^{\tilde{u}}$
     the unique probability measure on the space of trajectories of
     augmented states and actions and by
     $\mathbbm{E}_{s_0,\lambda_0}^{\tilde{u}}$ the expectation operator
     corresponding to $\mathbbm{P}_{s_0,\lambda_0}^{\tilde{u}}$.

     With the definition of this new finite-horizon augmented MDP, we
     focus on the criterion of expected total rewards. Given an initial
     state $(s_0,\lambda_0) \in \tilde{\mathcal{S}}$ and a policy $\tilde
     u \in \tilde{\mathcal U}^{\rm RH}$, we define the $T$-horizon
     expected rewards as follows.
     \begin{eqnarray*}
     V_0^{\tilde{u}}(s_0,\lambda_0) &:=& \mathbbm{E}_{s_0,\lambda_0}^{\tilde{u}}
     \big[\sum_{t=0}^{T}\tilde{r}_{t}(s_t,\lambda_t,a_t)] = \mathbbm{E}_{s_0,\lambda_0}^{\tilde{u}}
     \Big[\mathbb{I}\left\{0 \le \lambda_T\right\}\Big] =
     \mathbbm{E}_{s_0,\lambda_0}^{\tilde{u}}
     \Big[\mathbb{I}\left\{\sum_{t=0}^{T-1}{r}(s_t,a_t) \le \lambda_0\right\}\Big] \\&=&
     \mathbbm{P}_{s_0,\lambda_0}^{\tilde{u}}\Big(\sum_{t=0}^{T-1}{r}(s_t,a_t)
     \le \lambda_0\Big),
     \end{eqnarray*}
     where the third equality recursively utilizes the fact
     $\lambda_{t+1}=\lambda_t-r(s_t,a_t)$. This expectation of total rewards is exactly the same as
     the probability defined in (\ref{equ:pro}). We define a
     finite-horizon augmented MDP $\tilde{\mathcal M}_2$ that aims to
     minimize the above expected total rewards for each initial state
     $(s_0,\lambda_0) \in \tilde{\mathcal S}$, i.e.,
     \begin{flalign}\label{equ:newmdp}
     &\tilde{\mathcal M}_2: \hspace{3cm} V_0^*(s_0,\lambda_0) =
     \inf\limits_{\tilde{u} \in \tilde{\mathcal U}^{\rm RH}}
     V_0^{\tilde{u}}(s_0,\lambda_0), \qquad (s_0,\lambda_0) \in
     \tilde{\mathcal S},&
     \end{flalign}
     where $V_0^*(s_0,\lambda_0)$ is called the \emph{optimal value
     function} of this augmented MDP $\tilde{\mathcal M}_2$. A policy
     $\tilde{u}^* \in \tilde{\mathcal U}^{\rm RH}$ is \emph{optimal} if
     it attains the above optimal values, i.e.,
     \begin{equation*}
     V_0^{\tilde{u}^*}(s_0,\lambda_0) = V_0^*(s_0,\lambda_0), \qquad
     \forall (s_0,\lambda_0) \in \tilde{\mathcal S}.
     \end{equation*}

     Since $\tilde{\mathcal M}_2$ is a standard MDP with expectation
     criterion, we can restrict its policy search space to a
     deterministic Markov policy space $\tilde{\mathcal U}^{\rm DM}$ and
     solve it by dynamic programming \citep{puterman1994markov}. To this end, we
     define the Bellman operator and the Bellman optimality operator. Let
     $\mathcal B(\tilde{\mathcal{S}})$ be the space of all bounded
     functions on $\tilde{\mathcal{S}}$ and $\varphi: \tilde{\mathcal S}
     \rightarrow \mathcal P(\mathcal A)$, we define the Bellman operator
     $\mathbbm L^\varphi: \mathcal B(\tilde{\mathcal{S}}) \rightarrow
     \mathcal B(\tilde{\mathcal{S}})$ for policy evaluation as
     \begin{equation}\label{equ:bellman}
     \mathbbm L^\varphi v(s,\lambda) := \sum\limits_{a \in
        \mathcal{A}(s)}\varphi(a|s,\lambda) \left\{\sum\limits_{s' \in
        \mathcal{S}}P(s'|s,a)v(s',\lambda-r(s,a))\right\},\quad v \in
     \mathcal B(\tilde{\mathcal{S}}),~ (s,\lambda) \in
     \tilde{\mathcal{S}},
     \end{equation}
     where $\varphi$ can be viewed as a Markov decision rule that can be
     randomized or deterministic. We also define the Bellman optimality
     operator $\mathbbm L^*: \mathcal B(\tilde{\mathcal{S}}) \rightarrow
     \mathcal B(\tilde{\mathcal{S}})$ for optimization as
     \begin{equation}\label{equ:bellman_opt}
     \mathbbm L^* v(s,\lambda) := \min\limits_{a \in \mathcal{A}(s)}
     \left\{\sum\limits_{s' \in
        \mathcal{S}}P(s'|s,a)v(s',\lambda-r(s,a))\right\},\qquad v \in
     \mathcal B(\tilde{\mathcal{S}}),~ (s,\lambda) \in
     \tilde{\mathcal{S}}.
     \end{equation}

     By utilizing the Bellman operator in (\ref{equ:bellman}) and the
     Bellman optimality operator in (\ref{equ:bellman_opt}) recursively,
     we can evaluate $V_0^{\tilde{u}}(s_0,\lambda_0)$ for any Markov
     policy $\tilde{u} \in \tilde{\mathcal U}^{\rm RM}$ and solve the
     optimal value function $V_0^{*}(s_0,\lambda_0)$ as well as the
     optimal policy $\tilde{u}^*$ for the standard MDP $\tilde{\mathcal
     M}_2$. We summarize this result in Lemma \ref{lem:dp} as follows.

     \begin{lem}\label{lem:dp}
     (a) Given a Markov policy
     $\tilde{u}=(\tilde{u}_0,\ldots,\tilde{u}_{T-1}) \in \tilde{\mathcal
        U}^{\rm RM}$, the value $V_0^{\tilde{u}}(s_0,\lambda_0)$ can be
     computed by successively conducting a series of Bellman operators
     $(\mathbbm L^{\tilde{u}_t};~ t=0,\ldots,T-1)$ with initial value
     $V_T^{\tilde{u}}(s_T,\lambda_T):=\mathbb{I}\left\{0 \le
     \lambda_T\right\}$, that is,
     \begin{equation}\label{equ:dp}
        V_t^{\tilde{u}}(s_t,\lambda_t)=\mathbbm
        L^{\tilde{u}_t}V_{t+1}^{\tilde{u}}(s_t,\lambda_t), \qquad
        t=0,\dots,T-1,~ (s_t,\lambda_t) \in \tilde{\mathcal S}.
     \end{equation}
     (b) The optimal value function $V_0^*(s_0,\lambda_0)$ can be solved
     by successively conducting the Bellman optimality operator $\mathbbm
     L^*$ with initial value $V_T^*(s_T,\lambda_T):=\mathbb{I}\left\{0
     \le \lambda_T\right\}$, that is,
     \begin{equation}\label{equ:opt_dp}
        V_t^*(s_t,\lambda_t)=\mathbbm L^*V_{t+1}^*(s_t,\lambda_t), \qquad
        t=0,\ldots,T-1,~ (s_t,\lambda_t) \in \tilde{\mathcal S}.
     \end{equation}
     Further, there exists $a_t^*\in \mathcal A(s_t)$ that attains the
     minimum in $\mathbbm L^* V_{t+1}^*(s_t,\lambda_t)$, and the
     deterministic Markov policy $\tilde{u}^*=(\tilde{u}_t^*;~ t \ge 0)
     \in \tilde{\mathcal U}^{\rm {DM}}$ with
     $\tilde{u}_t^*(s_t,\lambda_t) = a_t^*$ is an optimal policy for the
     augmented MDP $\tilde {\mathcal M}_2$ in (\ref{equ:newmdp}).
     \end{lem}

     The next lemma establishes the relations between the finite-horizon
     probabilistic minimization MDP $\hat{\mathcal M}_2(\lambda_0)$ in
     (\ref{equ:promdp}) and the finite-horizon augmented MDP
     $\tilde{\mathcal M}_2$ with expectation criterion in
     (\ref{equ:newmdp}), and further derives the optimal policy and the
     optimal value function of $\hat{\mathcal M}_2(\lambda_0)$.

     \begin{lem}\label{lem:optpolicy}
     (a)  Given a deterministic Markov policy $\tilde u=(\tilde
     u_0,\ldots,\tilde u_{T-1}) \in \tilde{\mathcal U}^{\rm DM}$ and a
     target level $\lambda_0 \in \Lambda_0$, if a deterministic
     history-dependent policy $u=( u_0,\ldots, u_{T-1}) \in
     \mathcal{U}^{\rm DH}$ satisfies $ u_{t}(s_0,a_0,\ldots,s_t)=
     \tilde{u}_t(s_t,\lambda_0-\sum\limits_{\tau=0}^{t-1}
     r_\tau(s_\tau,a_\tau))$, then we have
     \begin{equation} \label{equ:eval_equal}
        F^{u}(s_0,\lambda_0)=V_0^{\tilde u}(s_0,\lambda_0), \qquad \forall
        s_0 \in \mathcal S.
     \end{equation}
     (b) Suppose the deterministic Markov policy $\tilde u^*=(\tilde
        u^*_0,\ldots,\tilde u^*_{T-1}) \in \tilde{\mathcal U}^{\rm DM}$ is an
     optimal policy for the augmented MDP $\tilde{\mathcal M}_2$ in
     (\ref{equ:newmdp}), then the deterministic history-dependent policy
     $u^*=(u^*_0,\ldots, u^*_{T-1}) \in
        \mathcal{U}^{\rm DH}$ with $
     u^*_{t}(s_0,a_0,\ldots,s_t) :=
     \tilde{u}_t^*(s_t,\lambda_0-\sum\limits_{\tau=0}^{t-1}
     r_\tau(s_\tau,a_\tau))$ is an optimal policy for the finite-horizon
     probabilistic minimization MDP $\hat{\mathcal M}_2(\lambda_0)$ in
     (\ref{equ:promdp}). Moreover, we have
     \begin{equation}
        \label{equ:value_equal} F^*(s_0,\lambda_0)=V_0^*(s_0,\lambda_0),
        \qquad \forall (s_0,\lambda_0) \in \tilde{\mathcal S}.
     \end{equation}
     \end{lem}

     \begin{remark}
     Lemma~\ref{lem:optpolicy} implies that the optimum of the
     finite-horizon probabilistic minimization MDP $\hat{\mathcal
        M}_2(\lambda_0)$ in (\ref{equ:promdp}) can be attained by a
     \textit{deterministic history-dependent} policy in $\mathcal{U}^{\rm
        DH}$, which is not Markovian since
     $\lambda_t:=\lambda_0-\sum\limits_{\tau=0}^{t-1}r_\tau(s_\tau,a_\tau)$
     relies on the history rewards up to time~$t$. Therefore, we cannot
     limit our policy space to $\mathcal{U}^{\rm DM}$ for solving
     $\hat{\mathcal M}_2(\lambda_0)$, which is contrast to the optimality
     of $\mathcal{U}^{\rm DM}$ for standard MDPs.
     \end{remark}

     Lemma~\ref{lem:dp} can be derived based on the classical results of
     finite-horizon standard MDPs \citep[Chapter~4]{puterman1994markov} and
     Lemma~\ref{lem:optpolicy} follows directly from the aforementioned
     construction of the augmented MDP $\tilde{\mathcal M}_2$. Their
     detailed proofs are omitted due to space limitations. With
     Lemmas~\ref{lem:dp} and \ref{lem:optpolicy}, the probabilistic
     minimization MDP $\hat{\mathcal M}_2(\lambda_0)$ in
     (\ref{equ:promdp}) can be solved by using backward dynamic
     programming with (\ref{equ:opt_dp}). It is worth pointing out that
     (\ref{equ:dp}) provides a dynamic programming perspective for
     evaluating $\VaR^u(s_0)$ based on (\ref{equ:new_tvar}). Furthermore,
     we establish the relation between the finite-horizon VaR MDP
     $\mathcal M_2$ in (\ref{equ:tvarmdp}) and the finite-horizon
     augmented MDP $\tilde{\mathcal M}_2$ in (\ref{equ:newmdp}), as
     stated in Theorem~\ref{thm:equiva}.

     \begin{thm}\label{thm:equiva}{\rm (\textbf{Optimization Equivalence between $\mathcal M_2$ and
            $\tilde{\mathcal M}_2$})}

     \begin{itemize}
        \item[(a)] The optimal value function of the finite-horizon VaR maximization MDP $\mathcal M_2$ in (\ref{equ:tvarmdp})~ is given by        \begin{equation}\label{equ:opt_value}
            \VaR^*(s_0) = \min\limits_{\lambda_0 \in \Lambda_0}\left\{\lambda_0:
            \min\limits_{\tilde{u} \in \tilde{\mathcal U}^{\rm DM}}
            V_0^{\tilde{u}}(s_0,\lambda_0) \ge \alpha\right\}, \qquad s_0 \in
            \mathcal S.
        \end{equation}

\item[(b)]
Let $\VaR^*_{-}(s_0,\Lambda_0)$ be the left predecessor of
$\VaR^*(s_0)$ in $\Lambda_0$, and suppose $\tilde{u}^*$ is an
optimal deterministic Markov policy
        of the augmented MDP $\tilde {\mathcal M_2}$ in (\ref{equ:newmdp}),
        then there exists an optimal deterministic history-dependent policy
        $u^*=(u^*_{0}, \ldots, u^*_{T-1}) \in \mathcal U^{\rm {DH}}$ for
        $\mathcal M_2$, given by
        \begin{equation}\label{equ:optimal_policy}
            u^*_{t}(s_0,a_0,\dots,s_t) = \tilde{u}_t^*(s_t,
            \VaR^*_{-}(s_0,\Lambda_0)-\sum\limits_{\tau=0}^{t-1} r(s_\tau,a_\tau)),
            \qquad t=0,\dots,T-1.
        \end{equation}
     \end{itemize}
     \end{thm}

     Theorem~\ref{thm:equiva} implies that we can limit our policy
     search space to $\mathcal U^{\rm DH}$, which is finite. According
     to Theorem~\ref{thm:equiva},  we can convert the finite-horizon VaR
     MDP to a \emph{bilevel optimization problem} with constraints, i.e.,
     \begin{equation}\label{equ:bileve}
     \VaR^*(s_0)=\max\limits_{u \in \mathcal U^{\rm
            DH}}\min\limits_{\lambda_0 \in \Lambda_0}\left\{\lambda_0:
     F^u(s_0,\lambda_0)\ge \alpha\right\}=\min\limits_{\lambda_0 \in
        \Lambda_0}\left\{\lambda_0: \min\limits_{u \in \mathcal U^{\rm
            DH}}F^u(s_0,\lambda_0)\ge \alpha\right\}, \quad s_0 \in \mathcal S,
            \end{equation}
            where the inner level is a policy optimization problem of minimizing
            the finite-horizon probabilistic MDP whose optimal value should be
            larger than $\alpha$, and the outer level is a single parameter
            optimization problem with variable $\lambda_0 \in \Lambda_0$. A
            brute-force method for solving the bilevel problem
            (\ref{equ:bileve}) is to solve a series of finite-horizon
            probabilistic minimization MDPs $\left\{\hat{\mathcal
     M}_2(\lambda_0): \lambda_0 \in \Lambda_0\right\}$ in
     (\ref{equ:promdp}), which is also equivalent to solving a
     finite-horizon standard MDP $\tilde{\mathcal M}_2$ in
     (\ref{equ:newmdp}) with augmented state space. Then we can obtain
     the optimal VaR using (\ref{equ:opt_value}) and an optimal policy
     using (\ref{equ:optimal_policy}). However, this naive method is
     computationally intractable since the size of $\Lambda_0$ increases
     exponentially with horizon $T$.

     Using the similar idea of Algorithm~\ref{alg:PI_steady} in
     Section~\ref{sec:svarmdp}, we also develop a policy iteration type
     algorithm to optimize the finite-horizon VaR MDP $\mathcal M_2$.
     With (\ref{equ:relation}) and the definition of finite-horizon
     probabilistic minimization MDP $\hat{\mathcal M}_2(\lambda_0)$ in
     (\ref{equ:promdp}), we derive the following policy improvement rule
     and the optimality condition of the finite-horizon VaR MDP $\mathcal
     M_2$.

     \begin{cor} \label{cor:imp}
     {\rm (\textbf{Policy Improvement Rule for Finite-Horizon VaR Maximization~MDPs})}

      Given a policy $u \in \mathcal U^{\rm DH}$, if there exists another policy $u' \in \mathcal U^{\rm DH}$ such that $F^{u'} (s_0,
     \VaR^u(s_0))<\alpha$ for some $s_0 \in \mathcal S$, then we have $\VaR^{u'}(s_0)>\VaR^u(s_0)$. If it holds that $F^{u'} (s_0, \VaR^u(s_0))<\alpha$ for any $s_0 \in \mathcal S$, then $u'$ is a strictly improved policy, i.e., $\VaR^{u'}(\cdot)>\VaR^u(\cdot)$.

     \end{cor}

     \begin{cor}\label{cor:opt}
     {\rm (\textbf{Optimality Condition for Finite-Horizon VaR Maximization MDPs})}

     A policy $u^* \in \mathcal U^{\rm DH}$ is optimal for the
     finite-horizon VaR MDP $\mathcal M_2$ in (\ref{equ:tvarmdp}) if and
     only if it holds that $F^* (s_0, \VaR^{u^*}(s_0))\ge\alpha$  for any $ s_0
     \in \mathcal S$.
     \end{cor}
     %

     With Corollaries~\ref{cor:imp} and \ref{cor:opt}, we can propose a
     policy iteration type algorithm to efficiently solve the
     finite-horizon VaR MDP (\ref{equ:tvarmdp}), where the policy
     evaluation step computes the VaR of the current policy and the
     policy improvement step solves a probabilistic minimization MDP to
     find a better policy. Specifically, for a given policy $u \in
     \mathcal U^{\rm DH}$, we can evaluate the corresponding value
     $\VaR^u(s_0)$ with the definition in (\ref{equ:new_tvar}) by
     recursively using the Bellman operator (\ref{equ:dp}). Then we solve
     a finite-horizon probabilistic minimization MDP $\hat{\mathcal
     M}_2(\lambda_0)$ by letting $\lambda_0=\VaR^u(s_0)$, and obtain an
     optimal policy $u' \in \mathcal U^{\rm DH}$ and its optimal value
     function $F^*(s_0,\VaR^u(s_0))$. If $F^*(s_0,\VaR^u(s_0)) < \alpha$,
     then we update $u = u'$ and repeat this procedure, until
     $F^*(s_0,\VaR^u(s_0)) \ge \alpha$. The detailed procedure is stated
     in Algorithm~\ref{alg:PI_finte_1}. The algorithm's convergence is
     guaranteed by the following Theorem~\ref{thm:conver_finite}.

     \begin{algorithm}[h]
     \caption{Policy iteration type algorithm for finite-horizon VaR maximization MDP
        $\mathcal M_2$}\label{alg:PI_finte_1}
     \begin{algorithmic}[1]
        \Require {MDP parameters $\mathcal{M}
            = \langle \mathcal S, \mathcal A, (\mathcal A(s), s \in \mathcal S), P, r  \rangle$},  horizon $T$, initial state $s_0 \in \mathcal S$.
        \Ensure {An optimal policy and the maximum VaR.}
        \State \textbf{Initialization}:
        Arbitrarily choose a policy $u^{(0)} \in \mathcal U^{\rm DH}$, $k \gets 0$.
        \State \textbf{Policy Evaluation}: \\
        \quad Compute $F^{u^{(k)}}(s_0,\lambda_0)$ for all $\lambda_0 \in \Lambda_0$, by dynamic programming (\ref{equ:dp}) and (\ref{equ:eval_equal}).\\
        \quad Compute $\VaR^{u^{(k)}}(s_0)=\min\left\{\lambda_0 \in \Lambda_0: F^{u^{(k)}}(s_0,\lambda_0)\ge \alpha\right\}.$
        \State \textbf{Policy Improvement}: \\
        \quad Update target level: $\lambda_0^{(k)}=\VaR^{u^{(k)}}(s_0)$.\\
        \quad Solve the augmented MDP $\tilde{\mathcal M}_2$ (by dynamic programming (\ref{equ:opt_dp})): \\
        \quad\quad  $V_0^*(s_0,\lambda_0^{(k)}) = \min\limits_{\tilde u \in \tilde{\mathcal U}^{\rm DM}} V_0^{\tilde u}(s_0,\lambda_0^{(k)}).$\\
        \quad \quad $\tilde{u}^* = (\tilde{u}_0^*,\ldots,\tilde{u}_{T-1}^*) = \argmin\limits_{\tilde u \in \tilde{\mathcal U}^{\rm DM}} V_0^{\tilde u}(s_0,\lambda_0^{(k)})$. \\
        \quad Generate a new policy by Lemma \ref{lem:optpolicy}:\\
        \quad\quad $u'_{t}(s_0,a_0,\dots,s_t) := \tilde{u}_t^*(s_t,\lambda_0^{(k)}-\sum\limits_{\tau=0}^{t-1} r_\tau(s_\tau,a_\tau)),\quad t=0,\dots,T-1.$
        \While{$F^*(s_0,\lambda_0^{(k)})<\alpha$}
        \State\textbf{Parameters Update}: $u^{(k+1)} \gets u', k \gets k+1$, and go to line 2.
        \EndWhile
        \State \textbf{return}  $u^{(k)}$ and $\VaR^{u^{(k)}}(s_0)$.
     \end{algorithmic}
     \end{algorithm}

     \begin{thm}\label{thm:conver_finite}
     Algorithm \ref{alg:PI_finte_1} converges to an optimal policy of the
     finite-horizon VaR maximization MDP $\mathcal M_2$ in (\ref{equ:tvarmdp}) within
     finite iterations.
     \end{thm}
     \begin{proof}
     The proof is similar to that of Algorithm~\ref{alg:PI_steady} in
     Theorem~\ref{thm:conver_steady}. For each given $s_0 \in \mathcal
     S$, the update condition $F^*(s_0,\lambda_0^{(k)})<\alpha$ in
     line~12 indicates that $\VaR^{u^{(k+1)}}(s_0)>\VaR^{u^{(k)}}(s_0)$
     by Corollary~\ref{cor:imp}. That is, the VaR in each iteration is
     strictly improved. With Theorem~\ref{thm:equiva}, the policy search space is limited to the deterministic history-dependent policy space
     $\mathcal U^{\rm DH}$ which is finite, the algorithm must stop after
     a finite number of iterations. The algorithm stops when the
     condition $F^*(s_0,\lambda_0^{(k)}) \geq \alpha$ in line~12 holds,
     which indicates that the converged policy $u^{(k)}$ is optimal by
     Corollary~\ref{cor:opt}.
     \end{proof}

     Algorithm~\ref{alg:PI_finte_1} can find the optimal policy $u^*$
     with a given initial state $s_0$. For different $s_0 \in \mathcal
     S$, we can repeat Algorithm~\ref{alg:PI_finte_1} to find the
     associated $u^*$ where some intermediate results can be shared for
     saving computations. Algorithms \ref{alg:PI_steady} and
     \ref{alg:PI_finte_1} have a form similar to the classical policy
     iteration in MDP theory. They are reasonably expected to have a fast
     convergence speed in most scenarios, which is widely observed in
     classical policy iteration. We will demonstrate their convergence
     behaviors through the later experiments in Section~\ref{sec:numer}.

     \begin{remark}
     The finite-horizon VaR MDP has also been studied by
     \cite{li2022quantile}, which establishes a decomposition form of VaR
     of accumulated rewards and derives a value iteration type algorithm
     to find optimal policies. As a comparison, we establish a specified
     equivalence between VaR maximization MDP and probabilistic minimization MDP,
     and propose a policy iteration type algorithm. Our policy iteration
     algorithm is an effective supplement to the value iteration
     algorithm for solving VaR MDPs.
     \end{remark}

     \section{Extensions}\label{sec:exten}
     The previous sections study the VaR maximization MDPs, which
     correspond to maximizing the guaranteed rewards at a given
     confidence level. When we consider costs instead of rewards, it is
     reasonable to study VaR minimization problems. In this section, we
     extend our optimization approaches to VaR minimization in both
     steady-state and finite-horizon MDPs. Note that we use the same
     notations as those in Sections~\ref{sec:svarmdp} and
     \ref{sec:tvarmdp} unless otherwise declared, and $r(s,a)$ represents
     the instantaneous cost in this section.

     \subsection{Steady-State VaR Minimization MDPs} \label{sec:smdp_min}
     In this subsection, we study steady-state VaR minimization MDPs,
     where the objective is to minimize the VaR of instantaneous costs when the
     system tends to be steady over an infinite horizon, i.e.,
     \begin{flalign}\label{equ:svarmdp_min}
     &\mathcal M_3: \hspace{5cm} \VaR^*:= \inf\limits_{u \in \mathcal
        U^{\rm RS}}\VaR^u.&
        \end{flalign}
        We can use the similar framework in
        Section~\ref{sec:svarmdp} to study the above problem, except that
        Lemma~\ref{lem_VaRmax} should be replaced by Lemma~\ref{lem_VaRmin}.
        By taking $Y=R^u_{\infty}$ in Lemma~\ref{lem_VaRmin}, we can obtain
        the equivalence between steady-state VaR and steady-state
        probability, i.e.,
         \begin{equation}\label{equ:relation_steady_min}
            \mathbbm P(R^u_{\infty} \le \lambda_{-}(\Lambda)) \ge \alpha \Longleftrightarrow
            \VaR^u < \lambda, \qquad \forall u \in \mathcal U^{\rm RS},\lambda
            \in \mathbb R.
        \end{equation}
         Thus, we focus on maximizing $\mathbbm P(R^u_{\infty} \le \lambda_{-}(\Lambda))$
        such that the left-hand-side of (\ref{equ:relation_steady_min}) can
        be satisfied and $\VaR^u$ can be reduced. We formulate the probabilistic maximization
        MDP as
     \begin{flalign}\label{equ:sprobmdp_min}
     &\hat{\mathcal{M}}_3(\lambda): \hspace{5cm} F^*(\lambda)
     :=\sup\limits_{u\in \mathcal U^{\rm {RS}}}F^u(\lambda),&
     \end{flalign}
     where $F^u(\lambda) := \mathbbm P(R^u_{\infty} \le \lambda)$. This is a maximization counterpart of the
     probabilistic minimization MDP $\hat{\mathcal{M}}_1(\lambda)$
     (\ref{equ:promdp_s}) in Section~\ref{sec:svarmdp}, which can also be
     solved as a standard average maximization MDP with reward function
     $r_\lambda(s,a):=\mathbb{I}\left\{r(s,a) \le \lambda\right\}$.
     Similar to Theorem~\ref{thm:equiva_s}, we derive the following
     theorem about the equivalence of optimal value and optimal policy
     between (\ref{equ:svarmdp_min}) and (\ref{equ:sprobmdp_min}).

     \begin{thm}\label{thm:equiva_min_s} {\rm (\textbf{Optimization Equivalence between $\mathcal M_3$ and
            $\hat{\mathcal M}_3(\lambda)$})}

     \begin{itemize}
        \item[(a)] The optimal value of the steady-state VaR minimization MDP $\mathcal M_3$ in (\ref{equ:svarmdp_min})
        is given by
        \begin{equation}\label{equ:opt_value_min_s}
            \VaR^* = \min\limits_{\lambda \in \Lambda}\left\{\lambda:
            \max\limits_{u \in \mathcal U^{\rm DS}}F^u(\lambda) \ge
            \alpha\right\}.
        \end{equation}
        \item[(b)] Suppose ${u}^*$ is an optimal deterministic stationary policy of the probabilistic maximization MDP $\hat{\mathcal M}_3(\VaR^*)$, then
        $u^*$ is also optimal for $\mathcal M_3$ .
     \end{itemize}
     \end{thm}

     \begin{proof}
        For notational simplicity,  we denote $\lambda^* := \min\limits_{\lambda \in \Lambda}\left\{\lambda:
        \max\limits_{u \in \mathcal U^{\rm DS}}F^u(\lambda) \ge
        \alpha\right\}$. Let $\hat u^* \in \mathcal U^{\rm DS}$ be the optimal policy of the probabilistic maximization MDP $\hat{\mathcal M}_3(\lambda^*)$. With these definitions, to prove (a) and (b), it is sufficient to show that
        \begin{equation}\label{equ:opt_ge_min_s}
            \lambda^* \le \VaR^u, \quad\forall u \in \mathcal U^{\rm RS},
        \end{equation}
        \begin{equation}\label{equ:opt_equ_min_s}
            \lambda^* \ge \VaR^{\hat u^*}.
        \end{equation}
        First, we prove (\ref{equ:opt_ge_min_s}). For any given $u \in \mathcal U^{\rm RS}$, it follows from the definition of $\VaR^u$, we have
        \begin{equation*}
            F^u(\VaR^u) \ge \alpha,
        \end{equation*}
     which directly leads to
     \begin{equation}
        \nonumber
        \max\limits_{u' \in \mathcal U^{\rm DS}}F^{u'}(\VaR^u)=\max\limits_{u' \in \mathcal U^{\rm RS}}F^{u'}(\VaR^u) \ge \alpha,
     \end{equation}
    where the first equality is guaranteed by the optimality of deterministic stationary policies in probabilistic maximization MDP $\hat{\mathcal M}_3(\VaR^u)$. By the definition of $\lambda^*$, we directly have $\lambda^* \le \VaR^u$.

    Next, we prove (\ref{equ:opt_equ_min_s}). The definitions of $\lambda^*$ and $\hat u^*$ imply that
    \begin{equation}
        \nonumber
    F^{\hat u^*}(\lambda^*)=\max\limits_{u \in \mathcal U^{\rm DS}}F^{u}(\lambda^*) \ge \alpha,
    \end{equation}
 which directly leads to $\lambda^* \ge \VaR^{\hat u^*}$ by combining with the definition of $\VaR^{\hat u^*}$.
     \end{proof}

      Different from Theorem~\ref{thm:equiva_s}, the optimal policy for steady-state VaR minimization MDP is derived by solving a probabilistic maximization MDP $\hat{\mathcal M}_3(\VaR^*)$ with target level $\VaR^*$ instead of its left predecessor. Theorem~\ref{thm:equiva_min_s} indicates that the minimum $\VaR^*$
        can be attained by a deterministic stationary policy. In what follows, we restrict to deterministic stationary policy space $\mathcal U^{\rm DS}$.
     In order to efficiently tackle the bilevel optimization problem (\ref{equ:opt_value_min_s}), we derive the following theorem, which serves as a foundation for policy improvement and optimal policy identification by solving the inner problem
     $\hat{\mathcal M}_3(\lambda)$.   

     \begin{thm} \label{thm:imp_mins}
     {\rm(\textbf{Policy Improvement and Optimality Rules for $\mathcal
            M_3$})}

     Given a policy $u \in \mathcal U^{{\rm DS}}$, and let $\VaR^u_{-}(\Lambda)$ be the left predecessor of $\VaR^u$ in $\Lambda$,
     \begin{itemize}
        \item[(a)] if there exists another policy $u' \in \mathcal U^{\rm {DS}}$
        such that $ F^{u'}(\VaR^u_{-}(\Lambda)) \ge\alpha$, then $\VaR^{u'}<\VaR^u$,
        i.e., $u'$ is a strictly improved policy;

        \item[(b)] if and only if $ F^*(\VaR^u_{-}(\Lambda)) < \alpha$, the policy $u$ is optimal for $\mathcal
        M_3$.
     \end{itemize}
     \end{thm}

    \begin{proof}
    (a) The result follows directly  from
    (\ref{equ:relation_steady_min}) by letting $\lambda=\VaR^u$.

    (b) We first prove the necessity. Suppose $u$ is an optimal policy
    of the steady-state VaR minimization MDP $\mathcal M_3$ in
    (\ref{equ:svarmdp_min}), then we have
    \begin{equation}\nonumber
        \VaR^{u} \le \VaR^{u'}, \qquad \forall u' \in \mathcal U^{\rm DS}.
    \end{equation}
    According to (\ref{equ:relation_steady_min}), we obtain
    \begin{equation}\nonumber
        \mathbbm P(R^{u'}_{\infty} \le \VaR^u_{-}(\Lambda)) < \alpha  ,\quad\forall u'
        \in \mathcal U^{\rm DS}.
    \end{equation}
    Therefore, the maximum probability $\sup\limits_{u'\in \mathcal
        U^{\rm {DS}}} \mathbbm P(R^{u'}_{\infty} \le \VaR^u_{-}(\Lambda)) =
    F^*(\VaR^u_{-}(\Lambda)) < \alpha$.

    The sufficiency can be also proved by using the same argument.
\end{proof}

     According to Theorem~\ref{thm:imp_mins}, we can iteratively improve
     the policy by solving a probabilistic maximization MDP
     $\hat{\mathcal M}_3(\lambda)$, finally converging to an optimal
     policy. Thus, we can also develop an algorithm for solving the
     steady-state VaR minimization MDP by modifying several lines in
     Algorithm~\ref{alg:PI_steady} as follows.
     \begin{itemize}
 \item \textbf{Line 6:} Replace with ``Update target level: $\lambda_k = \VaR^{u^{(k)}}_{-}(\Lambda)$, where $\VaR^{u^{(k)}}_{-}(\Lambda)$ is the left predecessor of $\VaR^{u^{(k)}}$ in $\Lambda$".
     \item \textbf{Line 7:} Replace with ``Solve the probabilistic maximization MDP $\hat{\mathcal M}_3(\lambda_k)$ in
     (\ref{equ:sprobmdp_min})".
     \item \textbf{Line 8:} Substitute ``$\min\limits_{u\in \mathcal{U}^{\rm DS}}$" with ``$\max\limits_{u\in \mathcal{U}^{\rm DS}}$", and ``$\argmin\limits_{u\in \mathcal{U}^{\rm DS}}$" with ``$\argmax\limits_{u\in \mathcal{U}^{\rm
     DS}}$".
     \item \textbf{Line 9:} Replace with ``\textbf{While} $F^*(\lambda_k) \ge
     \alpha$".
     \end{itemize}

     Similarly, we can also prove that the above modified algorithm can
     converge to an optimal policy of the steady-state VaR minimization
     MDP $\mathcal M_3$ within finite iterations.

     \subsection{Finite-Horizon VaR Minimization MDPs} \label{sec:tmdp_min}
     We aim to minimize the VaR of accumulated costs over a finite
     horizon, i.e.,
     \begin{flalign}\label{equ:tvarmdp_min}
     &\mathcal M_4: \hspace{4cm} \VaR^*(s_0) := \inf\limits_{u \in
        \mathcal U^{\rm RH}}   \VaR^u(s_0), \qquad s_0 \in \mathcal{S}.&
        \end{flalign}
         Similarly, we can
        rewrite (\ref{equ:tvarmdp_min}) as
        \begin{equation*}
     \inf\limits_{u \in \mathcal U^{\rm RH}} \min\limits_{\lambda_0 \in
        \Lambda_0}\left\{\lambda_0: \mathbbm P^u_{s_0}(R_{0:T} \le
     \lambda_0)\ge\alpha\right\}, \quad
     s_0 \in \mathcal{S}.
     \end{equation*}
     According to Lemma \ref{lem_VaRmin}, for a policy $u\in \mathcal U^{\rm RH}$, let $\VaR^u_-(s_0, \Lambda_0)$ be the left predecessor of $\VaR^u(s_0)$ in $\Lambda_0$, we focus on maximizing $\mathbbm P^{u'}_{s_0}(R_{0:T} \le \VaR^u_-(s_0, \Lambda_0))$
     such that the VaR value of the new policy $u'$ can be reduced. Thus, we formulate the finite-horizon probabilistic maximization
     MDP with target value $\lambda_0 \in \mathbb R$ as
     \begin{flalign}\label{equ:promdp_max}
     &\hat{\mathcal M}_4(\lambda_0): \hspace{4cm} F^*(s_0, \lambda_0) :=
     \sup\limits_{u\in \mathcal U^{\rm {RH}}}F^u(s_0, \lambda_0), \qquad
     s_0 \in \mathcal{S}.&
     \end{flalign}
     Similarly, $\hat{\mathcal
     M}_4(\lambda_0)$ is not a standard MDP and it can be equivalently
     solved by an augmented MDP $\tilde{\mathcal M}_4$ with tuple
     $\langle
     \tilde{\mathcal{S}},\tilde{\mathcal{A}},(\tilde{\mathcal{A}}(\tilde{s}),
     \tilde{s} \in \tilde{\mathcal{S}}),\tilde P,\tilde r\rangle$
     maximizing the expectation of the finite-horizon accumulated
     rewards, i.e.,
     \begin{flalign}\label{equ:newmdp-min}
     &\tilde{\mathcal M}_4: \hspace{4cm} V_0^*(s_0,\lambda_0) =
     \sup\limits_{\tilde{u} \in \tilde{\mathcal U}^{\rm RH}}
     V_0^{\tilde{u}}(s_0,\lambda_0), \qquad (s_0,\lambda_0) \in
     \tilde{\mathcal S}.&
     \end{flalign}
     Thus, we can use the classical backward dynamic programming to solve
     $\tilde{\mathcal M}_4$ and also prove that $V_0^*(s_0,\lambda_0) =
     F^*(s_0, \lambda_0)$. The following theorem establishes the
     equivalence between the finite-horizon VaR minimization MDP
     $\mathcal M_4$ and the augmented MDP $\tilde{\mathcal{M}}_4$.

     \begin{thm}\label{thm:equiva_min}
     {\rm (\textbf{Optimization Equivalence between $\mathcal M_4$ and
            $\tilde{\mathcal M}_4$})}

     \begin{itemize}
        \item[(a)] The optimal value function of the finite-horizon VaR minimization MDP $\mathcal M_4$ in
        (\ref{equ:tvarmdp_min})~ is given by
        \begin{equation}\label{equ:opt_value2}
            \VaR^*(s_0) = \min\limits_{\lambda_0 \in \Lambda_0}\left\{\lambda_0:
            \max\limits_{\tilde{u} \in \tilde{\mathcal U}^{\rm DM}}
            V_0^{\tilde{u}}(s_0,\lambda_0) \ge \alpha\right\}, \qquad s_0 \in
            \mathcal S.
        \end{equation}

        \item[(b)]Suppose $\tilde{u}^*$ is an optimal deterministic Markov policy of the augmented MDP $\tilde{\mathcal M_4}$ in (\ref{equ:newmdp-min}),
        then there exists an optimal deterministic history-dependent policy
        $u^*=(u^*_{0}, \ldots, u^*_{T-1}) \in \mathcal U^{\rm {DH}}$ for
        $\mathcal M_4$, given by
        \begin{equation}\label{equ:optimal_policy2}
            u^*_{t}(s_0,a_0,\ldots,s_t) = \tilde{u}_t^*(s_t,
            \VaR^*(s_0)-\sum\limits_{\tau=0}^{t-1} r(s_\tau,a_\tau)), \qquad
            t=0,\ldots,T-1.
        \end{equation}
     \end{itemize}
     \end{thm}

     Similarly, to efficiently solve the bilevel optimization problem
     (\ref{equ:opt_value2}), we derive the following theorem to improve
     policy and find optima.

     \begin{thm} \label{thm:imp_min}
     {\rm (\textbf{Policy Improvement and Optimality Rules for $\mathcal
            M_4$})}

     Given a policy $u \in \mathcal U^{\rm DH}$,  and let $\VaR^u_{-}(s_0,\Lambda_0)$ be the left predecessor of $\VaR^u(s_0)$ in $\Lambda_0$,
     \begin{itemize}
        \item[(a)] if there exists another policy $u' \in \mathcal U^{\rm DH}$ such that $F^{u'}(s_0,\VaR^u_{-}(s_0,\Lambda_0)) \ge \alpha$, then $\VaR^{u'}(s_0)<\VaR^u(s_0)$, $s_0 \in \mathcal S$;

        \item[(b)] if and only if $F^{*}(s_0,\VaR^u_{-}(s_0,\Lambda_0)) < \alpha$ for any $s_0 \in \mathcal
        S$, the policy $u$ is optimal for $\mathcal M_4$.
     \end{itemize}
     \end{thm}

     Similar to the approach presented in Section~\ref{sec:tvarmdp}, we
     can develop a policy iteration type algorithm to solve the finite-horizon VaR
     minimization MDP, leveraging the theoretical foundations established
     in Theorems~\ref{thm:equiva_min} and \ref{thm:imp_min}. The
     structure of the algorithm is similar to
     Algorithm~\ref{alg:PI_finte_1}, except the modifications of several
     lines as follows.
     \begin{itemize}
     \item \textbf{Line 6:} Replace with ``Update target
level: $\lambda_0^{(k)} = \VaR^{u^{(k)}}_{-}(s_0,\Lambda_0)$, where
$\VaR^{u^{(k)}}_{-}(s_0,\Lambda_0)$ is the left predecessor of
$\VaR^{u^{(k)}}(s_0)$ in $\Lambda_0$".
     \item \textbf{Line 7:} Replace with ``Solve the augmented MDP $\tilde{\mathcal
     M}_4$".
     \item \textbf{Line 8} and \textbf{9:} Substitute ``$\min\limits_{\tilde{u}\in \tilde{\mathcal{U}}^{\rm DM}}$" with ``$\max\limits_{\tilde{u}\in \tilde{\mathcal{U}}^{\rm DM}}$", and ``$\argmin\limits_{\tilde{u}\in \tilde{\mathcal{U}}^{\rm DM}}$" with ``$\argmax\limits_{\tilde{u}\in \tilde{\mathcal{U}}^{\rm
     DM}}$".
     \item \textbf{Line 12:} Replace with ``\textbf{While} $F^*(s_0, \lambda_0^{(k)}) \ge
     \alpha$".
     \end{itemize}

     With Sections~\ref{sec:svarmdp}-\ref{sec:exten}, we have explored
     the optimization of VaR MDPs under different objectives, including
     maximization and minimization, steady-state and finite-horizon. For
     any current policy of VaR MDPs, we can identify an improved policy
     by solving a probabilistic MDP, which is guaranteed by the duality
     between VaR and probabilities. Notably, the maximization and
     minimization of VaR MDPs rely on different theoretical foundations,
     i.e., Lemmas~\ref{lem_VaRmax} and \ref{lem_VaRmin} reflect their
     distinct nature. Despite this difference, all these VaR MDPs can be
     reformulated as a unified form of bilevel optimization and share a
     common framework of policy iteration type algorithms to efficiently
     find optimal policies. This unified perspective may provide a
     foundation for extending the optimization framework to other bilevel
     problems, highlighting the versatility of our proposed approach.

\section{Numerical Experiments}\label{sec:numer}
In this section, we present numerical experiments to validate the
theoretical results and algorithms developed in
Sections~\ref{sec:svarmdp}-\ref{sec:exten}.
Subsection~\ref{sec:numeri_svar} validates the optimality and
computational efficiency of our algorithm in steady-state VaR MDPs.
Subsection~\ref{sec:numeri_tvar} gives a performance comparison
between our policy iteration algorithm and the value iteration
algorithm by \cite{li2022quantile} in finite-horizon VaR MDPs.
Subsection~\ref{sec:microgrid} demonstrates the applicability of our
approach through an example of renewable energy management in
microgrid systems. All experiments are conducted on a computer with
Python v3.7, an AMD Ryzen 3995WX processor (64 cores), and 256 GB of
memory.

\subsection{Steady-State VaR MDPs}\label{sec:numeri_svar}
In this subsection, we conduct numerical examples to verify the
optimality and computational efficiency of our policy iteration
algorithms for both maximizing and minimizing the VaR of
steady-state rewards in infinite-horizon MDPs. The experiments are
conducted with different problem sizes that $(\lvert \mathcal S
\rvert,\lvert \mathcal A \rvert) \in
\left\{(100,100),(100,1000),(1000,100),(1000,1000)\right\}$, where
the transition probabilities are randomly generated and the
instantaneous rewards are uniformly sampled from range $(0,100)$.
The probability level is set as $\alpha = 0.1$.

To verify the optimality and efficiency of
Algorithm~\ref{alg:PI_steady}, we compare its performance against a
baseline method that solves $\mathcal M_1$ by naively enumerating
$|\Lambda|$ individual probabilistic minimization MDPs, as
established in Theorem~\ref{thm:equiva_s}. The comparison results
are summarized in Table~\ref{tab:infty}, where $T_{\mathrm{iter}}$
and $\mathrm{VaR}^*_{\mathrm{iter}}$ represent the computation time
(in seconds) and the VaR value achieved by
Algorithm~\ref{alg:PI_steady}, respectively. The counterpart values
of the baseline method are denoted with subscript ``base" as
$T_{\mathrm{base}}$ and $\mathrm{VaR}^*_{\mathrm{base}}$,
respectively.

    \begin{table}[ht]
        \centering
            \begin{tabular}{|c|c|c|c|c|c|c|}
                \hline
                Scenario & $|\mathcal S|$ &  $|\mathcal A|$  & $T_{\mathrm{iter}}$ & $T_{\mathrm{base}}$ & $\VaR^*_{\mathrm{iter}}$ & $\VaR^*_{\mathrm{base}}$ \\ \hline
                1  & $100$    & $100$     & $12$  & $1,\!573$ & $97.399,\!87$ & $97.399,\!87$ \\ \hline
                2  & $100$     & $1000$ &  $100$ & $83,\!083$ & $99.798,\!07$ & $99.798,\!07$ \\ \hline
                3  & $1000$    & $100$    & $133$ & $171,\!044$ & $97.547,\!76$ & $97.547,\!76$ \\ \hline
                4  & $1000$   & $1000$    & $1,\!349$ & $8,\!895,\!074^{\dagger}$ & $99.783,\!81$ & \textemdash \\ \hline
        \end{tabular}
        \caption{Comparison of Algorithm~\ref{alg:PI_steady} and the baseline method in solving $\mathcal M_1$.}
        \label{tab:infty}
        \raggedright{  \small $^{\dagger}$Estimated by solving $1000$ probabilistic minimization MDPs and extrapolating to $|\Lambda|=951,\!511$ probabilistic minimization MDPs. The value of $\VaR^*_{\mathrm{base}}$ is not yet available.}
    \end{table}

     As shown in Table~\ref{tab:infty}, Algorithm~\ref{alg:PI_steady}
        achieves the same maximum VaR values as the baseline method, thereby
        confirming the convergence of global optimality. Moreover,
        Algorithm~\ref{alg:PI_steady} demonstrates a significant advantage
        of computational efficiency, outperforming the baseline by orders of magnitude in all cases.
        While the baseline method suffers from severe scalability issues, our algorithm maintains much more tractable computation time across scenarios.
        In Scenario~1, Algorithm~\ref{alg:PI_steady} finishes in just 12 seconds, over 100 times faster than the baseline. As the problem size grows, this efficiency advantage becomes even more significant.
        Notably, in the large-scale case $(|\mathcal S|,|\mathcal A|)=(1000,1000)$,  Algorithm~\ref{alg:PI_steady} finishes in $1,349$ seconds (approximately $22$ minutes), whereas the baseline is expected to cost over $8.8$ million seconds (approximately $103$ days), achieving over 6,500-fold acceleration.
        These results highlight the superior efficiency and scalability of Algorithm~\ref{alg:PI_steady}, especially for large-scale MDPs where the baseline method becomes computationally infeasible.
        The observations in Table~\ref{tab:infty} also empirically
        indicate that the computation time of
        Algorithm~\ref{alg:PI_steady} may grow at a polynomial rate, despite the exponential growth of the policy space
        size $|\mathcal A|^{|\mathcal S|}$. This feature resembles the computational efficiency of the policy iteration in the classical MDP theory.
    \begin{figure}[ht]
        \centering
        \subfigure[$(|\mathcal S|, |\mathcal A|) = (100,100)$]{
            \includegraphics[width=0.45\textwidth]{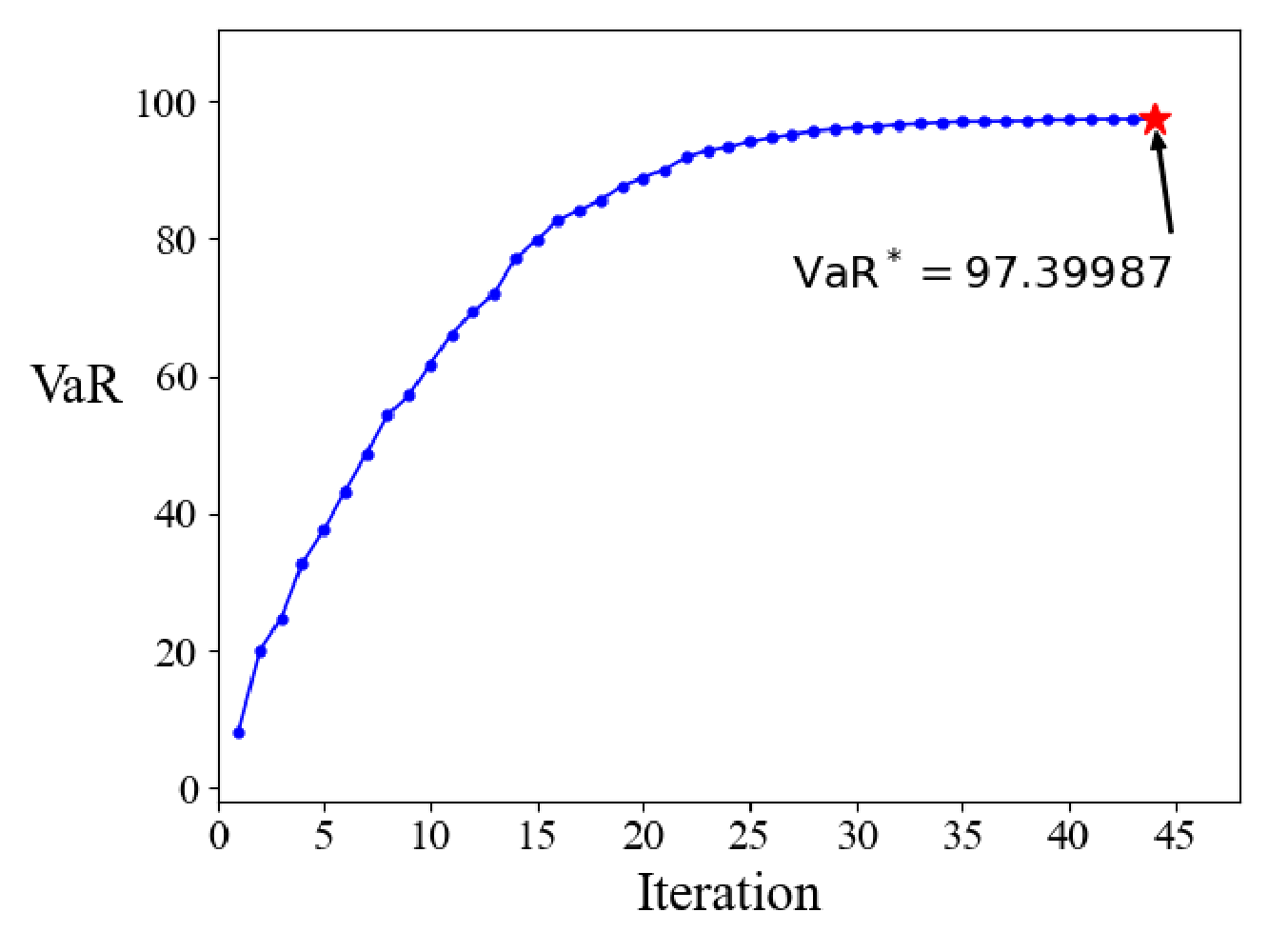}
            \label{fig:infty1}
        }\hfill
        \subfigure[$(|\mathcal S|, |\mathcal A|) = (100,1000)$]{
            \includegraphics[width=0.45\textwidth]{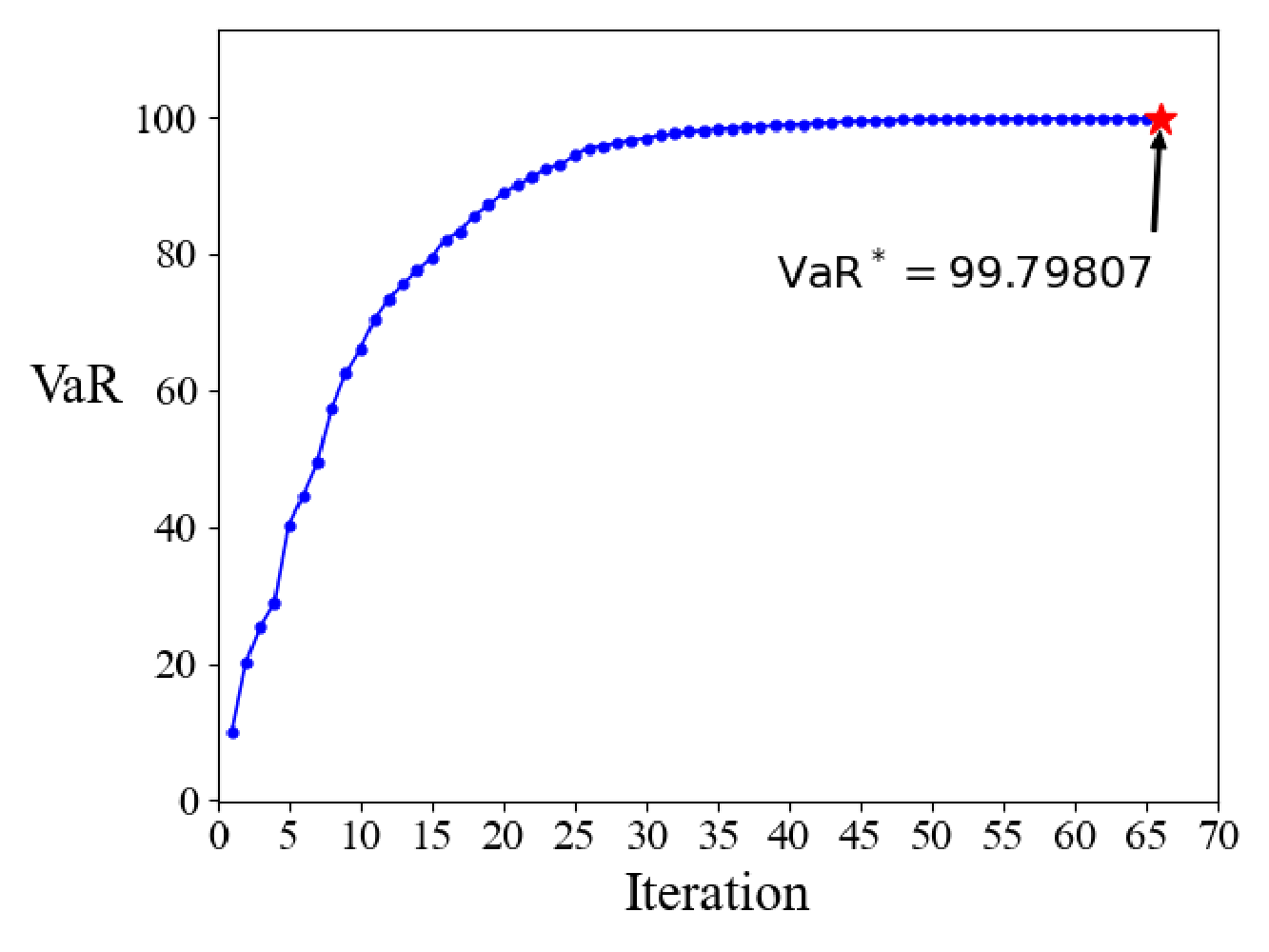}
            \label{fig:infty2}
        }\hfill
        \subfigure[$(|\mathcal S|, |\mathcal A|) = (1000,100)$]{
            \includegraphics[width=0.45\textwidth]{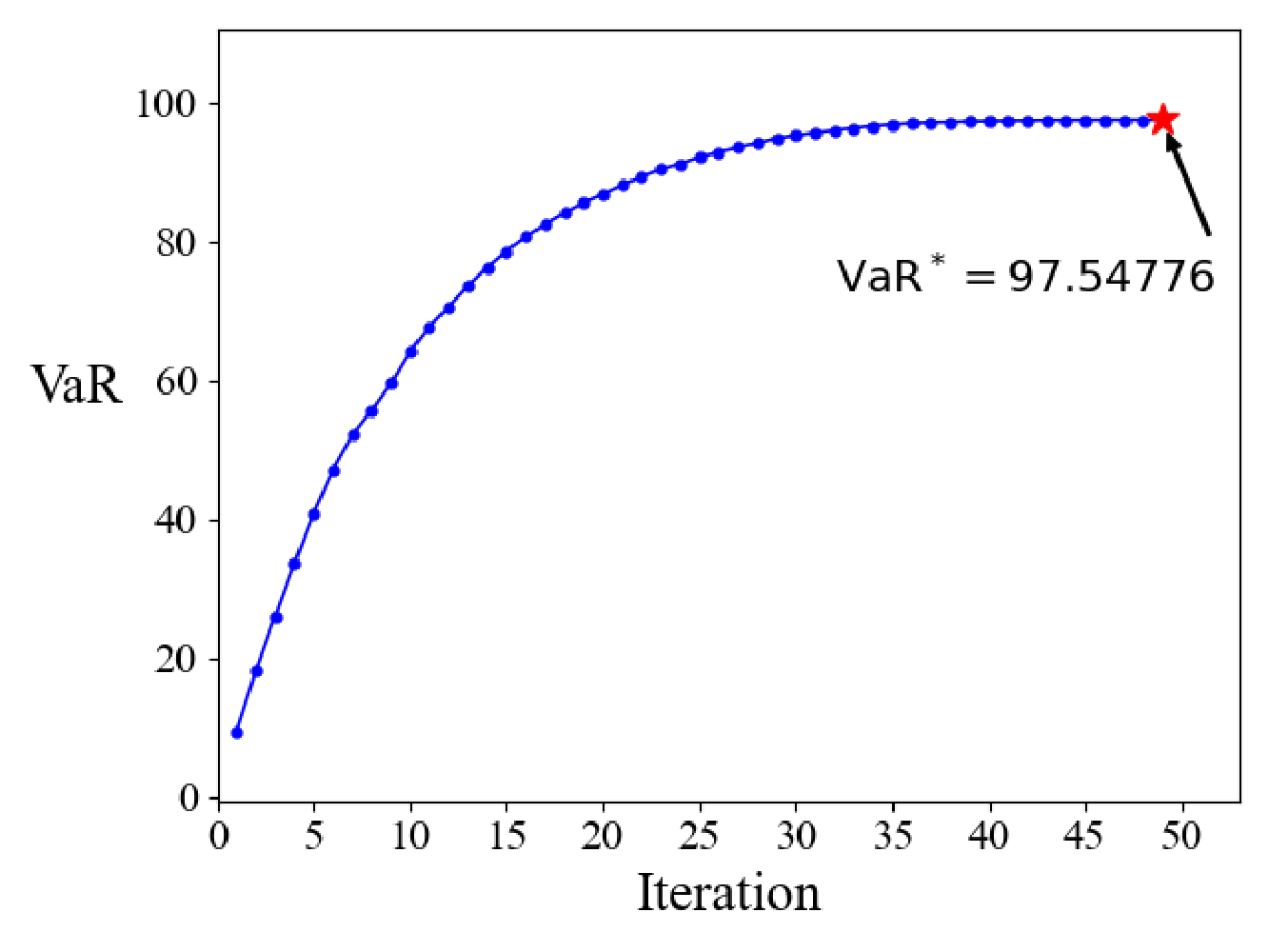}
            \label{fig:infty3}
        }\hfill
        \subfigure[$(|\mathcal S|, |\mathcal A|) = (1000,1000)$]{
            \includegraphics[width=0.45\textwidth]{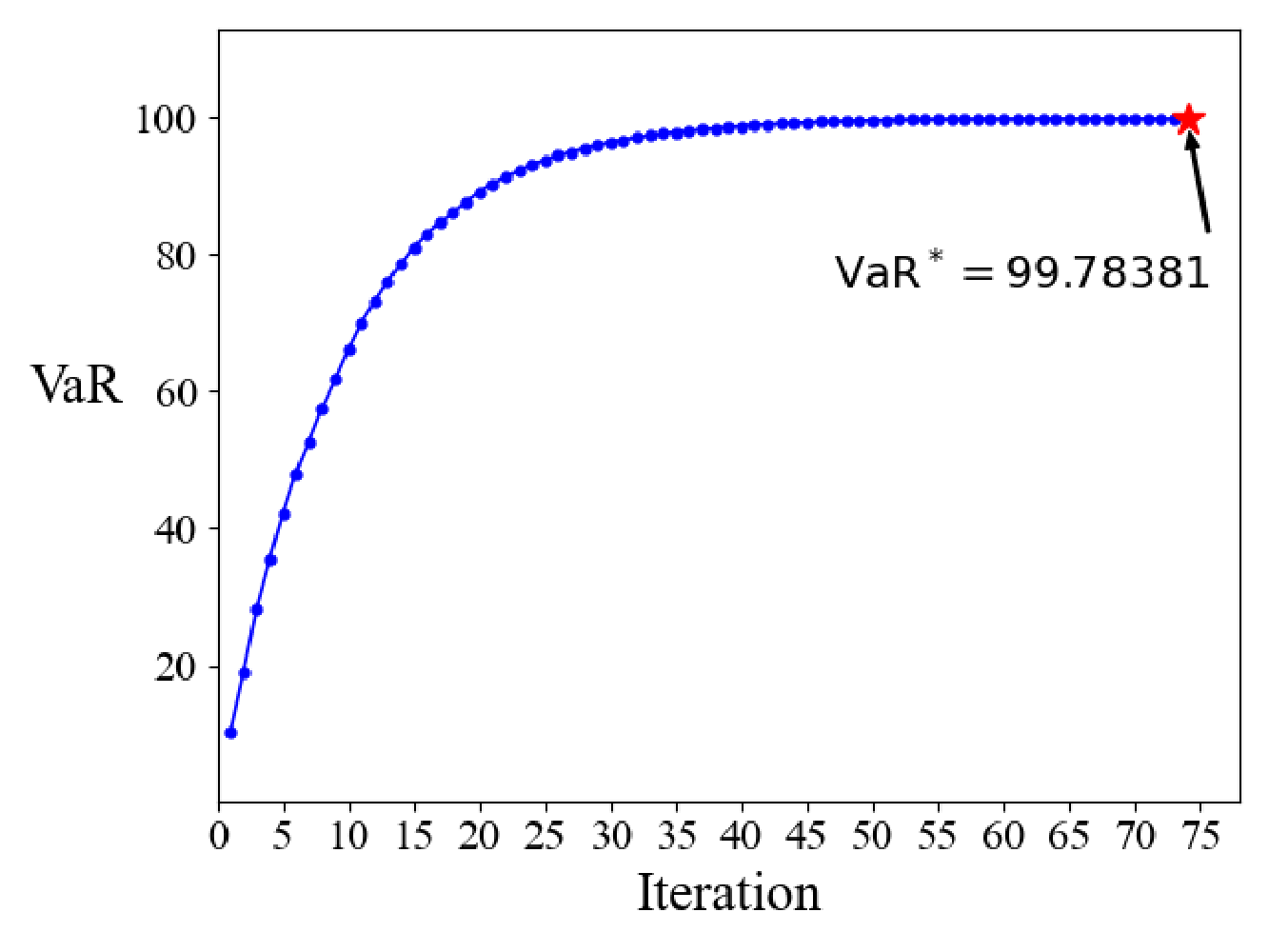}
            \label{fig:infty4}
        }
        \caption{The VaR improvement procedure of Algorithm~\ref{alg:PI_steady} in different scenarios.}
        \label{fig:infty}
    \end{figure}

    Figure~\ref{fig:infty} illustrates the policy iteration and
    improvement procedures of Algorithm~\ref{alg:PI_steady} in different
    problem sizes. We can observe that the algorithm strictly improves
    the VaR value at each iteration and ultimately converges to the
    global optimum. Interestingly, the number of iterations required
    remains relatively stable under different problem sizes.  However, as
        presented in Table~\ref{tab:infty}, the total computation time
        $T_{\mathrm{iter}}$ grows polynomially with problem size. This is because the computation time of each iteration, i.e., solving a probabilistic minimization MDP, increases with problem size.

    \begin{figure}[ht]
        \centering
        \subfigure[$(|\mathcal S|, |\mathcal A|) = (100,100)$]{
            \includegraphics[width=0.48\textwidth]{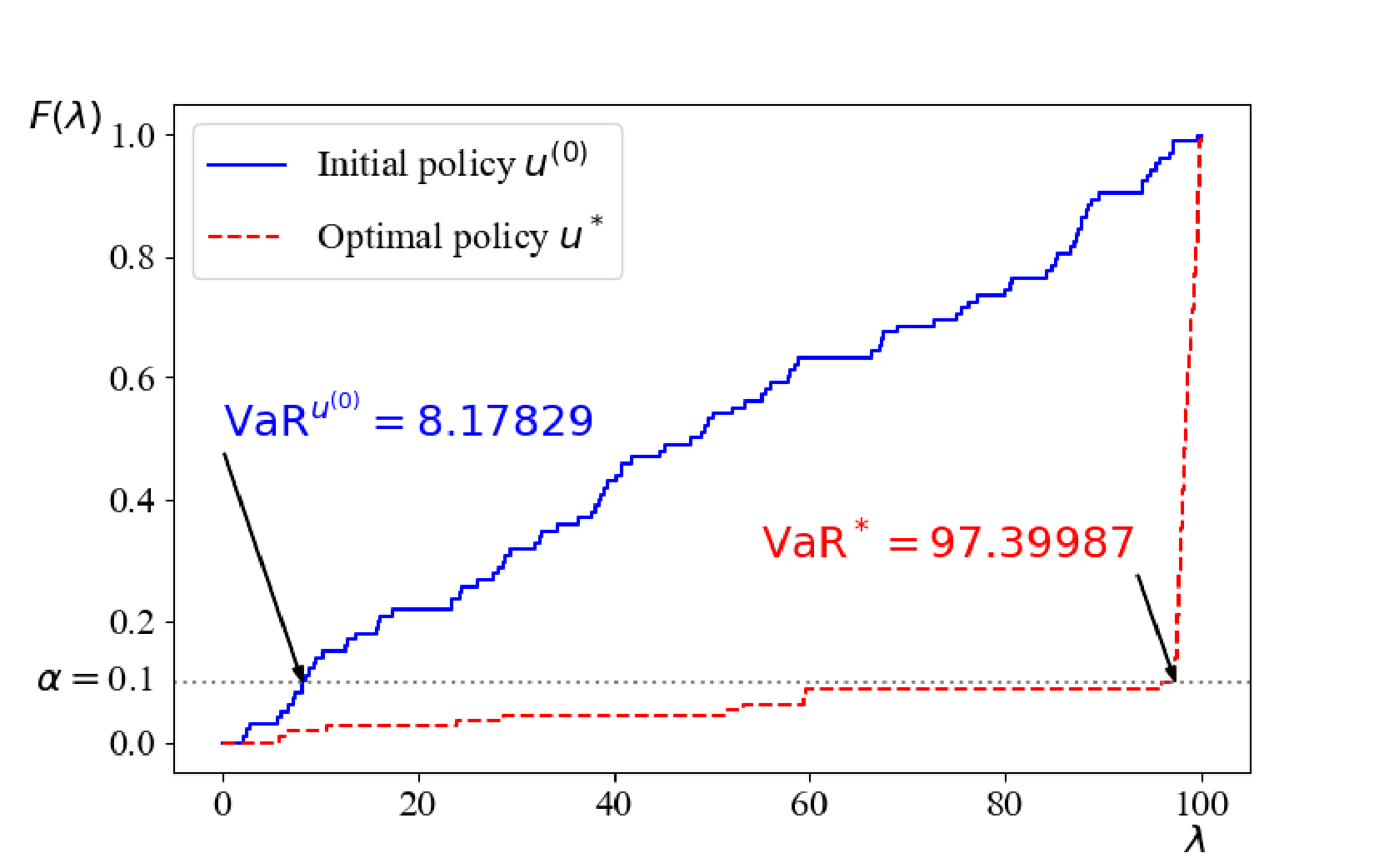}
            \label{fig:u1}
        }\hfill
        \subfigure[$(|\mathcal S|, |\mathcal A|) = (100,1000)$]{
            \includegraphics[width=0.48\textwidth]{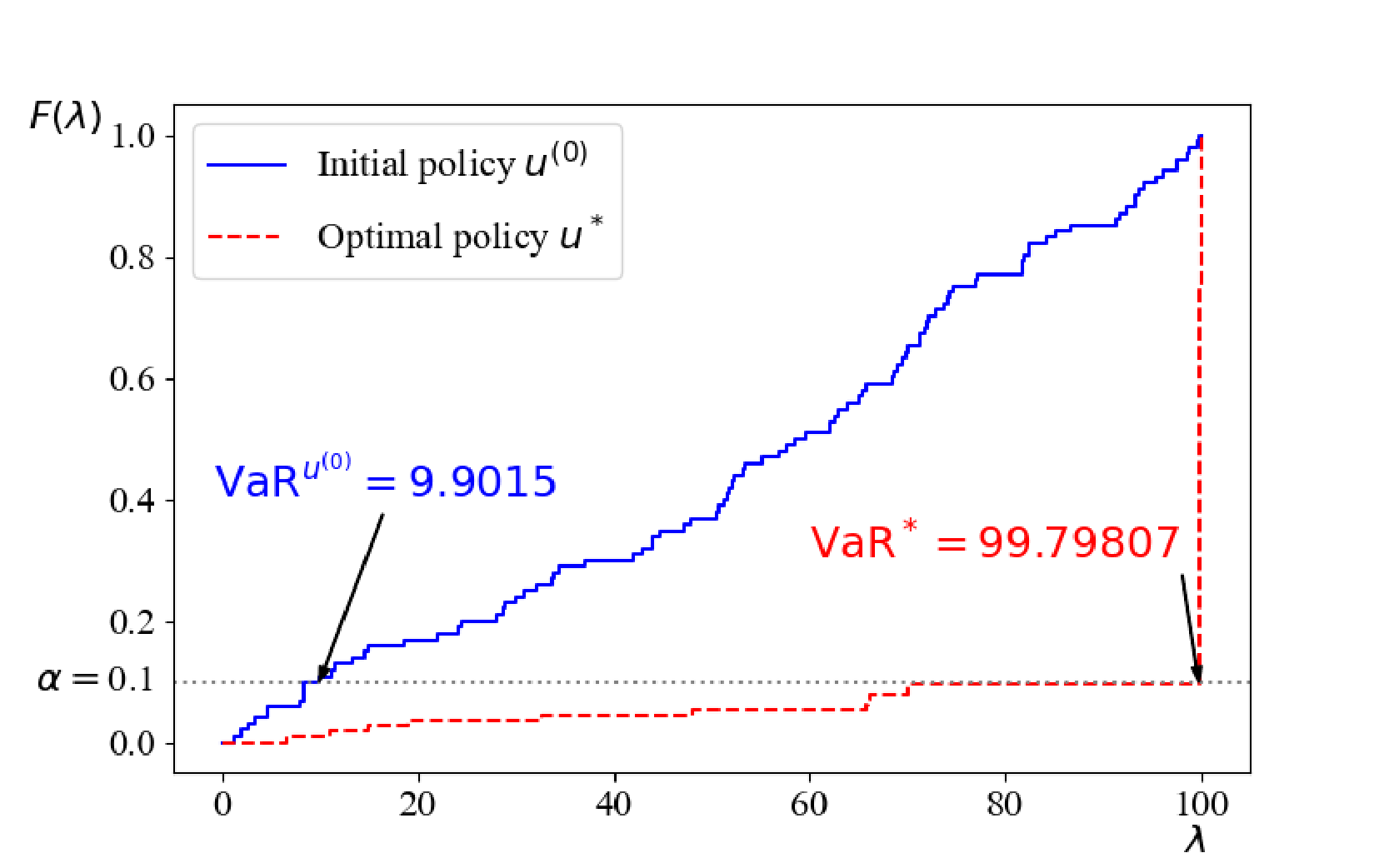}
            \label{fig:u2}
        }\hfill
        \subfigure[$(|\mathcal S|, |\mathcal A|) = (1000,100)$]{
            \includegraphics[width=0.48\textwidth]{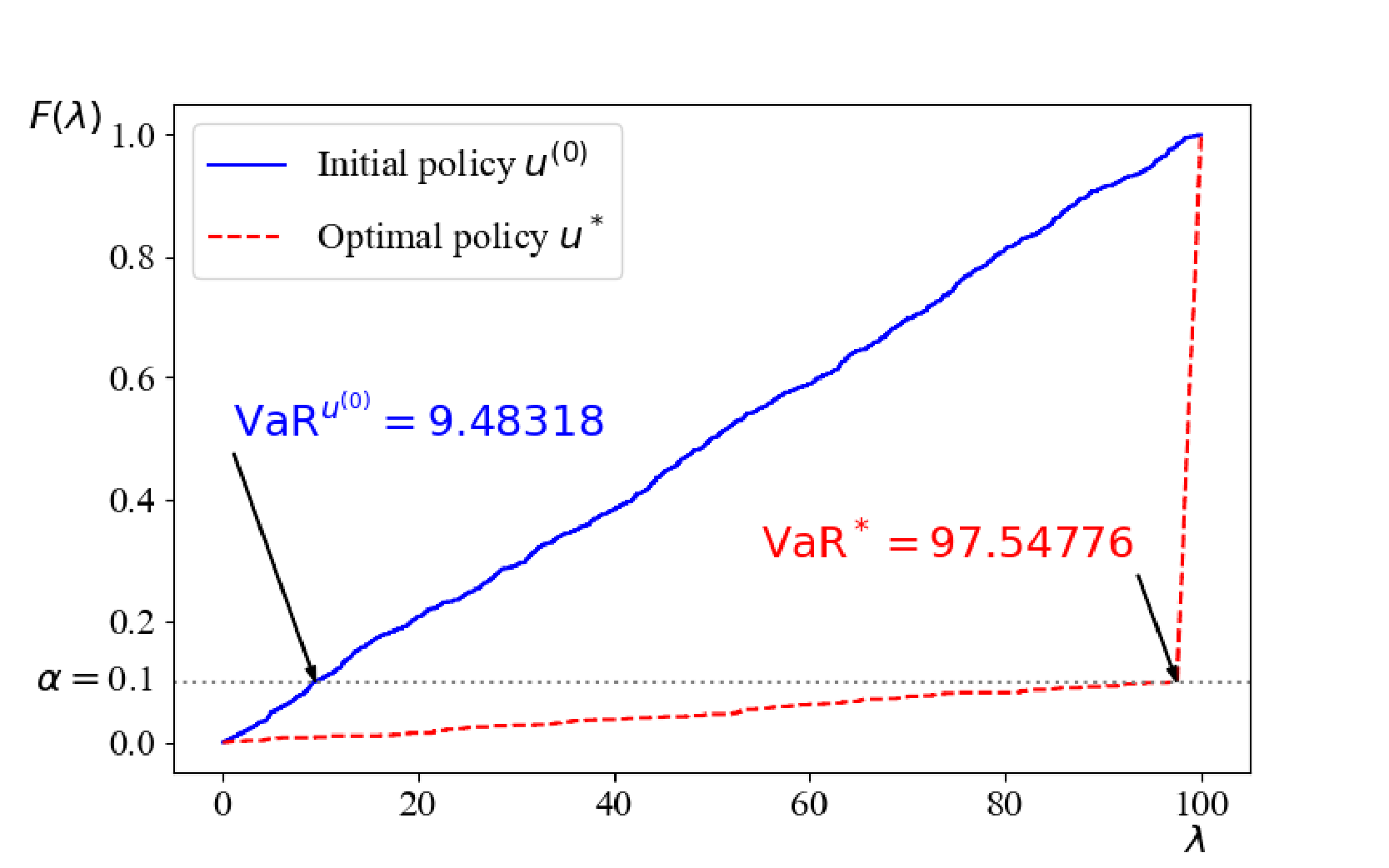}
            \label{fig:u3}
        }\hfill
        \subfigure[$(|\mathcal S|, |\mathcal A|) = (1000,1000)$]{
            \includegraphics[width=0.48\textwidth]{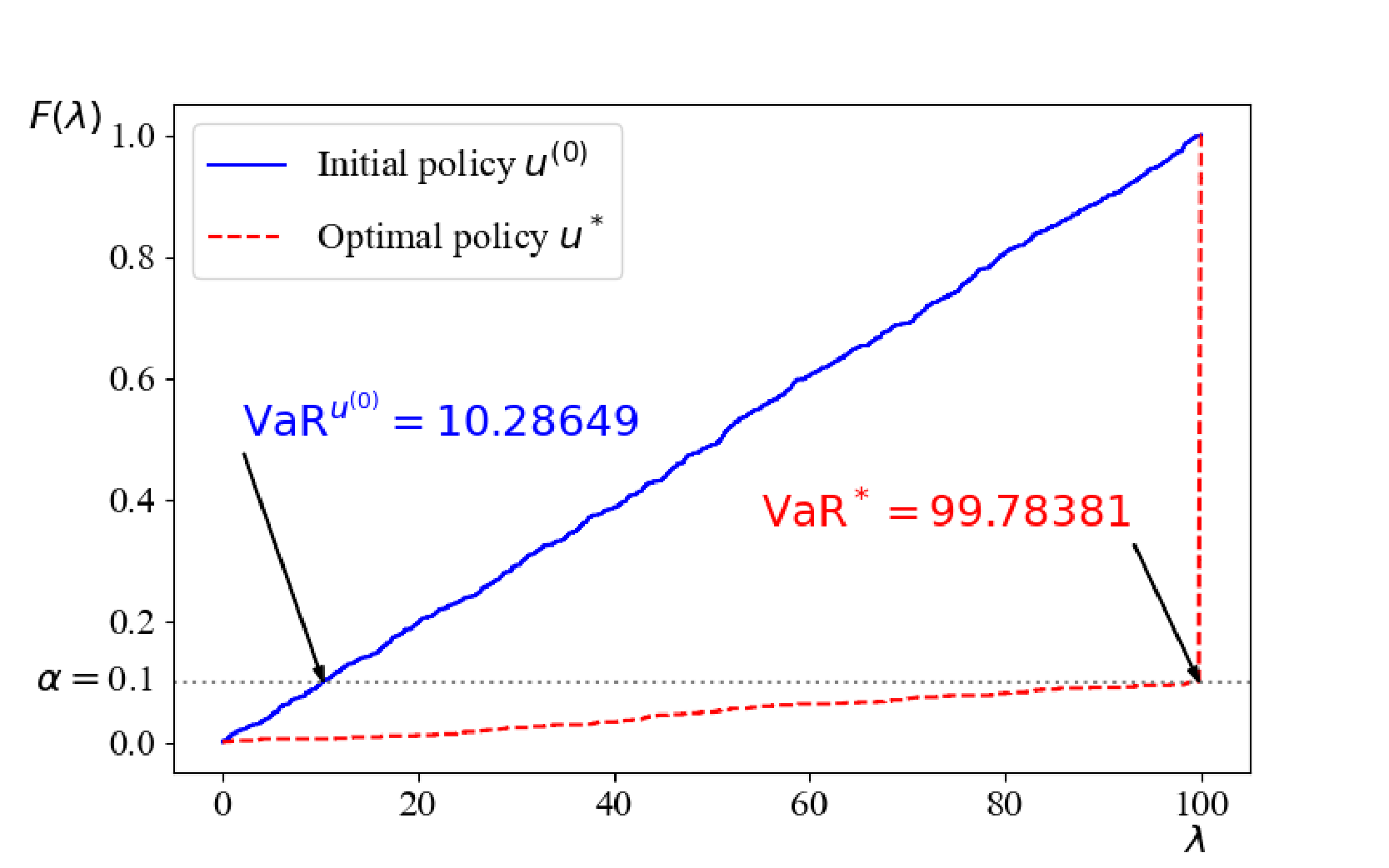}
            \label{fig:u4}
        }
        \caption{CDF comparison of the initial and optimal policies for $\mathcal M_1$ in different scenarios.}
        \label{fig:u_compare}
    \end{figure}
    Figure~\ref{fig:u_compare} further illustrates the cumulative
    distribution functions (CDFs) of the steady-state rewards under the
    initial policy and the optimal policy obtained by
    Algorithm~\ref{alg:PI_steady}, respectively. The initial policy
    $u^{(0)}$ is randomly selected at the beginning of
    Algorithm~\ref{alg:PI_steady}. As illustrated in
    Figure~\ref{fig:u_compare}, the CDF of $u^{(0)}$ exhibits a nearly
    uniform staircase pattern, indicating that the corresponding
    steady-state reward $R_{\infty}^{u^{(0)}}$ is approximately
    uniformly distributed. In contrast, the optimal policy $u^*$ yields
    a substantially improved distribution: It maximizes the VaR of
    $R_{\infty}^{u^*}$ by assigning negligible probability on lower
    reward values and concentrating over $90\%$ of the probability mass
    on the higher end of the support. This clearly demonstrates the
    algorithm's effectiveness in steering the system toward more
    favorable long-term outcomes.

    As a comparison, we also consider the problem of minimizing the VaR
    of steady-state costs in infinite-horizon MDPs, as studied in
    Subsection~\ref{sec:smdp_min}. We verify the optimality and
    efficiency of our proposed policy iteration algorithm for solving
    $\mathcal M_3$, by comparing its performance against the baseline
    method that solves $|\Lambda|$ individual probabilistic maximization
    MDPs based on Theorem~\ref{thm:equiva_min_s}. The experimental
    parameter settings are identical to those in the steady-state VaR
    maximization MDPs, except that the probability level is set as
    $\alpha = 0.9$. The numerical results are summarized in Table~\ref{tab:infty_min},
    Figures~\ref{fig:infty_min} and \ref{fig:u_m_compare}.

\begin{figure}[ht]
    \centering
    \subfigure[$(|\mathcal S|, |\mathcal A|) = (100,100)$]{
        \includegraphics[width=0.45\textwidth]{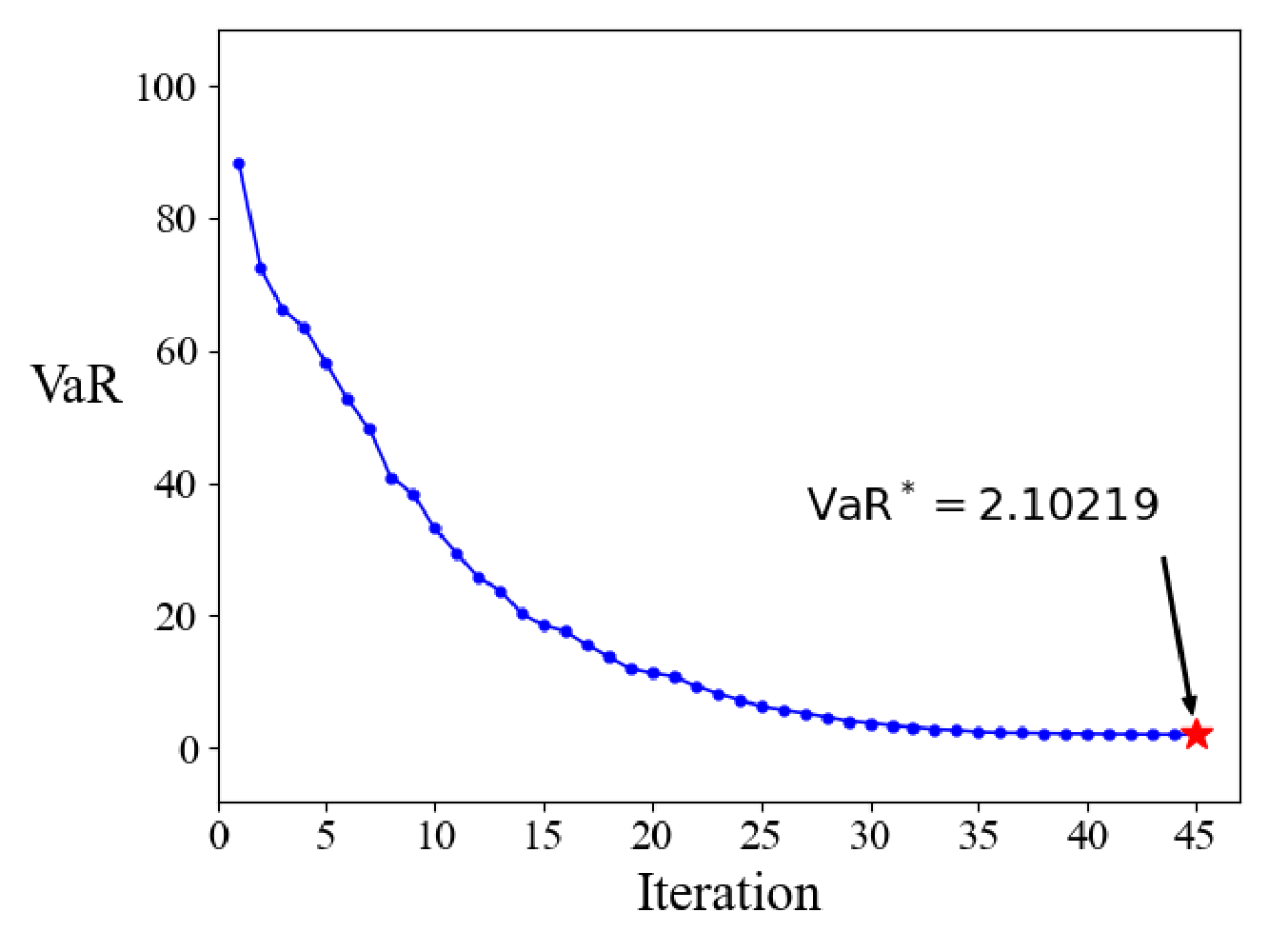}
        \label{fig:infty_m1}
    }\hfill
     \subfigure[$(|\mathcal S|, |\mathcal A|) = (100,1000)$]{
        \includegraphics[width=0.45\textwidth]{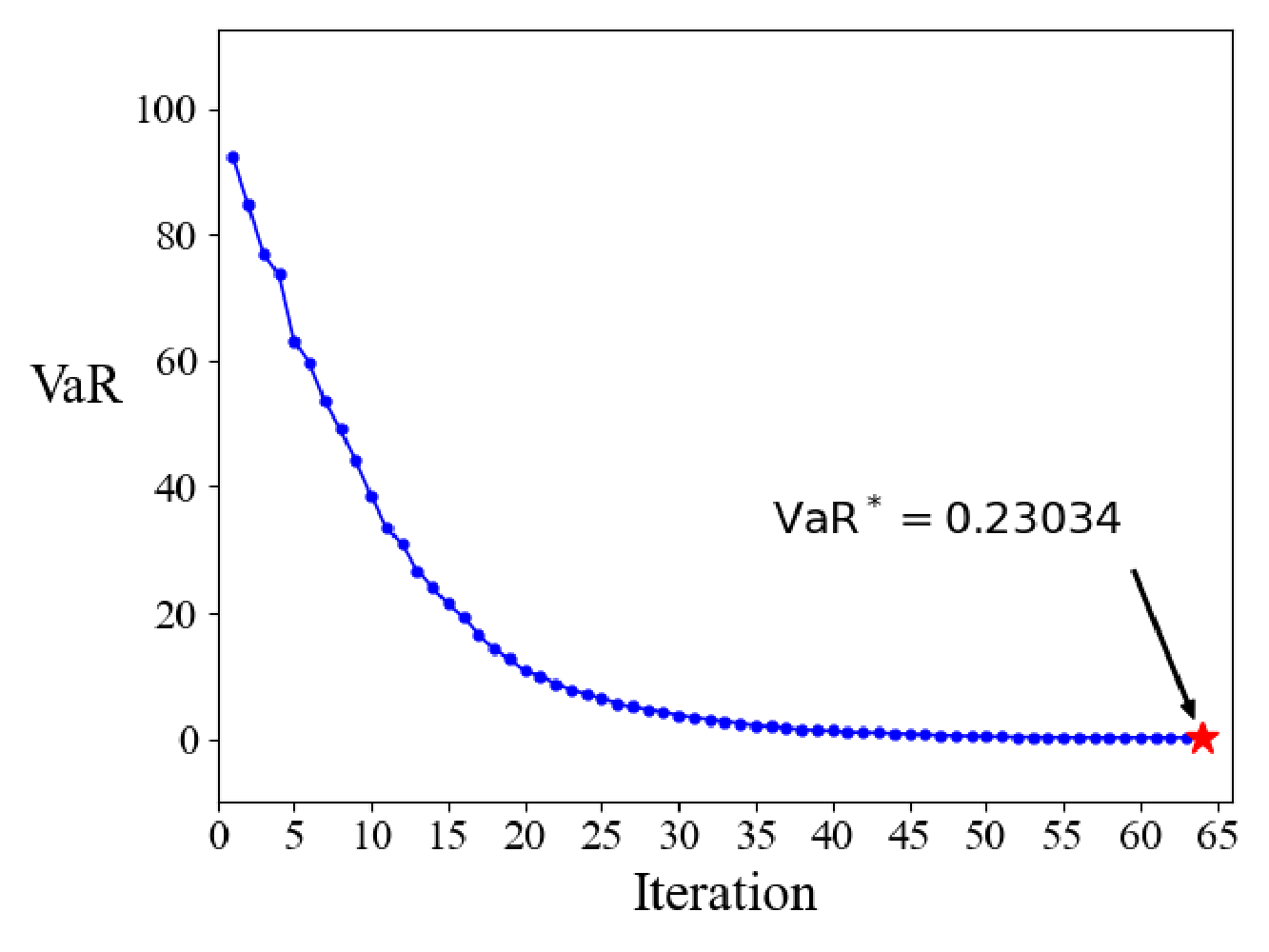}
        \label{fig:infty_m2}
    }\hfill
     \subfigure[$(|\mathcal S|, |\mathcal A|) = (1000,100)$]{
        \includegraphics[width=0.45\textwidth]{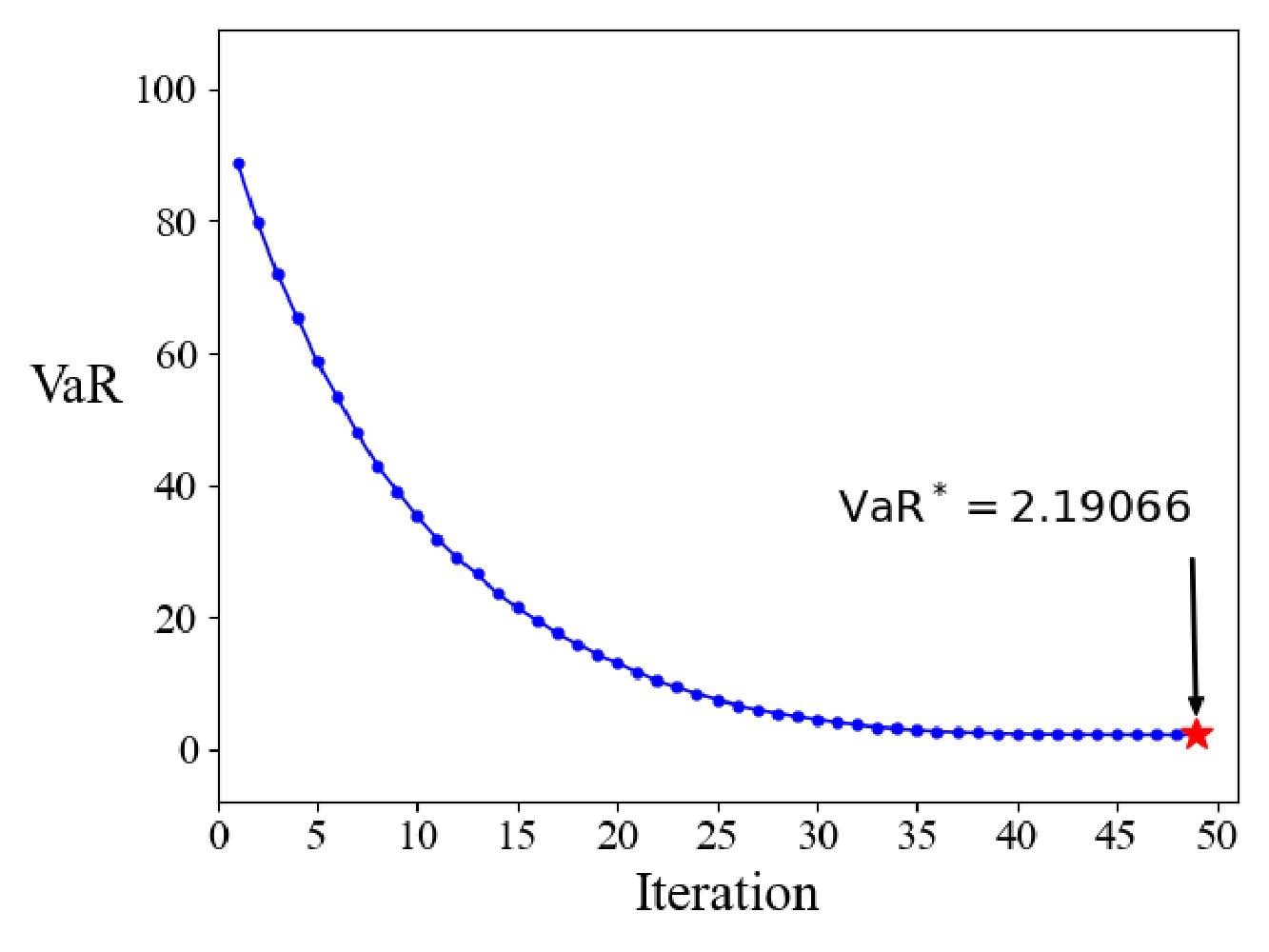}
        \label{fig:infty_m3}
    }\hfill
     \subfigure[$(|\mathcal S|, |\mathcal A|) = (1000,1000)$]{
        \includegraphics[width=0.45\textwidth]{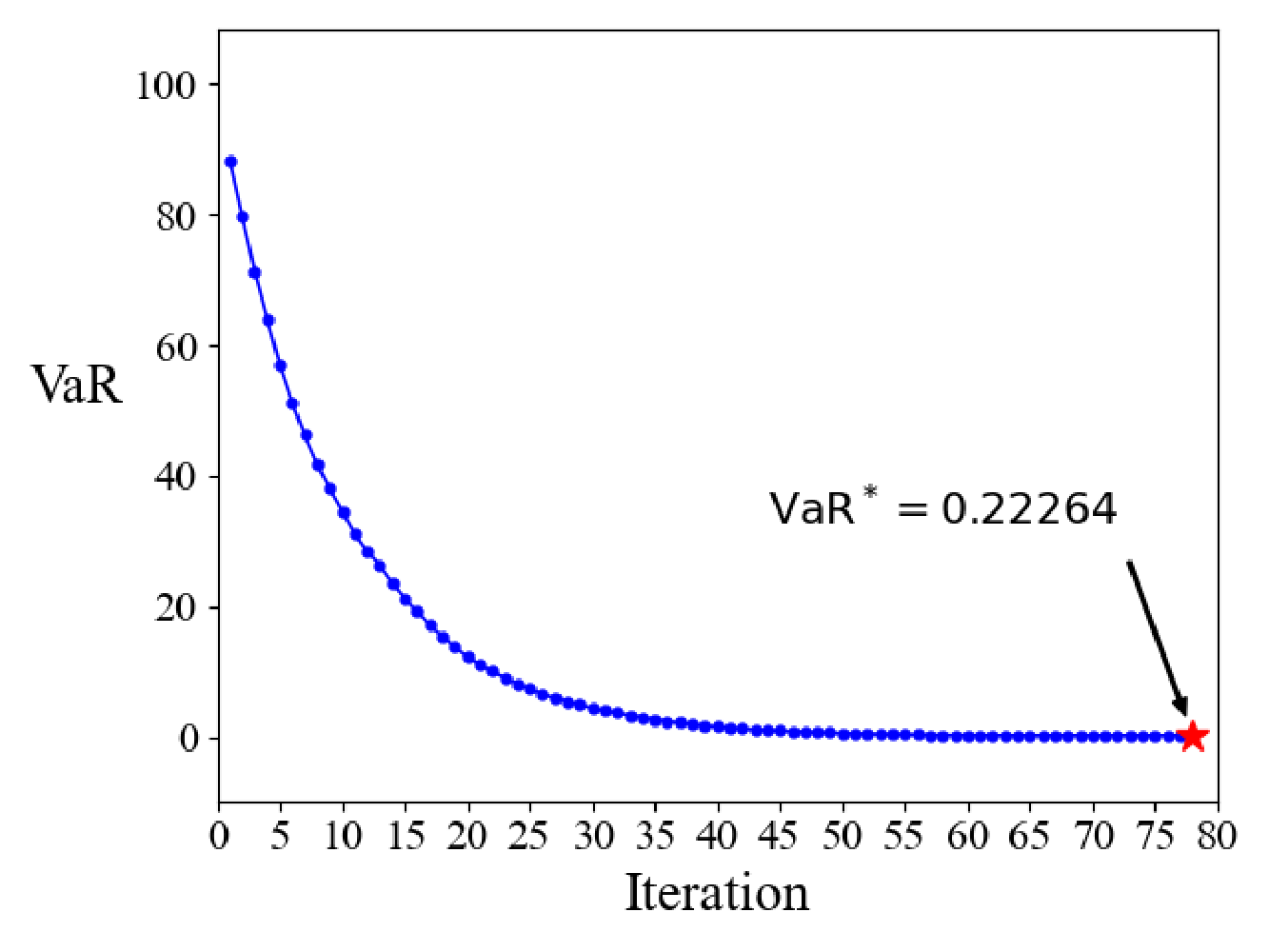}
        \label{fig:infty_m4}
    }
    \caption{The VaR improvement procedure of our algorithm solving $\mathcal M_3$ in different scenarios.}
    \label{fig:infty_min}
\end{figure}

   \begin{table}[htbp]
    \centering
    \begin{tabular}{|c|c|c|c|c|c|c|}
        \hline
        Scenario & $|\mathcal S|$ &  $|\mathcal A|$ & $T_{\mathrm{iter}}$ & $T_{\mathrm{base}}$ & $\VaR^*_{\mathrm{iter}}$ & $\VaR^*_{\mathrm{base}}$ \\ \hline
        1  & $100$ &$100$     & $11$ & $2,\!280$ & $2.102,\!19$ & $2.102,\!19$ \\ \hline
        2  & $100$ & $1000$    &   $90$ & $163,\!260$ & $0.230,\!34$ & $0.230,\!34$ \\ \hline
        3  & $1000$ & $100$     & $132$ & $349,\!631$& $2.190,\!66$ & $2.190,\!66$ \\ \hline
        4  & $1000$   & $1000$     & $1,\!498$ & $25,\!274,\!843^{\dagger}$& $0.222,\!64$ & \textemdash \\ \hline
    \end{tabular}
    \caption{Comparison of our algorithm and the baseline method in solving $\mathcal M_3$.}
    \label{tab:infty_min}
    \raggedright{ \small $^{\dagger}$ Estimated by solving $1000$ probabilistic maximization MDPs and extrapolating to $|\Lambda|=951,\!868$ probabilistic maximization MDPs. The value of $\VaR^*_{\mathrm{base}}$ is not yet available.}
\end{table}
      \begin{figure}[htbp]
    \centering
    \subfigure[$(|\mathcal S|, |\mathcal A|) = (100,100)$]{
        \includegraphics[width=0.48\textwidth]{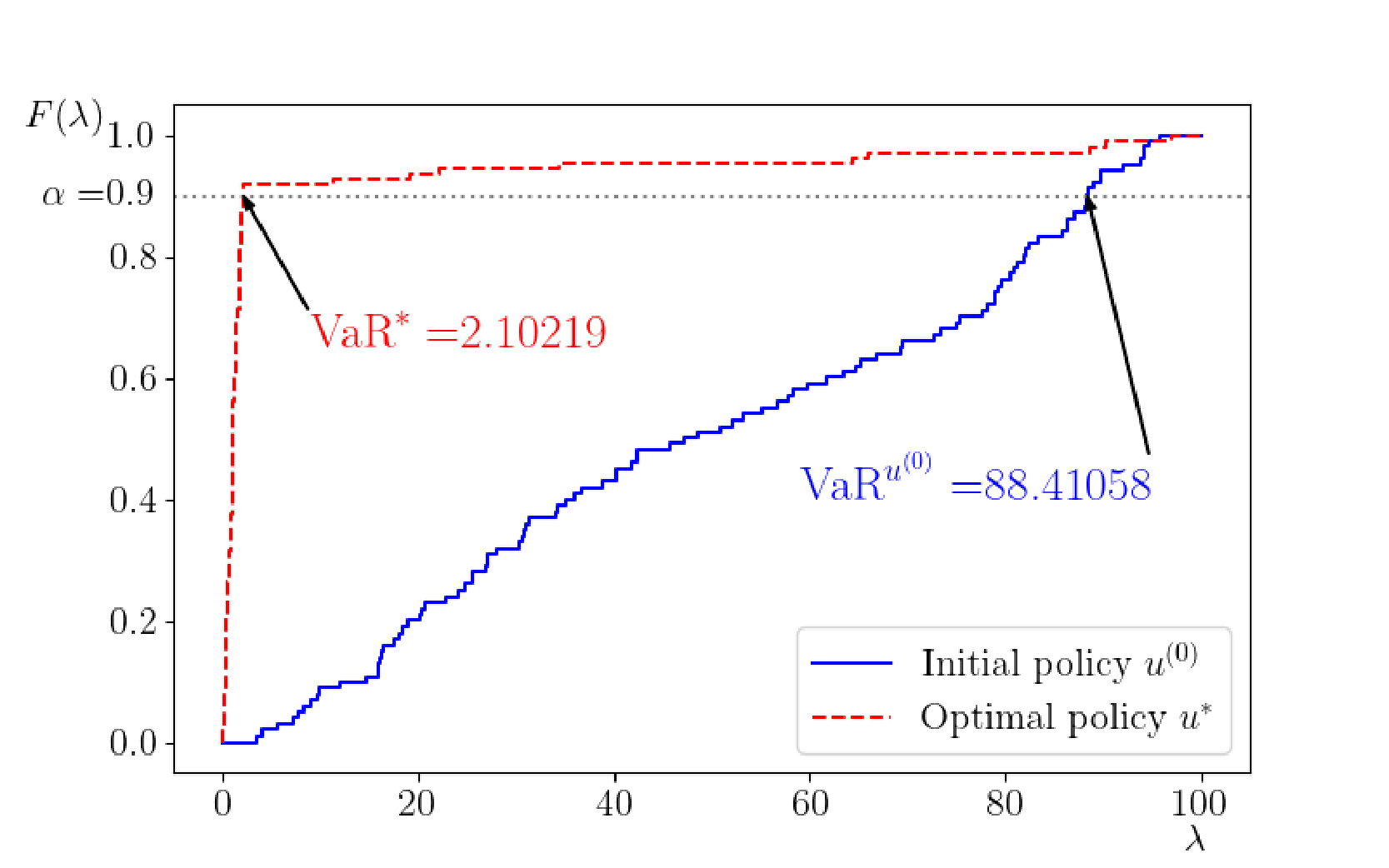}
        \label{fig:u_m1}
    }\hfill
    \subfigure[$(|\mathcal S|, |\mathcal A|) = (100,1000)$]{
        \includegraphics[width=0.48\textwidth]{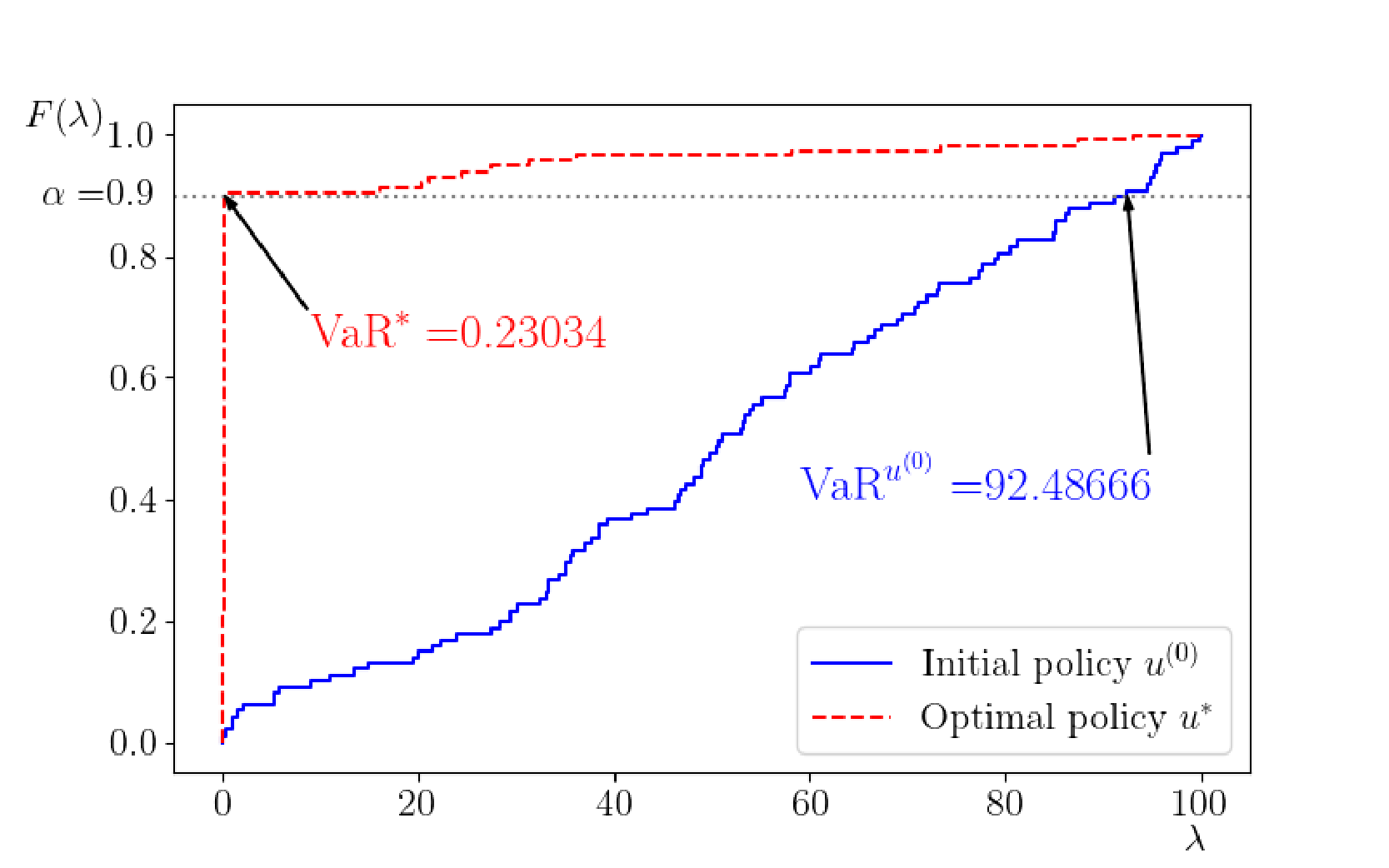}
        \label{fig:u_m2}
    }\hfill
    \subfigure[$(|\mathcal S|, |\mathcal A|) = (1000,100)$]{
        \includegraphics[width=0.48\textwidth]{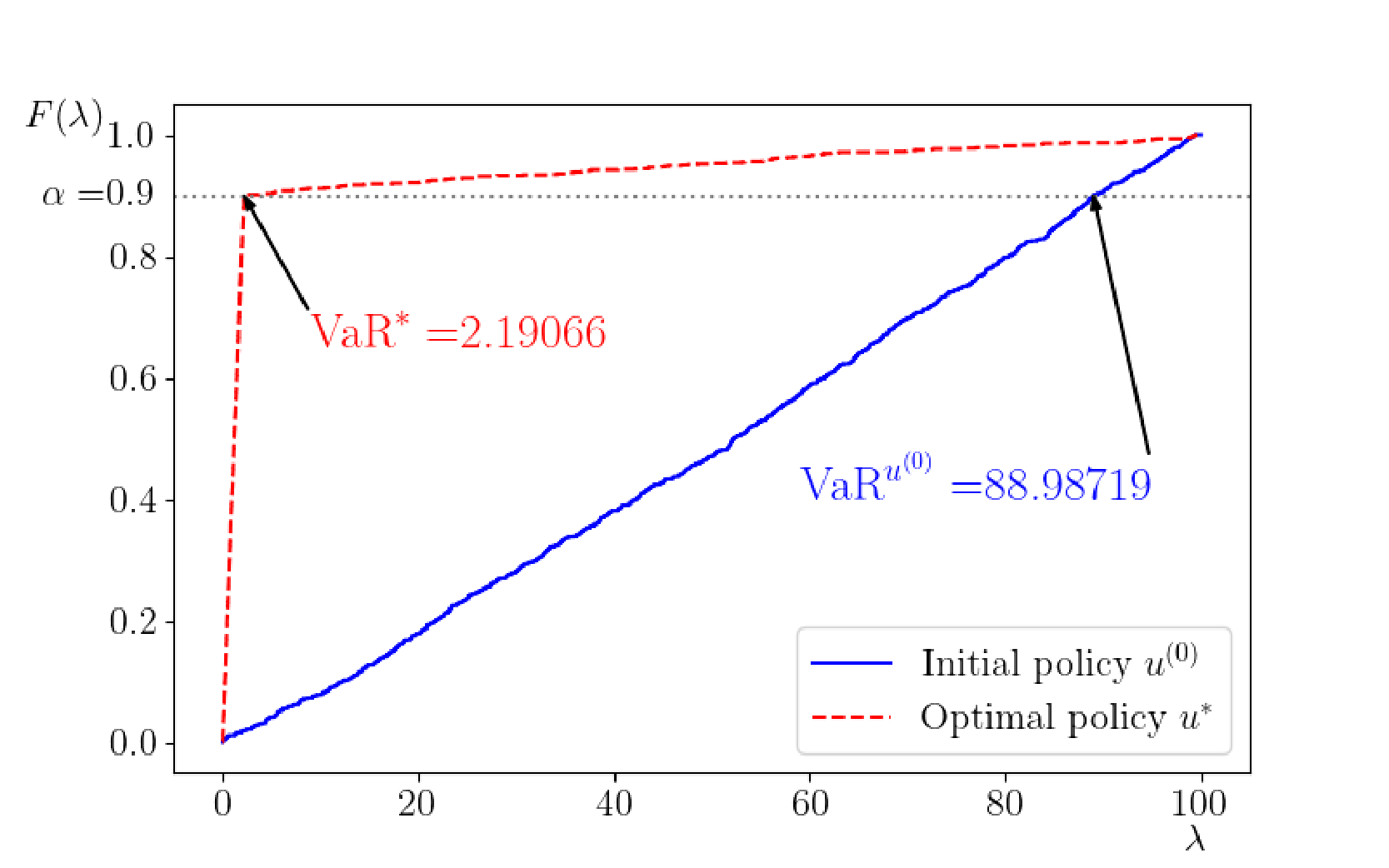}
        \label{fig:u_m3}
    }\hfill
    \subfigure[$(|\mathcal S|, |\mathcal A|) = (1000,1000)$]{
        \includegraphics[width=0.48\textwidth]{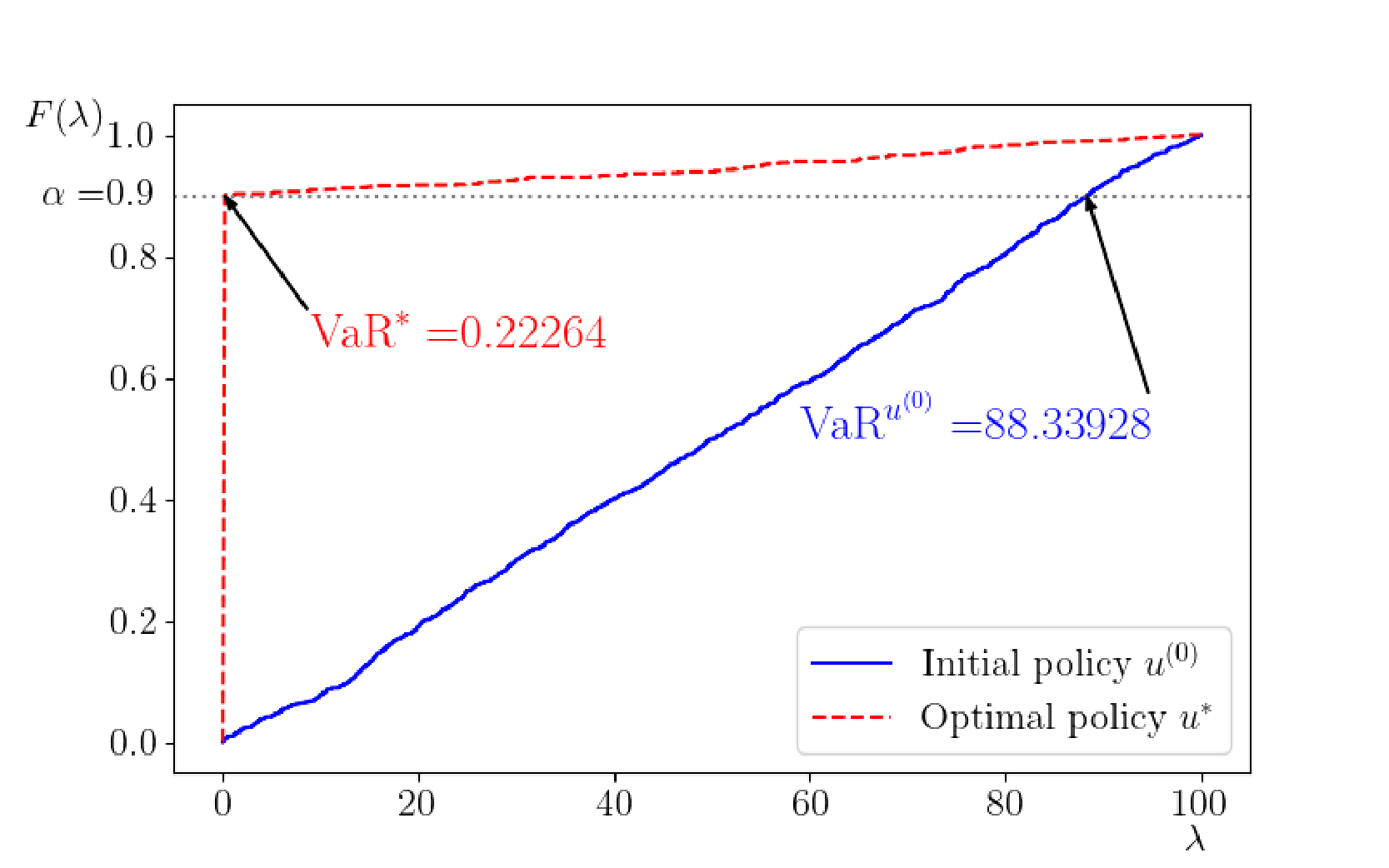}
        \label{fig:u_m4}
    }
    \caption{CDF comparison of the initial and optimal policies for ${\mathcal M}_3$ in different scenarios.}
    \label{fig:u_m_compare}
\end{figure}

Specifically, Table~\ref{tab:infty_min} presents the performance
comparison between our policy iteration algorithm and the baseline
method, Figure~\ref{fig:infty_min} illustrates the iteration
procedure of our policy iteration algorithm, and
Figure~\ref{fig:u_m_compare} depicts the CDFs of the steady-state
costs under an initial policy and the obtained optimal policy. These
results exhibit the similar patterns to those shown in
Table~\ref{tab:infty_min}, Figures~\ref{fig:infty} and
\ref{fig:u_compare}. Consistent with the observations for solving
$\mathcal{M}_1$, our policy iteration algorithm maintains both
optimality and computational efficiency when applied to solving
$\mathcal{M}_3$, aligning well with its theoretical results in
Subsection~\ref{sec:smdp_min}.

        \subsection{Finite-Horizon VaR MDPs}\label{sec:numeri_tvar}
        In this subsection, we consider the finite-horizon VaR maximization
        MDP. We conduct numerical experiments to compare the computational
        efficiency between our policy iteration type algorithm
        (Algorithm~\ref{alg:PI_finte_1}) and the value iteration type
        algorithm proposed by \cite{li2022quantile}, hereafter referred to
        as ``Ours" and ``Li's", respectively. To ensure a comprehensive
        evaluation, we consider MDP models with varying configurations.
        Specifically, the experiments consist of four subplots, each
        corresponding to the scaling of a different MDP parameter: time
        horizon $T$, state space size $|\mathcal{S}|$, action space size
        $|\mathcal{A}|$, and maximum reward $R_{\max}$. For each
        configuration, the transition probabilities are randomly generated,
        and the instantaneous rewards are uniformly sampled from the integer
        set $\{0,1,2,\dots, R_{\max}\}$. The baseline configuration is set
        as $T = 50$, $|\mathcal{S}| = 10$, $|\mathcal{A}| = 10$, and
        $R_{\max}=100$. In each case, a single parameter is varied while the
        others are held fixed at their baseline values. To ensure
        statistical robustness, we conduct $12$ independent replications for each
        configuration using different random seeds. The curves in
        Figure~\ref{fig:finite} report the average computation time across
        these trials, with shadow areas representing the standard
        deviations.

        \begin{figure}[ht]
            \centering
            \subfigure[The scaling of time horizon $T$]{
                \includegraphics[width=0.45\textwidth]{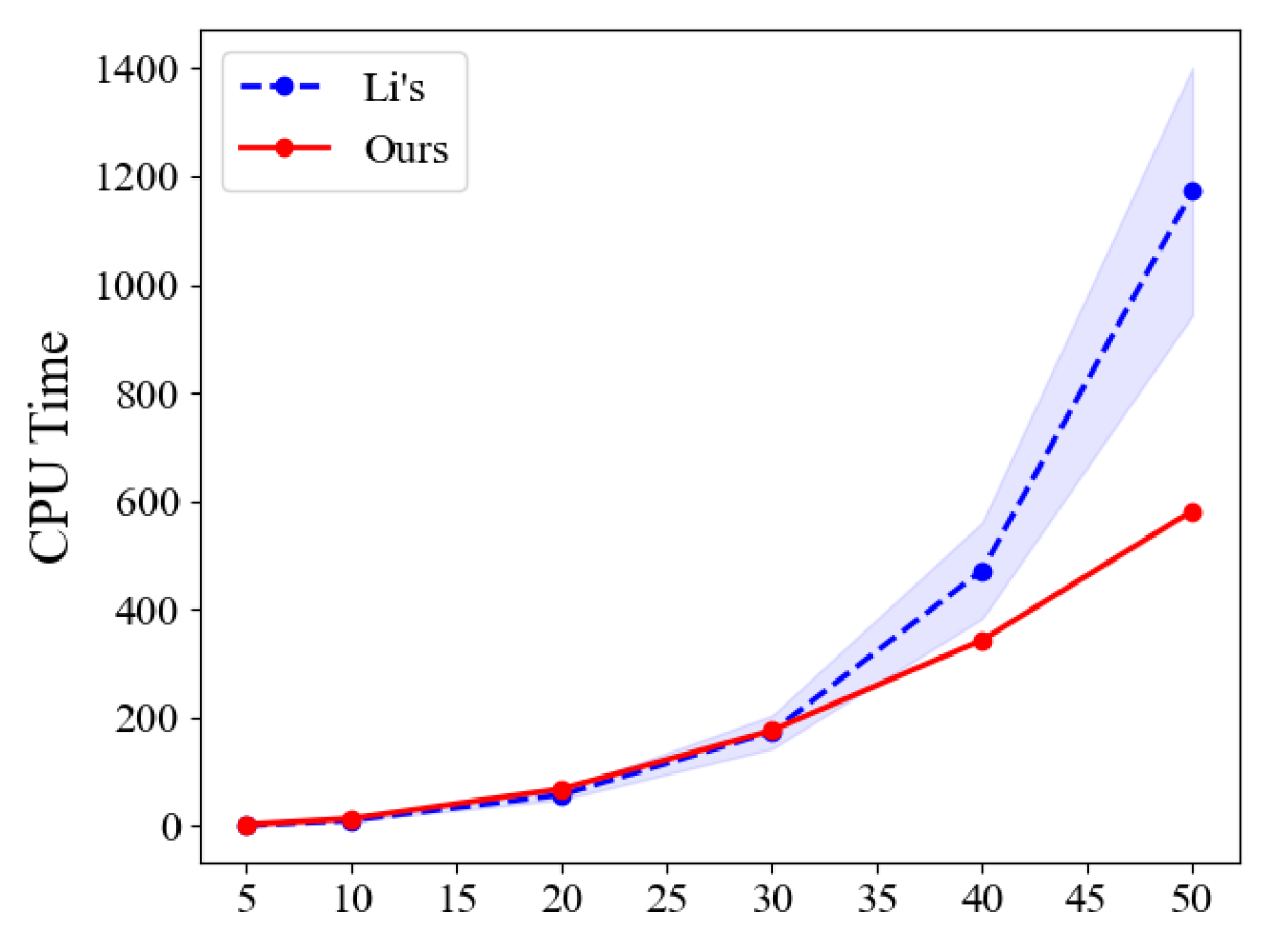}
                \label{fig:finite_T}
            }\hfill
            \subfigure[The scaling of state space size $|\mathcal S|$]{
                \includegraphics[width=0.45\textwidth]{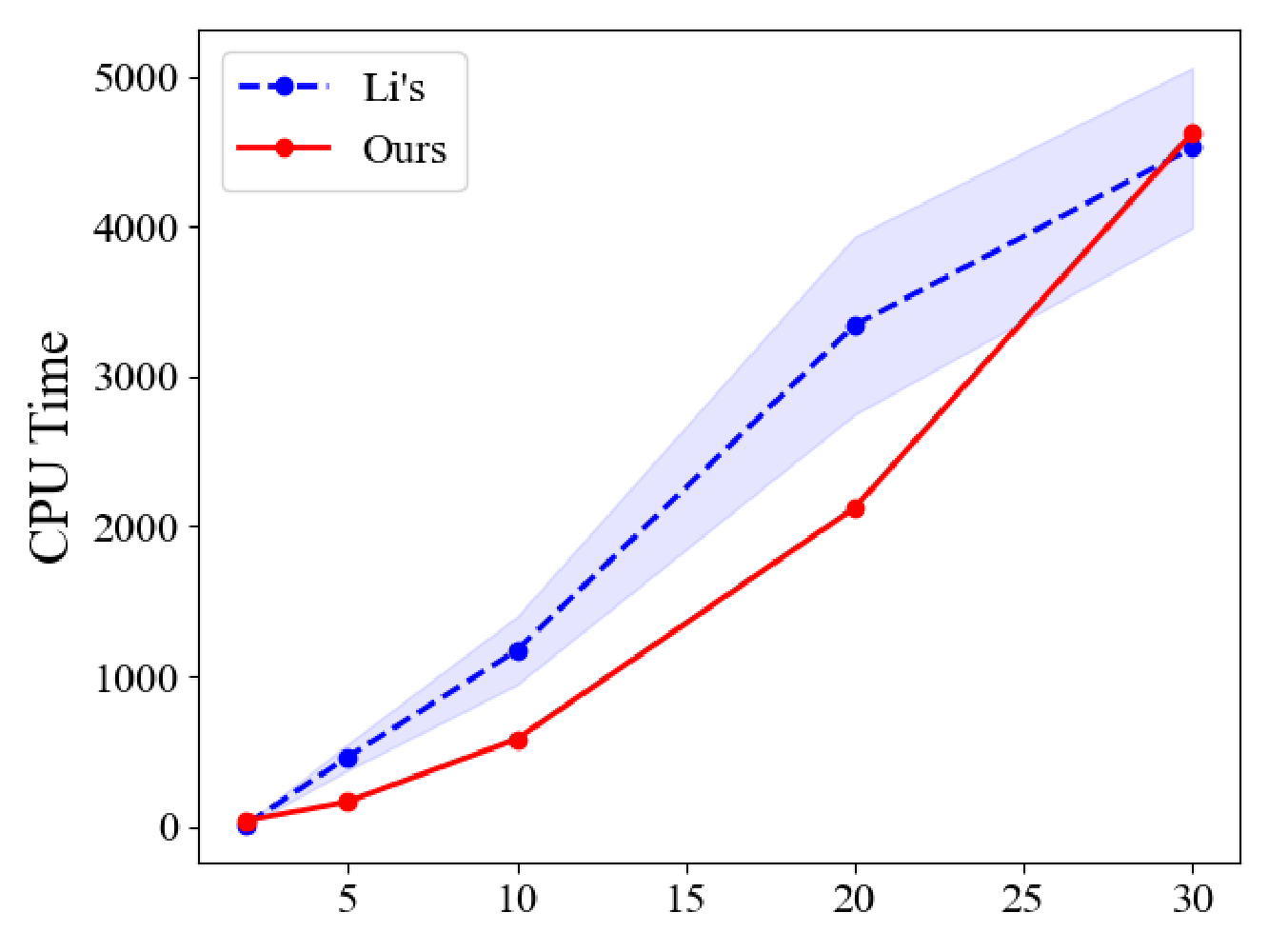}
                \label{fig:finite_S}
            }\hfill
            \subfigure[The scaling of action space size $|\mathcal A|$]{
                \includegraphics[width=0.45\textwidth]{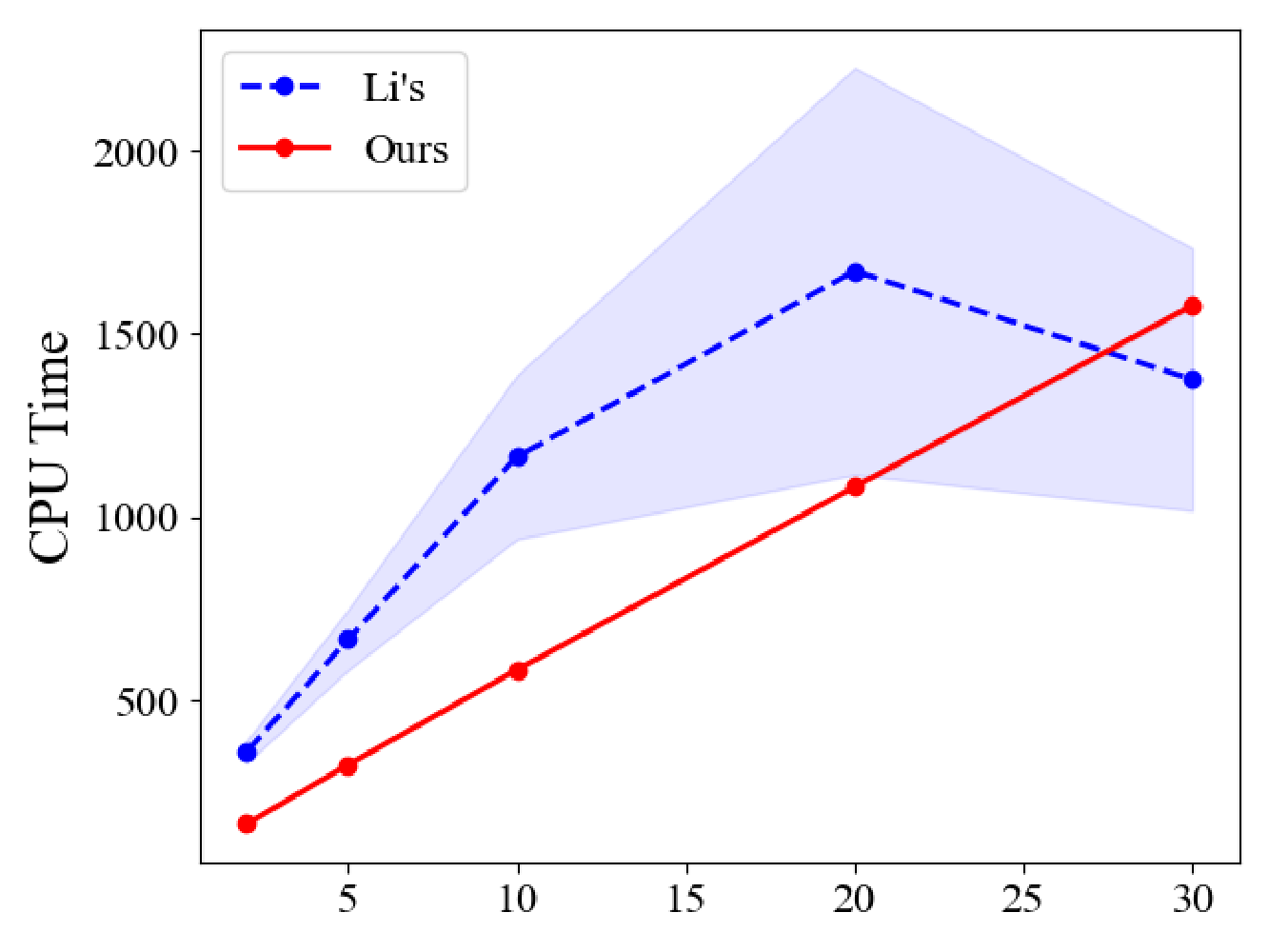}
                \label{fig:finite_A}
            }\hfill
            \subfigure[The maximum reward $R_{\max}$]{
                \includegraphics[width=0.45\textwidth]{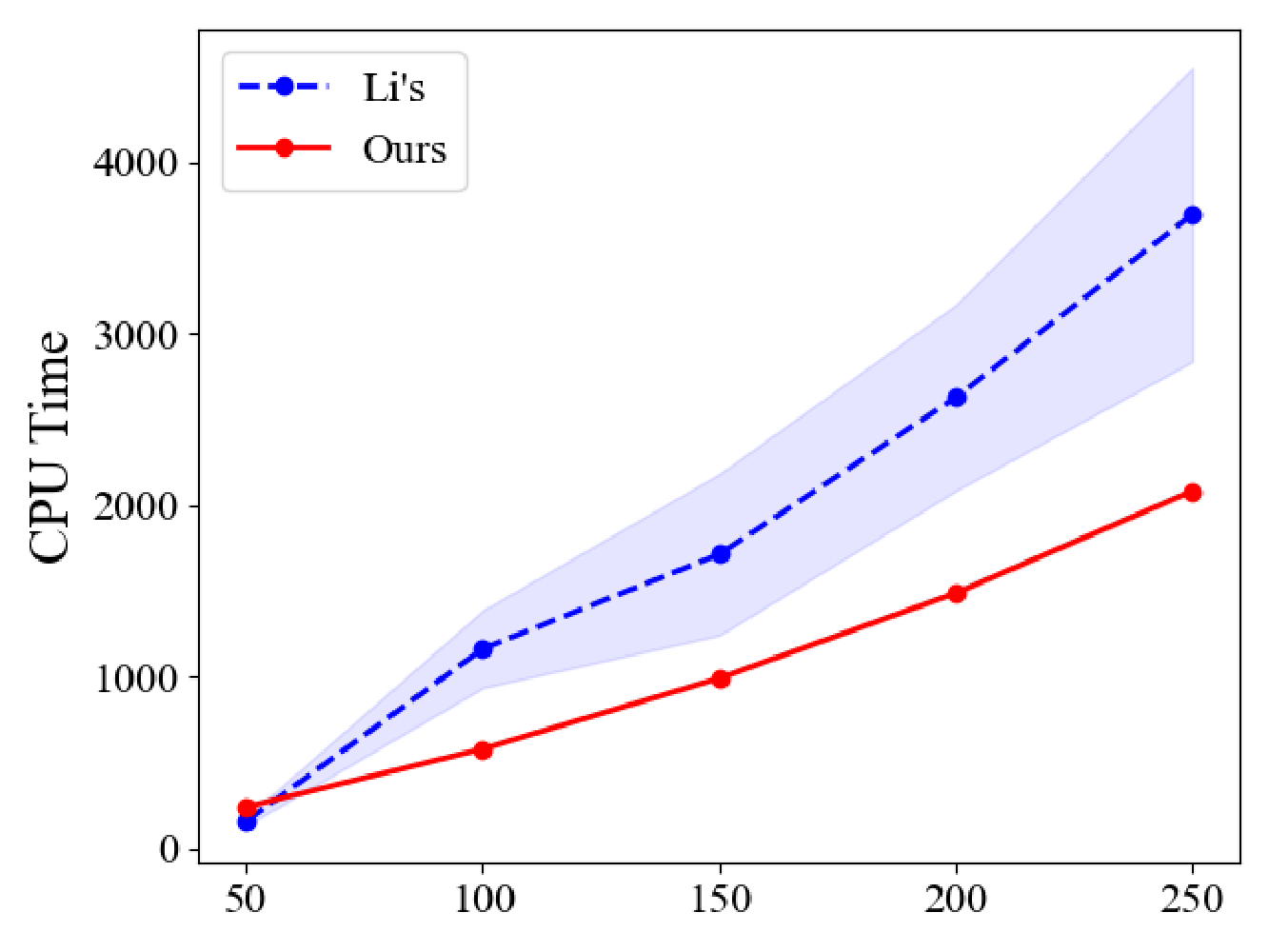}
                \label{fig:finite_R}
            }
            \caption{The computation time comparison between our algorithm and Li's algorithm, where the shadow areas represent the standard deviation of 12 trials.}
            \label{fig:finite}
        \end{figure}

        As shown in Figure~\ref{fig:finite}, both algorithms exhibit comparable computational efficiency across all parameter scales.
        A critical distinction of computational performance between
        these two algorithms lies in the stability of computation time.
         The standard deviation of Li's algorithm is big, while ours is negligible. Averaged over different $T$ values, the computation time's standard deviation is $1.2$ for our algorithm versus $59.7$ for Li's.
            Our algorithm exhibits near-zero standard deviations across different scenarios, reflecting strong robustness against random parameter settings. The above analysis also illustrates that our algorithm is usually more efficient than the value iteration-type algorithm by \cite{li2022quantile}, which is similar to the comparison of policy
            iteration and value iteration in the classical MDP theory.
            These features highlight the potential advantages of our approach's applicability to real-world problems.

            \subsection{Application to Renewable Energy Microgrid Systems}\label{sec:microgrid}
            In this subsection, we apply Algorithm~\ref{alg:PI_steady} to
            balance the renewable power supply and demand in a microgrid system.
            As shown in Figure~\ref{fig:microgrids}, the microgrid has to meet
            its uncertain electric power demand through the scheduling among
            renewable generation, energy storage, and trades with the main grid.
            The electric power trade with the main grid occurs when the
            renewable generation and discharge from the storage cannot satisfy
            the power demand. A positive (negative) power trade represents power
            selling to (buying from) the main grid. To improve the independence of
            the microgrid, we may use the VaR of the power trade to measure its
            reliance and a larger VaR is preferred. In this example, our
            objective is to optimize the charging or discharging policy of
            energy storage such that the $0.9$-VaR of the power trade can be
            maximized.

            \begin{figure}[ht]
                \centering
                \includegraphics[width=0.7\textwidth]{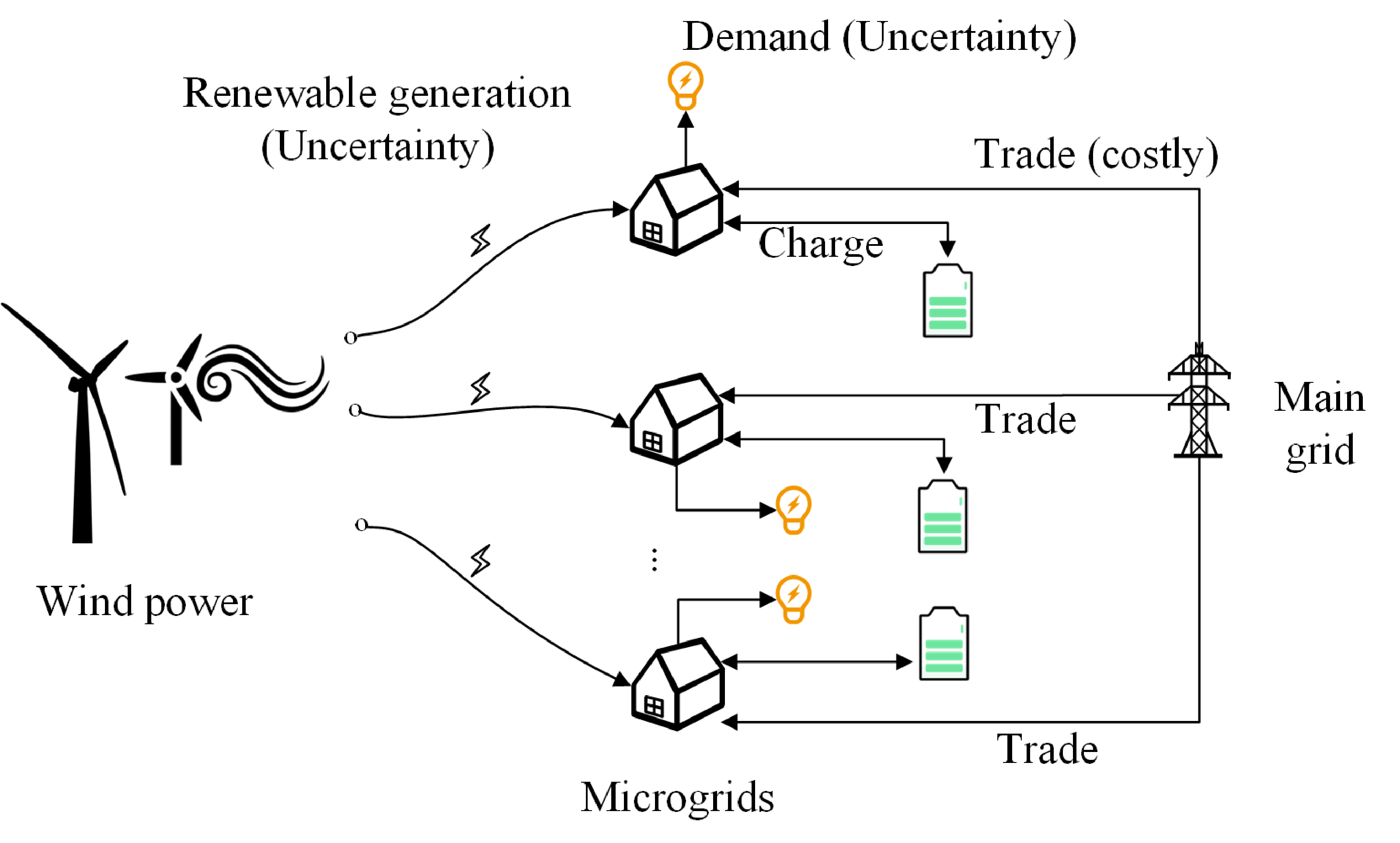}
                \caption{The illustration of a microgrid energy system.}
                \label{fig:microgrids}
            \end{figure}

            Since our algorithm is designed for MDPs with discrete state and
            action spaces, we use the uniform discretization method to
            discretize the continuous variables of the system with a $0.1$
            resolution. All the parameters of the system are defined as follows.
            \begin{itemize}
                \item \textbf{State Space.} The system state at time $t$ is defined as $\bm{s_t} = (g_t, b_t, d_t) \in \mathcal{S}$, where:
                \begin{itemize}
                    \item $g_t$: Renewable generation level and $g_t \in \mathcal{G} = \{0.0, 0.6, 1.2, 1.8, 2.4, 3.0\}$.
                    \item $b_t$: Storage level with capacity bounds, $B_{\min} \leq b_t \leq B_{\max}$, $B_{\min} = 0.4$, $B_{\max} = 3.4$.
                    \item $d_t$: Demand level and $d_t \in \mathcal{D} = \{0.6, 1.2, 1.8, 2.4, 3.0, 3.6\}$.
                \end{itemize}
                Since the storage levels are discretized into 31 states, the state space $\mathcal{S}$ contains $6 \times 31 \times 6 = 1116$ discrete states.
                \item \textbf{Action Space.} The action $a_t$ represents the charge/discharge power of the storage system at time $t$, where positive (negative) $a_t$ represents a discharge (charge) action.
                \begin{itemize}
                    \item  Maximum charge/discharge limits, $-C_{\max} \leq a_t \leq C_{\max}$, $C_{\max} = 1.2$. The action space $\mathcal{A}$ is discretized into $25$ discrete values.
                    \item Additional state-dependent constraints of capacity, $b_t - B_{\max} \leq a_t \leq b_t - B_{\min}$.
                \end{itemize}
                \item \textbf{State Transition Dynamics.} The system evolves according to the following rules:

                \textbf{(a) Renewable Generation}: We assumed the renewable generation follows a Markov chain whose state transition probability matrix $P_g(g_{t+1}|g_t)$ is given by:
                \begin{equation*}
                    \begin{bmatrix}
                        0.939 & 0.051 & 0.006 & 0.002 & 0.001 & 0.001 \\
                        0.400 & 0.443 & 0.103 & 0.029 & 0.011 & 0.014 \\
                        0.157 & 0.373 & 0.260 & 0.115 & 0.045 & 0.050 \\
                        0.079 & 0.240 & 0.250 & 0.192 & 0.104 & 0.135 \\
                        0.078 & 0.139 & 0.183 & 0.192 & 0.140 & 0.268 \\
                        0.042 & 0.074 & 0.081 & 0.099 & 0.095 & 0.609 \\
                    \end{bmatrix},
                \end{equation*}
                where rows represent the current generation states $g_t$ (from top to bottom: $0.0$ to $3.0$), and columns represent the next states $g_{t+1}$.

                \textbf{(b) Energy Demand}: We assume the energy demand also follows a Markov chain whose state transition probability matrix $P_d(d_{t+1}|d_t)$ is given by:
                \begin{equation*}
                    \begin{bmatrix}
                        0.751 & 0.249 & 0.000 & 0.000 & 0.000 & 0.000 \\
                        0.031 & 0.834 & 0.135 & 0.000 & 0.000 & 0.000 \\
                        0.000 & 0.107 & 0.819 & 0.074 & 0.000 & 0.000 \\
                        0.000 & 0.000 & 0.139 & 0.838 & 0.023 & 0.000 \\
                        0.000 & 0.000 & 0.000 & 0.189 & 0.794 & 0.017 \\
                        0.000 & 0.000 & 0.000 & 0.000 & 0.267 & 0.733 \\
                    \end{bmatrix},
                \end{equation*}
                where rows represent the current demand states $d_t$ (from top to bottom: $0.6$ to $3.6$), and columns represent the next states $d_{t+1}$.

                \textbf{(c) Storage Level}: The storage level follows a deterministic update function, i.e., $b_{t+1} = b_t - a_t$.
                \item \textbf{Reward Function}: The instantaneous reward $R_t$ represents the power trade between the microgrid and the main grid: $R_t = g_t + a_t - d_t$. Positive (negative) $R_t$ indicates power selling (buying).
                \item \textbf{Performance Criterion}: Maximize the VaR of steady-state rewards, $\max\limits_{u\in \mathcal U^{\rm{RS}}} \VaR^u$.

            \end{itemize}

            The wind speed data used in this example is obtained from the Measurement and Instrumentation Data Center (MIDC) at the National Renewable Energy Laboratory, comprising real-world measurements collected from 1996 to present \citep{jager1996nrel}.
            For demand modeling, hourly load data from the Independent Electricity System Operator (IESO) of Ontario are employed, covering the period from 2013 to 2022 \citep{su2010microgrid}. We use these
            data to generate the aforementioned transition probability
            matrices.


            We apply Algorithm~\ref{alg:PI_steady} to solve this problem. Figure~\ref{fig_engy4} illustrates the iteration procedure of VaR values under three probability levels ($\alpha = 0.1, 0.5, 0.9$).
            These trajectories exhibit strict monotonic improvement of VaR during the iteration procedure, ultimately converging to their maximum VaR values. Specifically:
            \begin{itemize}
                \item For $\alpha=0.9$, the curve shows linear growth characteristics, reaching an optimal value $0.6$. This positive value indicates the system achieves power surplus under high-$\alpha$ scenarios.
                \item For $\alpha=0.5$ and $0.1$, the curves converge to the negative VaR values $-0.6$ and $-1.6$, respectively, which indicate power buying from the main grid. Even with an optimal scheduling policy, the system has a power deficit of $1.6$ in the worst $10\%$ scenarios. 
            \end{itemize}

            \begin{figure}[htbp]
                \centering
                \includegraphics[width=0.65\columnwidth]{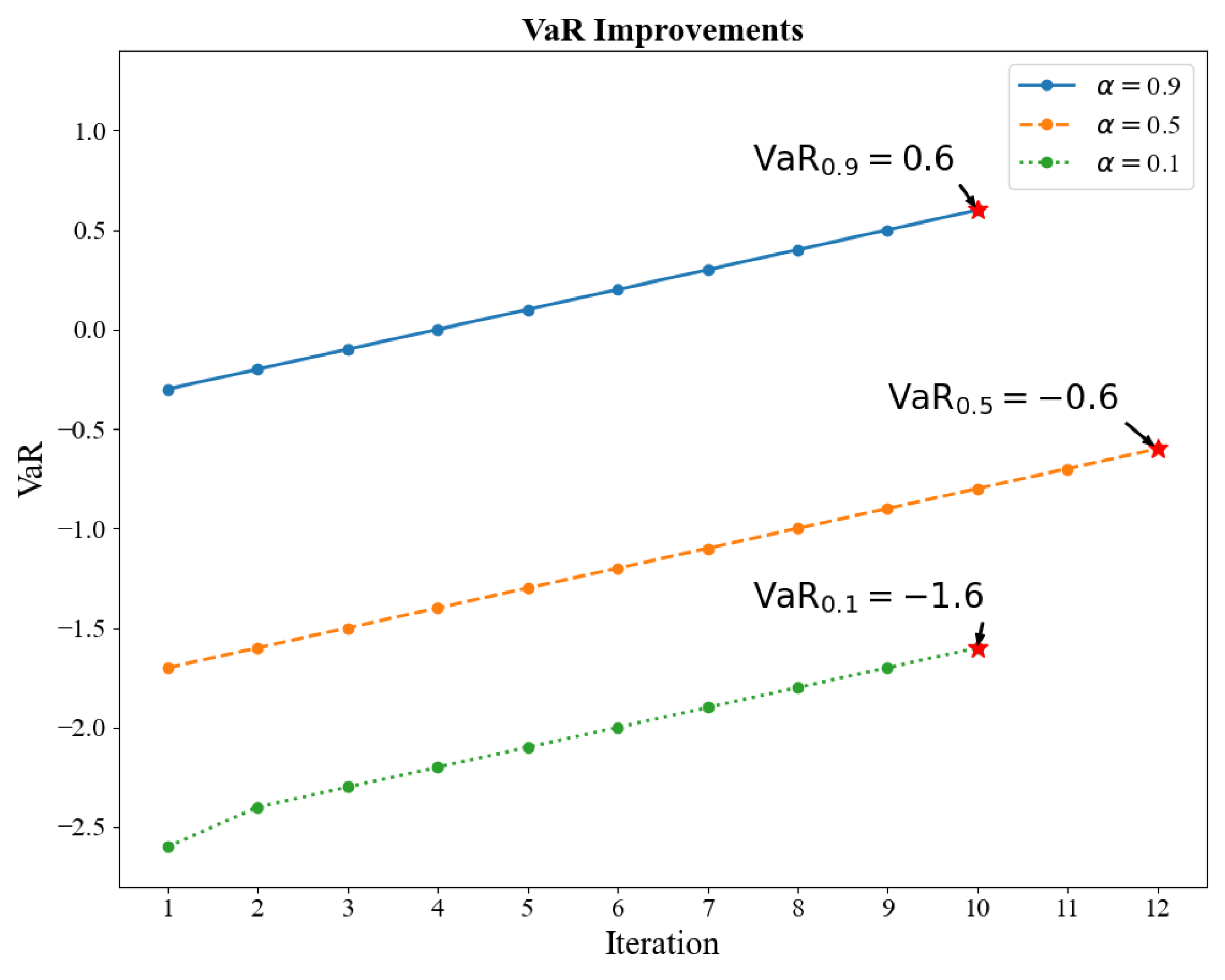}
                \caption{The improvement of VaR values in renewable energy microgrid systems.} \label{fig_engy4}
            \end{figure}

            \begin{figure}[htp]
                \centering
                \subfigure[$\alpha = 0.1$]{
                    \includegraphics[width=0.31\textwidth]{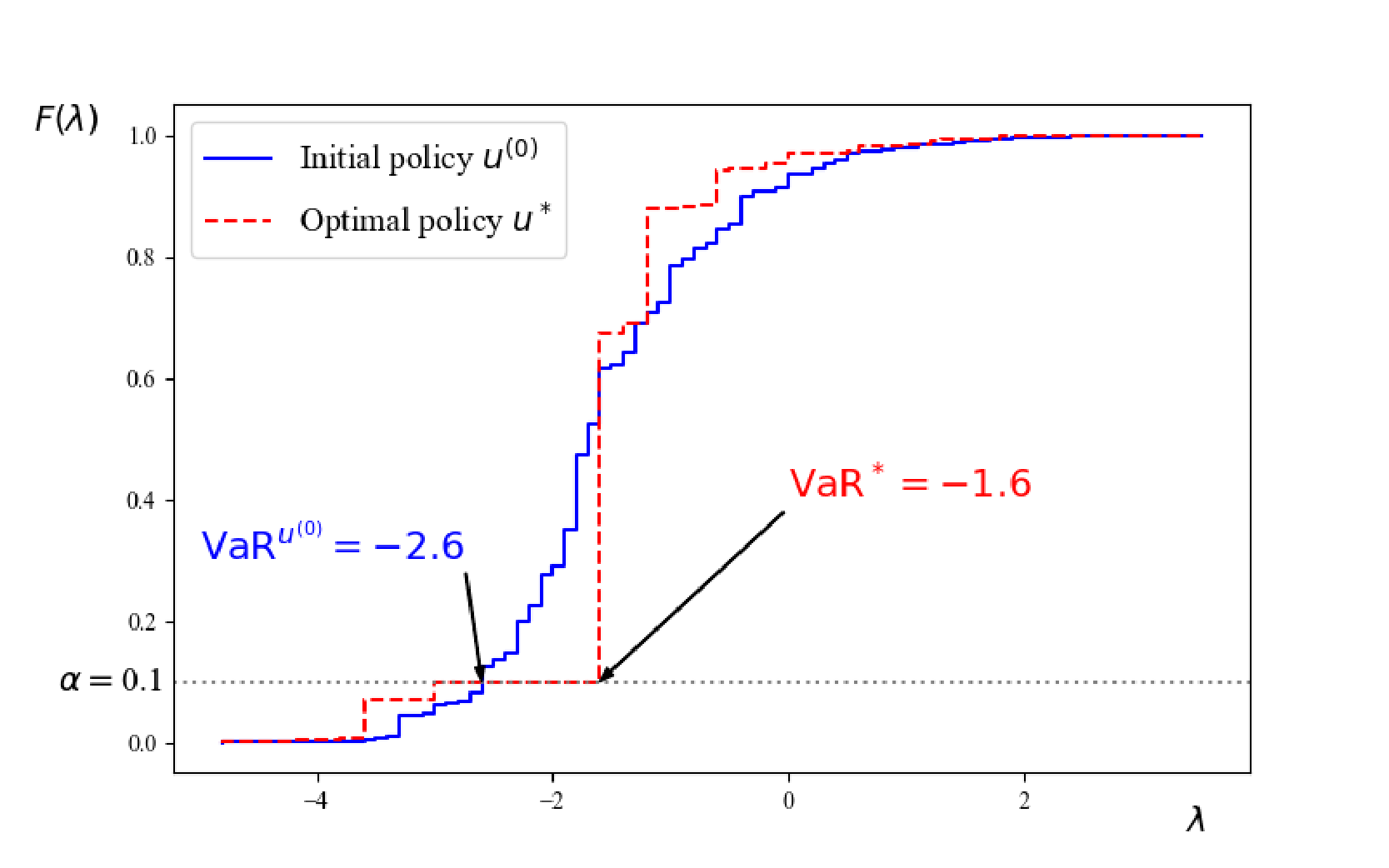}
                    \label{fig:u_e1}
                }
                \subfigure[$\alpha = 0.3$]{
                    \includegraphics[width=0.31\textwidth]{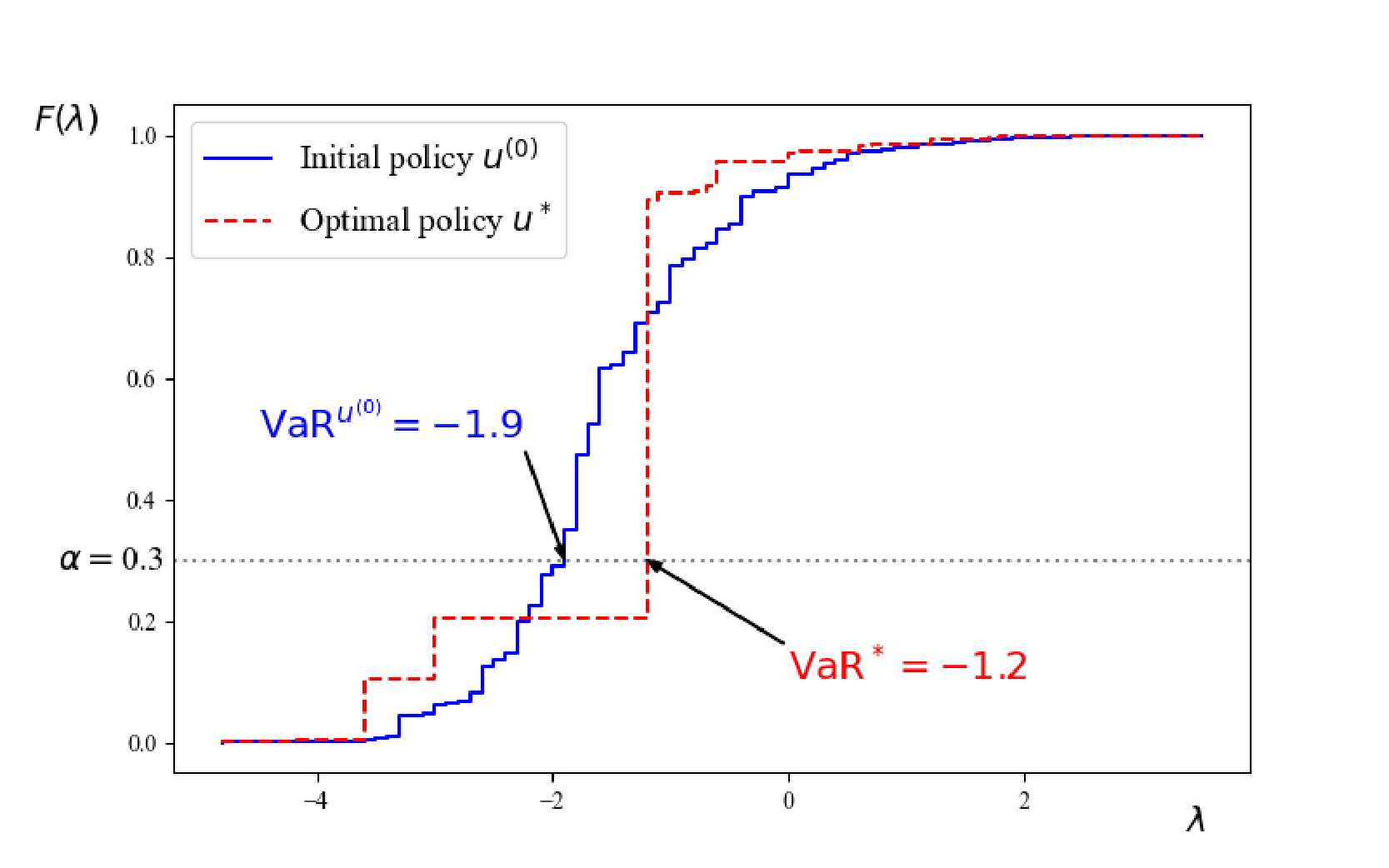}
                    \label{fig:u_e3}
                }
                \subfigure[$\alpha = 0.4$]{
                    \includegraphics[width=0.31\textwidth]{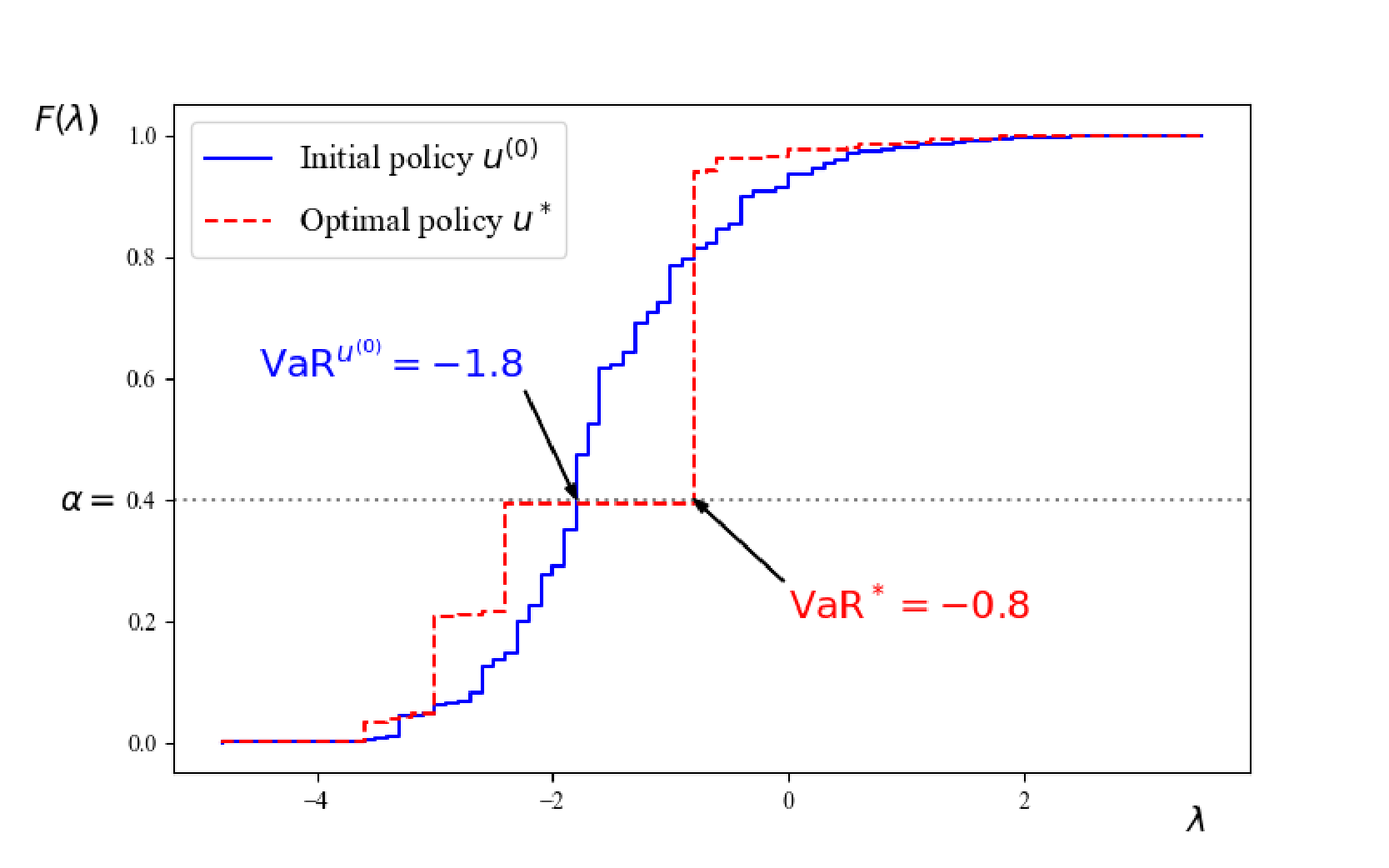}
                    \label{fig:u_e4}
                }
                \subfigure[$\alpha = 0.5$]{
                    \includegraphics[width=0.31\textwidth]{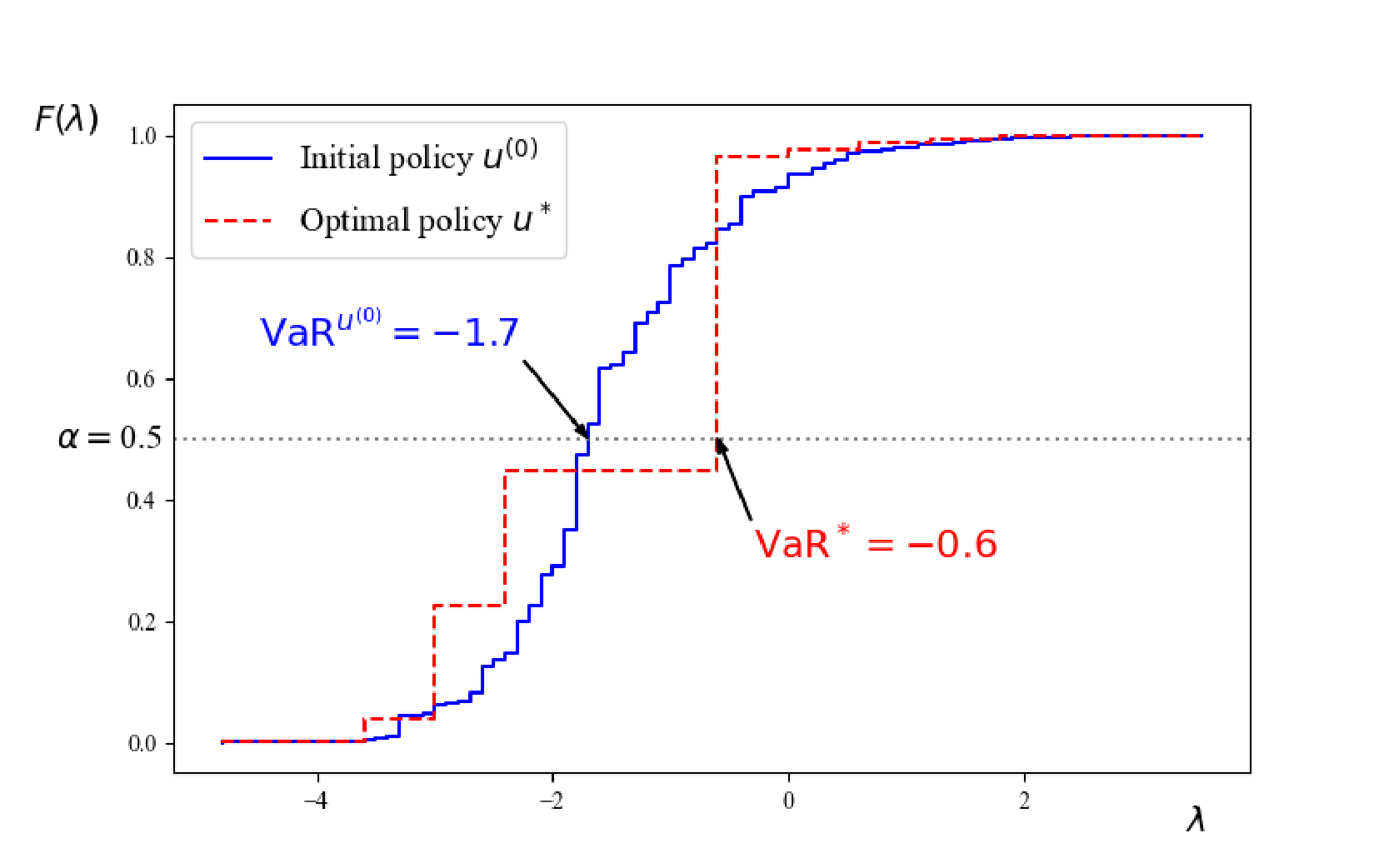}
                    \label{fig:u_e5}
                }
                \subfigure[$\alpha = 0.7$]{
                    \includegraphics[width=0.31\textwidth]{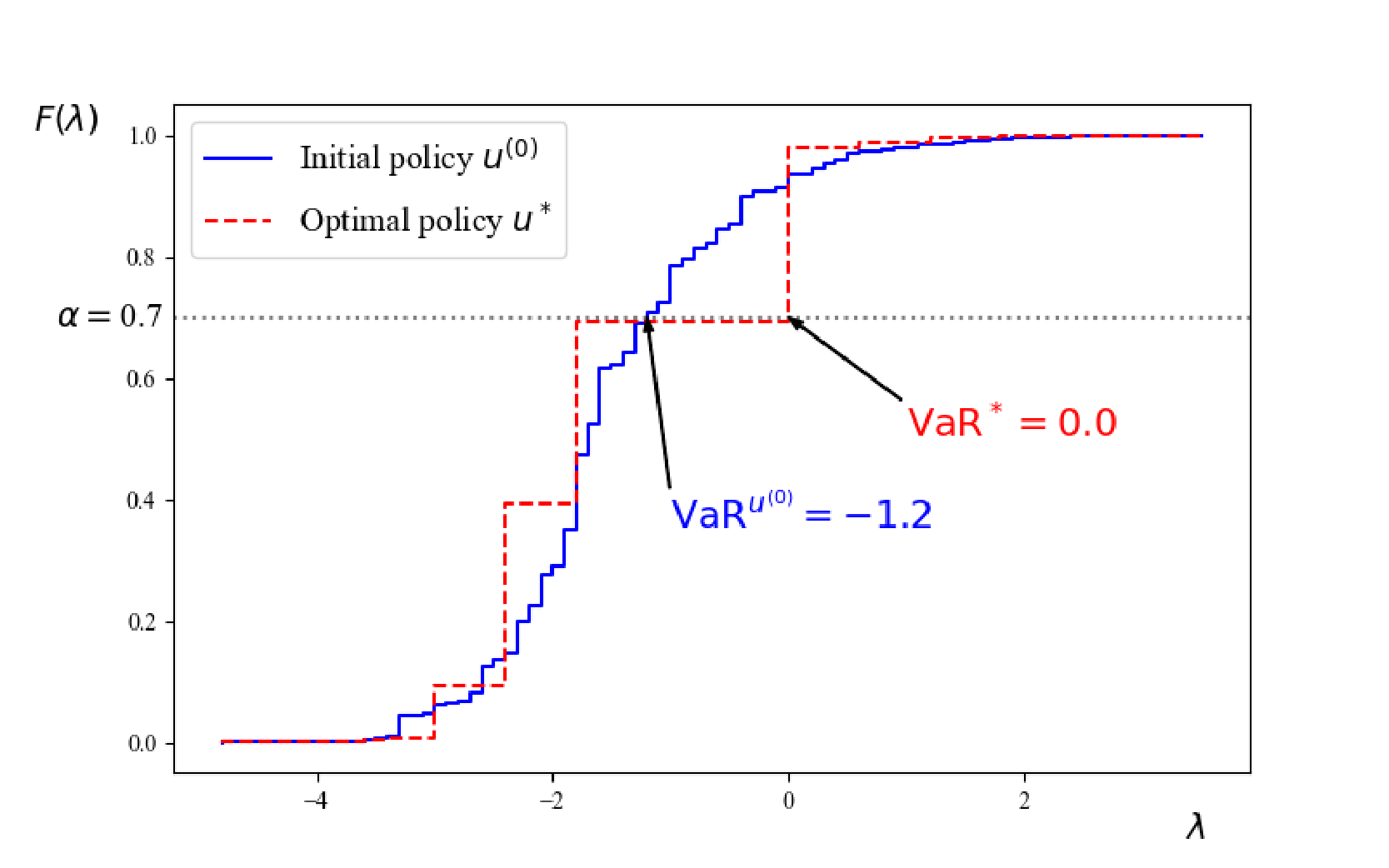}
                    \label{fig:u_e7}
                }
                \subfigure[$\alpha = 0.9$]{
                    \includegraphics[width=0.31\textwidth]{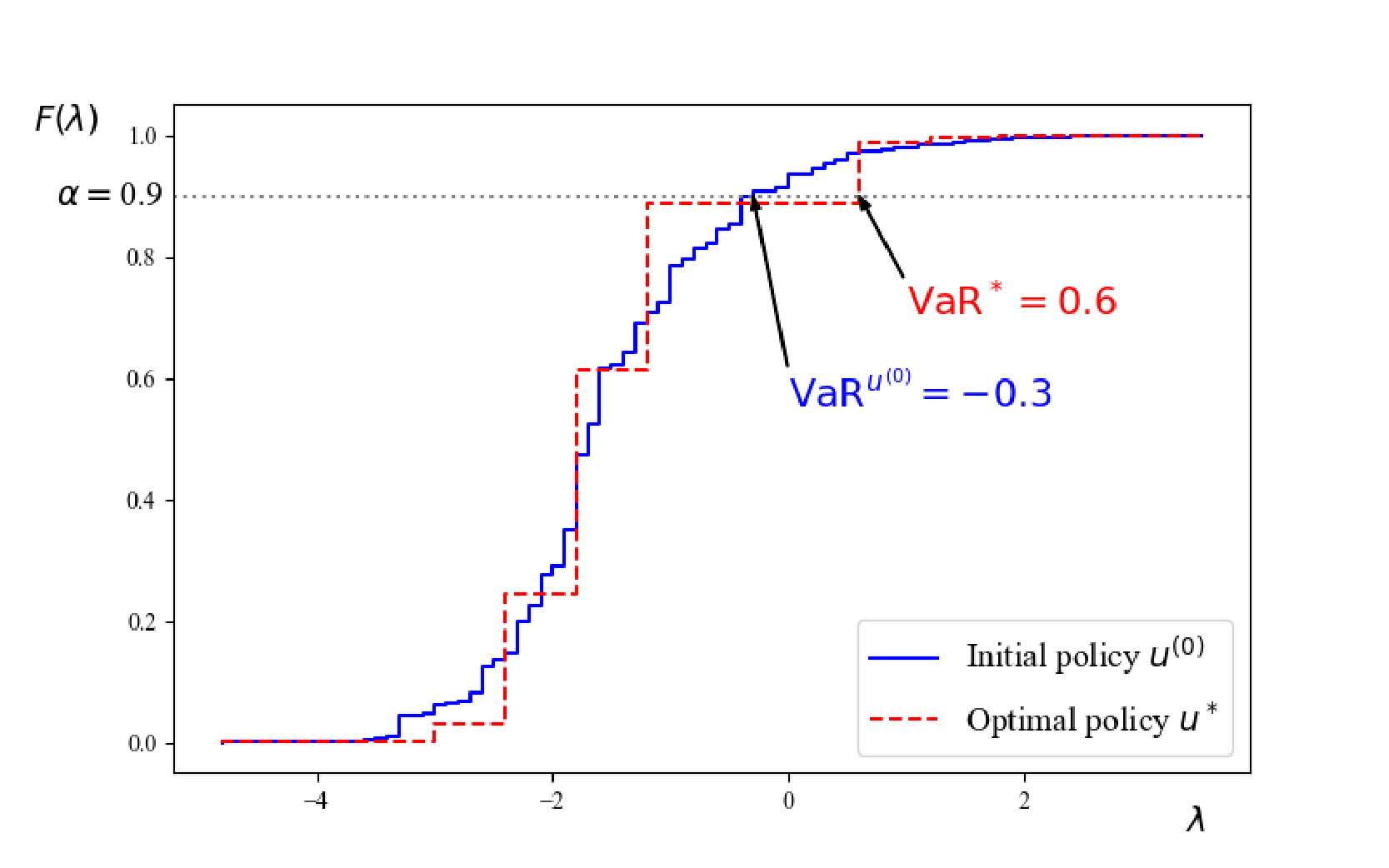}
                    \label{fig:u_e9}
                }
                \caption{CDFs of steady-state rewards under the initial and optimal policies in microgrids.}
                \label{fig:u_e_compare}
            \end{figure}

            Moreover, Figure~\ref{fig:u_e_compare} depicts the CDFs of
            steady-state rewards under the initial policy and the optimal
            policy. Compared with the initial policy, the optimal policy can
            achieve a significantly larger VaR value, i.e., improving the energy
            surplus. Optimistic scenarios (with $\alpha\ge 0.7$) can yield a
            positive power balance of the microgrid. By maximizing the VaR
            value, the microgrid system can effectively reduce the reliance on
            power buying from the main grid, improving its energy
            self-sufficiency under the high uncertainty of renewable energy
            generation.

\section{Conclusion}\label{sec:conclu}
This paper studies discrete-time MDPs with the VaR optimality
criterion, which has not been adequately explored in the literature.
Different from the existing risk decomposition method, we optimize
VaR MDPs from a new perspective by utilizing the essential relations
between VaR and probability metrics to establish its equivalence
with the probabilistic MDPs. We propose a unified bilevel
optimization framework to handle VaR optimization in various MDP
settings, including steady-state VaR MDPs and finite-horizon VaR
MDPs. The existence of an optimal deterministic policy for VaR MDPs
is proved. Our primary contribution lies in proposing an efficient
policy iteration type algorithm for finding a globally optimal
policy for VaR MDPs.

We study the VaR maximization of both steady-state rewards and
accumulated rewards. We convert the steady-state VaR maximization
MDP to a bilevel optimization problem, where the inner level is a
standard average minimization MDP and the outer level is a static
minimization problem. With the nested reformulation, we propose a
policy iteration algorithm by solving the inner standard average
minimization MDP to improve the policy, until the inner optimal
value exceeds the given probability level $\alpha$. The convergence
of the algorithm and its global optimality are further proved. The
above optimization approach is also applied to finite-horizon VaR
maximization MDPs, by replacing the inner standard average
minimization MDP with a standard finite-horizon augmented-state
minimization MDP. Additionally, we extend the results to minimize
steady-state VaR MDPs and finite-horizon VaR MDPs. Numerical
experiments are conducted to verify our main results, which show
that our policy iteration type algorithm achieves computational
efficiency comparable to, and in many cases significantly higher
than, the value iteration type algorithm in the existing literature
\citep{li2022quantile}.

One of the future research topics is to deal with discounted VaR
MDPs, optimizing the VaR of discounted accumulated rewards over an
infinite horizon. An effective policy iteration algorithm for
solving discounted VaR MDPs needs to be developed. Another promising
direction for future research involves developing data-driven
algorithms for VaR MDPs to address model-free scenarios where both
transition probabilities and reward functions are unknown. This
would enable the optimization of VaR MDPs through learning directly
from sampled trajectories.


\begin{thebibliography}{}

    \bibitem[{Ahn et~al.(1999)Ahn, Boudoukh, Richardson, and Whitelaw}]{ahn1999optimal}
    Ahn DH, Boudoukh J, Richardson M, Whitelaw RF (1999) Optimal risk management
    using options. \textit{The Journal of Finance} 54(1):359--375.

        \bibitem[{Babat et~al.(2018)Babat, Vera, and Zuluaga}]{babat2018computing}
        Babat O, Vera JC, Zuluaga LF (2018) Computing near-optimal value-at-risk
        portfolios using integer programming techniques. \textit{European Journal of
        Operational Research} 266(1):304--315.

        \bibitem[{Bandi et~al.(2024)Bandi, Han, and Proskynitopoulos}]{bandi2024robust}
        Bandi C, Han E, Proskynitopoulos A (2024) Robust queue inference from waiting
        times. \textit{Operations Research} 72(2):459--480.

        \bibitem[{Basak and Shapiro(2001)}]{basak2001value}
        Basak S, Shapiro A (2001) Value-at-risk-based risk management: Optimal policies
        and asset prices. \textit{The Review of Financial Studies} 14(2):371--405.

        \bibitem[{B{\"a}uerle and
            Ja{\'s}kiewicz(2024)}]{bauerle2024markov}
        B{\"a}uerle N, Ja{\'s}kiewicz A (2024) Markov decision processes
        with risk-sensitive criteria: An overview. \textit{Mathematical Methods of Operations Research} 99(1):141--178.


        \bibitem[{B{\"a}uerle and Ott(2011)}]{bauerle2011}
        B{\"a}uerle N, Ott J (2011) {Markov decision processes with
            average-value-at-risk criteria}. \textit{Mathematical Methods of Operations
            Research} 74(3):361--379.


        \bibitem[{B{\"a}uerle and Rieder(2014)}]{bauerle2014more}
        B{\"a}uerle N, Rieder U (2014) More risk-sensitive Markov decision processes.
        \textit{Mathematics of Operations Research} 39(1):105--120.

        \bibitem[{Benati and Rizzi(2007)}]{benati2007mixed}
        Benati S, Rizzi R (2007) A mixed integer linear programming formulation of the
        optimal mean/value-at-risk portfolio problem. \textit{European Journal of Operational
        Research} 176(1):423--434.

        \bibitem[{Berkowitz and O'Brien(2002)}]{berkowitz2002accurate}
        Berkowitz J, O'Brien J (2002) How accurate are value-at-risk models at
        commercial banks? \textit{The Journal of Finance} 57(3):1093--1111.


        \bibitem[{Chapman et~al.(2021)Chapman, Bonalli, Smith, Yang, Pavone, and
            Tomlin}]{chapman2021risk}
        Chapman MP, Bonalli R, Smith KM, Yang I, Pavone M, Tomlin CJ (2021)
        Risk-sensitive safety analysis using conditional value-at-risk. \textit{IEEE
        Transactions on Automatic Control} 67(12):6521--6536.


        \bibitem[{Chow et~al.(2015)Chow, Tamar, Mannor, and Pavone}]{chow2015risk}
        Chow Y, Tamar A, Mannor S, Pavone M (2015) Risk-sensitive and robust decision-making:
        A CVaR optimization approach. \textit{Advances in Neural Information Processing Systems (NIPS'2015)} 28:1522--1530.

        \bibitem[{Christoffersen(2011)}]{christoffersen2011elements}
        Christoffersen P (2011) \textit{Elements of Financial Risk Management}. Academic Press.



        \bibitem[{Cuoco et~al.(2008)Cuoco, He, and Isaenko}]{cuoco2008optimal}
        Cuoco D, He H, Isaenko S (2008) Optimal dynamic trading strategies with risk
        limits. \textit{Operations Research} 56(2):358--368.



        \bibitem[{Ding and Feinberg(2022{\natexlab{a}})}]{ding2022cvar}
        Ding R, Feinberg E (2022{\natexlab{a}}) CVaR optimization for MDPs: Existence
        and computation of optimal policies. \textit{ACM SIGMETRICS Performance Evaluation
        Review} 50(2):39--41.

        \bibitem[{Ding and Feinberg(2022{\natexlab{b}})}]{ding2022sequential}
        Ding R, Feinberg E (2022{\natexlab{b}}) Sequential optimization of CVaR. \textit{arXiv
        preprint arXiv:221107288}.

        \bibitem[{Duffie and Pan(1997)}]{duffie1997overview}
        Duffie D, Pan J (1997) An overview of value at risk. \textit{Journal of
        Derivatives} 4(3):7--49.


        \bibitem[{Feng et~al.(2015)Feng, W{\"a}chter, and Staum}]{feng2015practical}
        Feng M, W{\"a}chter A, Staum J (2015) Practical algorithms for value-at-risk
        portfolio optimization problems. \textit{Quantitative Finance Letters} 3(1):1--9.

        \bibitem[{Filar et~al.(1995)Filar, Krass, and Ross}]{filar1995percentile}
        Filar JA, Krass D, Ross KW (1995) Percentile performance criteria for limiting
        average Markov decision processes. \textit{IEEE Transactions on Automatic Control}
        40(1):2--10.

        \bibitem[{Gaivoronski and Pflug(2005)}]{gaivoronski2005value}
        Gaivoronski AA, Pflug G (2005) Value-at-risk in portfolio optimization:
        properties and computational approach. \textit{Journal of Risk} 7(2):1--31.

        \bibitem[{Ghaoui et~al.(2003)Ghaoui, Oks, and Oustry}]{ghaoui2003worst}
        Ghaoui LE, Oks M, Oustry F (2003) Worst-case value-at-risk and robust portfolio
        optimization: A conic programming approach. \textit{Operations Research} 51(4):543--556.

        \bibitem[{Gilbert and Weng(2016)}]{Gilbert2016} Gilbert H, Weng P (2016) Quantile reinforcement
        learning.  \textit{arXiv preprint arXiv:1611.00862}.

        \bibitem[{Goh et~al.(2012)Goh, Lim, Sim, and Zhang}]{goh2012portfolio}
        Goh JW, Lim KG, Sim M, Zhang W (2012) Portfolio value-at-risk optimization for
        asymmetrically distributed asset returns. \textit{European Journal of Operational
        Research} 221(2):397--406.

        \bibitem[{Hau et~al.(2023)Hau, Delage, Ghavamzadeh, and
            Petrik}]{hau2023dynamic}
        Hau JL, Delage E, Ghavamzadeh M, Petrik M (2023) On dynamic programming
        decompositions of static risk measures in Markov decision processes. \textit{Advances
        in Neural Information Processing Systems} 36:51734--51757.

        \bibitem[{Hau et~al.(2024)Hau, Delage, Derman, Ghavamzadeh, and
            Petrik}]{hau2024q}
        Hau JL, Delage E, Derman E, Ghavamzadeh M, Petrik M (2024) Q-learning for
        quantile MDPs: A decomposition, performance, and convergence analysis. \textit{arXiv
        preprint arXiv:241024128}.

        \bibitem[{Hern{\'a}ndez-Lerma and Lasserre(1996)}]{Hernandez-Lerma1996}
        Hern{\'a}ndez-Lerma O, Lasserre JB (1996) \textit{Discrete-Time Markov Control
        Processes}. {Springer Science \& Business Media.}

        \bibitem[{Howard and Matheson(1972)}]{howard1972}
        Howard RA, Matheson JE (1972) {Risk-sensitive Markov decision processes}.
        \textit{Management Science} 18(7):356--369.

        \bibitem[{Huang and Guo(2015)}]{huang2015mean}
        Huang Y, Guo X (2015) Mean-variance problems for finite horizon semi-Markov
        decision processes. \textit{Applied Mathematics \& Optimization} 72(2):233--259.

        \bibitem[{Huang and Guo(2016)}]{huang2016}
        Huang Y, Guo X (2016) {Minimum average value-at-risk for finite horizon
        semi-Markov decision processes in continuous time}. \textit{SIAM Journal on Optimization} 26(1):1--28.

        \bibitem[{Huo et~al.(2017)Huo, Zou, and Guo}]{huo2017}
        Huo H, Zou X, Guo X (2017) {The risk probability criterion for discounted continuous-time Markov decision processes}.
        \textit{Discrete Event Dynamic Systems} 27(4):675--699.

        \bibitem[{Jager and Andreas(1996)}]{jager1996nrel}
        Jager D, Andreas A (1996) Nrel national wind technology center (nwtc): M2
        tower; boulder, colorado (data). Tech. rep., National Renewable Energy
        Lab.(NREL), Golden, CO (United States). Available: \url{https://midcdmz.nrel.gov/apps/sitehome.pl?site=NWTC}

        \bibitem[{Jiang et~al.(2023)Jiang, Hu, and Peng}]{jiang2023quantile}
        Jiang J, Hu J, Peng Y (2023) Quantile-based deep reinforcement learning using
        two-timescale policy gradient algorithms. \textit{arXiv preprint arXiv:230507248}.


        \bibitem[{Kouvelis and Li(2019)}]{kouvelis2019integrated}
        Kouvelis P, Li R (2019) Integrated risk management for newsvendors with
        value-at-risk constraints. \textit{Manufacturing \& Service Operations Management}
        21(4):816--832.

        \bibitem[{Larsen et~al.(2002)Larsen, Mausser, and Uryasev}]{larsen2002algorithms}
        Larsen N, Mausser H, Uryasev S (2002) Algorithms for optimization of
        value-at-risk. In: \textit{Financial Engineering, E-commerce and Supply Chain}, pp
        19--46.


        \bibitem[{Li et~al.(2022)Li, Zhong, and Brandeau}]{li2022quantile}
        Li X, Zhong H, Brandeau ML (2022) Quantile Markov decision processes.
        \textit{Operations Research} 70(3):1428--1447.

        \bibitem[{Luedtke(2014)}]{luedtke2014branch}
        Luedtke J (2014) A branch-and-cut decomposition algorithm for solving
        chance-constrained mathematical programs with finite support. \textit{Mathematical
        Programming} 146(1):219--244.


        \bibitem[{Mausser and Romanko(2014)}]{mausser2014cvar}
        Mausser H, Romanko O (2014) CVaR proxies for minimizing scenario-based
        value-at-risk. \textit{Journal of Industrial and Management Optimization}
        10(4):1109--1127.


        \bibitem[{Olsen et~al.(2017)Olsen, Tian, Wallace, Nickel, Warren, Fraser,
                Selvam, and Hamilton}]{olsen2017use}
            Olsen MA, Tian F, Wallace AE, Nickel KB, Warren DK, Fraser VJ, Selvam N,
            Hamilton BH (2017) Use of quantile regression to determine the impact on
            total health care costs of surgical site infections following common
            ambulatory procedures. \textit{Annals of Surgery} 265(2):331--339.

        \bibitem[{Oum and Oren(2010)}]{oum2010optimal}
            Oum Y, Oren SS (2010) Optimal static hedging of volumetric risk in a
            competitive wholesale electricity market. \textit{Decision Analysis} 7(1):107--122.

        \bibitem[{Pavlikov et~al.(2018)Pavlikov, Veremyev, and
                Pasiliao}]{pavlikov2018optimization}
            Pavlikov K, Veremyev A, Pasiliao EL (2018) Optimization of value-at-risk:
            Computational aspects of MIP formulations. \textit{Journal of the Operational
            Research Society} 69(5):676--690.


         \bibitem[{Pflug and Pichler(2016)}]{pflug2016time}
        Pflug GC, Pichler A (2016) Time-consistent decisions and temporal decomposition
        of coherent risk functionals. \textit{Mathematics of Operations Research}
        41(2):682--699.

        \bibitem[{Puterman(1994)}]{puterman1994markov}
        Puterman ML (1994) \textit{Markov Decision Processes: Discrete Stochastic
            Dynamic Programming}. John Wiley \& Sons.


        \bibitem[{Rockafellar(2017)}]{rockafellar2017risk}
        Rockafellar RT (2017) Risk and utility in the duality framework of convex
        analysis. In: \textit{Jonathan M. Borwein Commemorative Conference}, Springer, pp
        21--42.


        \bibitem[Rockafellar and Uryasev(2002)]{Rockafellar02}
        Rockafellar RT, Uryasev S (2002) Conditional value-at-risk for general loss distributions.
        \textit{Journal of Banking Finance} 26:1443--1471.

        \bibitem[{Romanko and Mausser(2016)}]{romanko2016robust}
        Romanko O, Mausser H (2016) Robust scenario-based value-at-risk optimization.
        \textit{Annals of Operations Research} 237:203--218.



        \bibitem[{Spada et~al.(2018)Spada, Paraschiv, and
            Burgherr}]{spada2018comparison}
        Spada M, Paraschiv F, Burgherr P (2018) A comparison of risk measures for
        accidents in the energy sector and their implications on decision-making
        strategies. \textit{Energy} 154:277--288.

        \bibitem[{Su et~al.(2010)Su, Yuan, and Chow}]{su2010microgrid}
        Su W, Yuan Z, Chow MY (2010) {Microgrid planning and operation: Solar energy
            and wind energy}. In: \textit{IEEE PES General Meeting}, IEEE, pp 1--7.


        \bibitem[U\v{g}urlu(2017)]{Ugurlu17}
        U\v{g}urlu K (2017) Controlled Markov decision processes with AVaR criteria for unbounded costs.
        \textit{Journal of Computational and Applied Mathematics} 319:24--37.

        \bibitem[{White(1993)}]{white1993}
        White DJ (1993) {Minimizing a threshold probability in discounted Markov
            decision processes}. \textit{Journal of Mathematical Analysis and
            Applications} 173(2):634--646.


        \bibitem[{Xia et~al.(2023)Xia, Zhang, and Glynn}]{xia2023risk}
            Xia L, Zhang L, Glynn PW (2023) Risk-sensitive Markov decision processes with
            long-run CVaR criterion. \textit{Production and Operations Management}
            32(12):4049--4067.


    \end{thebibliography}
\end{document}